\documentclass{article}

\usepackage[modulo]{lineno}
\usepackage[english]{babel}

\usepackage{amsmath}

\usepackage{amsthm}
\usepackage{amssymb}
\usepackage{amsfonts}
\usepackage{mathrsfs}
\usepackage{diagbox} 
\usepackage{multicol} 
\usepackage{wrapfig} 

\usepackage{thmtools} 
\usepackage{thm-restate} 

\usepackage{caption} 
\usepackage{subcaption}

\usepackage{graphicx}
\usepackage{multirow} 

\usepackage{enumitem} 

\usepackage{comment}

\usepackage{todonotes}

\DeclareMathOperator{\DP}{DP}
\DeclareMathOperator{\FQ}{FQ}

\usepackage{algorithm}
\usepackage[noend]{algpseudocode}

  \usepackage[T1]{fontenc}
  \usepackage[utf8]{inputenc}
  \usepackage{lmodern}

\colorlet{myGreen}{green!50!black}
\colorlet{myLightgreen}{green}
\colorlet{myRed}{red!90!black}
\definecolor{myBlue}{rgb}{0.25, 0.0, 1.0}
\definecolor{myLightBlue}{rgb}{0.39, 0.58, 0.93}
\colorlet{myViolet}{myBlue!55!myRed}
\definecolor{myOrange}{rgb}{1.0, 0.66, 0.07}

\definecolor{magenta}{rgb}{0.94, 0.05, 0.53}
\definecolor{AO}{rgb}{0.0, 0.5, 0.0}
\definecolor{phthaloblue}{rgb}{0.0, 0.06, 0.54}
\definecolor{pistachio}{rgb}{0.58, 0.77, 0.45}
\definecolor{darkgoldenrod}{rgb}{0.72, 0.53, 0.04}

\usepackage{hyperref}

\hypersetup{
  colorlinks=true,
  linkcolor=myBlue,
  citecolor=myGreen,
  urlcolor=myGreen,
  bookmarksopen=true,
  bookmarksnumbered,
  bookmarksopenlevel=2,
  bookmarksdepth=3
  pdftitle = {Fast recognition of some parametric graph families},
  pdfauthor= {Nina Klobas, Matjaž Krnc},
}

\usepackage[export]{adjustbox}

\newif\ifsmall 
\newif\ifbig 
\smallfalse 
\bigtrue 

\usepackage[capitalise,nameinlink, noabbrev]{cleveref}

\usepackage{authblk}

\newtheorem{thm}{Theorem}
\newtheorem{coro}[thm]{Corollary}
\newtheorem{claim}[thm]{Claim}

\theoremstyle{remark}
\newtheorem{remark}[thm]{Remark}
\theoremstyle{definition}
\newtheorem{defin}[thm]{Definition}
\newtheorem{example}[defin]{Example}

\crefname{defin}{Definition}{Definitions}
\crefname{thm}{Theorem}{Theorems}
\crefname{coro}{Corollary}{Corollaries}
\crefname{rem}{Remark}{Remarks}

\author[]{Nina Klobas\footnote{Department of Computer Science, Durham University, UK. Email: \href{mailto:nina.klobas@durham.ac.uk} {\texttt{\normalshape nina.klobas@durham.ac.uk}}} \,
and 
Matjaž Krnc\footnote{The Faculty of Mathematics, Natural Sciences and Information Technologies, University of Primorska, Slovenia. Email: \href{mailto:matjaz.krnc@upr.si}{\texttt{\normalshape matjaz.krnc@upr.si}}}}

\title{Fast recognition of some parametric graph families}

\begin{document}
\maketitle
\begin{abstract}

  We identify all $[1, \lambda, 8]$-cycle regular $I$-graphs and all $[1, \lambda, 8]$-cycle regular double generalized Petersen graphs. As a consequence we describe linear recognition algorithms for these graph families.
  Using structural properties of folded cubes we devise a $o(N \log N)$ recognition algorithm for them.
  We also study their $[1,\lambda,4]$, $[1,\lambda,6]$ and $[2, \lambda, 6]$-cycle regularity and settle the value of parameter $\lambda$.

\end{abstract}
\section{Introduction}
Important graph classes such as
bipartite graphs, 
(weakly) chordal graphs, 
(line) perfect graphs and
(pseudo)forests 
are defined or characterized by their cycle structure.
A particularly strong description of a cyclic structure is the notion of \emph{cycle-regularity}, introduced by Mollard \cite{Mollard:1991}:
\begin{quotation}
\emph{For integers $l, \lambda, m$ a simple graph is $[l, \lambda, m]$-cycle regular if every path on $l+1$ vertices belongs to exactly $\lambda$ different cycles of length $m$.}
\end{quotation}
It is perhaps natural that cycle-regularity mostly appears  in the literature in the context of symmetric graph families such as hypercubes, Cayley graphs or circulants. Indeed, Mollard showed that certain extremal $[3,1,6]$-cycle regular graphs correspond exactly to the graphs induced by the two middle layers of odd-dimensional hypercubes.
Also, for
$[2,1,4]$-cycle regular graphs,  
Mulder \cite{Mulder:1979} showed that their degree is minimized in the case of Hadamard graphs, or in the case of hypercubes.
In relation with other graph families, Fouquet and Hahn \cite{Fouquet/Hahn:2001} described the symmetric aspect of certain cycle-regular classes,  while in \cite{Krnc/Wilson:2017} authors
describe all $[1,\lambda,8]$-cycle regular members of generalized Petersen graphs, and use this result to obtain a linear recognition algorithm for generalized Petersen graphs.
Understanding the structure of subgraphs of hypercubes which avoid all $4$-cycles does not seem to be easy.
Indeed, a question of Erd\H{o}s 
regarding how many edges can such a graph contain remains open after more than $30$ years \cite{MR1254162}.

In this paper we study cycle-regularity and isomorphism structure of three graph families, namely the folded cubes and two natural generalizations of generalized Petersen graphs: $I$-graphs and double generalized Petersen graphs. 
$I$-graphs, introduced in the Foster census \cite{Foster:1988}, are trivalent or cubic graphs with edges connecting vertices of two star polygons.
Double generalized Petersen graphs consist of two identical copies of generalized Petersen graphs, where instead of connecting the vertices inside a star polygon, we connect the vertices from two different star polygons accordingly.
These graphs were first introduced in 2012 by Zhou and Feng \cite{Zhou/Feng:2012}.
Folded cubes were studied by Brouwer in 1983 \cite{Brouwer:1983} and are formed by identifying antipodal vertices of the hypercube graph.

We describe which graphs from the first two graph families are  $[1,\lambda,8]$-cycle regular and then
use those results to devise corresponding
recognition algorithms which run in linear time.
We also
observe the  $[1,\lambda,4]$, $[1,\lambda,6]$ and $[2, \lambda, 6]$-cycle regularity of folded cubes
and using their structure property construct a $o(N \log N)$ time recognition algorithm for them.
We proceed by describing these three graph families, alongside the relevant related work.

\begin{figure}[ht!]
  \centering
  \begin{subfigure}{\ifbig.28 \fi \ifsmall .20 \fi \textwidth}
    \centering

\includegraphics[width=\linewidth]{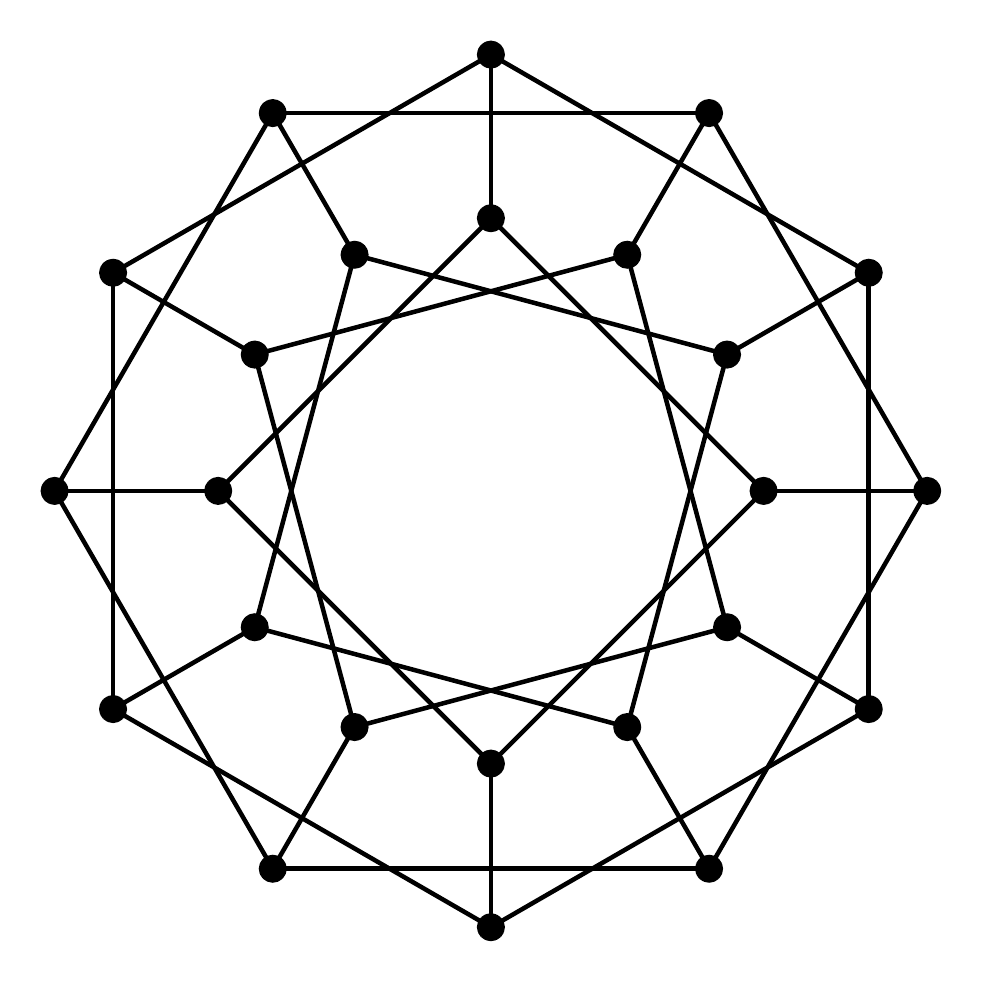}
  \end{subfigure}
  \quad
  \begin{subfigure}{\ifbig.22 \fi \ifsmall .18 \fi \textwidth}
    \centering
    \includegraphics[width=\linewidth]{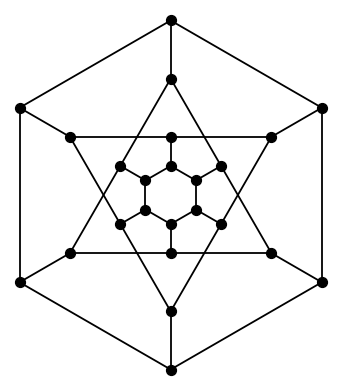}
  \end{subfigure}
  \quad
    \begin{subfigure}{\ifbig.26 \fi \ifsmall .20 \fi \textwidth}
    \centering
    \includegraphics[width=\linewidth]{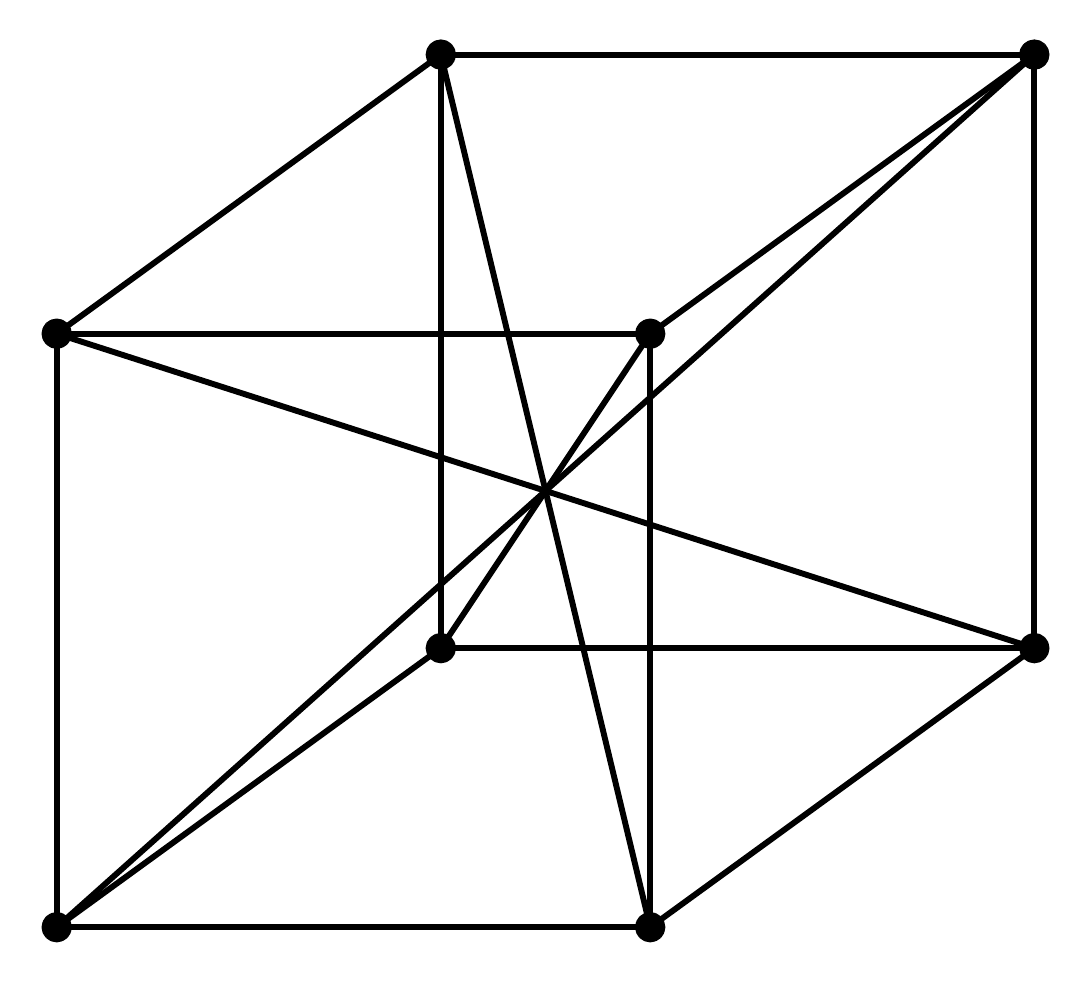}
  \end{subfigure}
  \caption{$I$-graph $I(12,2,3)$, double generalized Petersen graph $\DP(6,1)$, folded cube $\FQ_4$.}
\end{figure}


\subsection{Related work}

Generalized Petersen graphs were introduced in 1950 by Coxeter \cite{Coxeter:1950} and later named by Watkins \cite{Watkins:1969} in 1969.
Some of the results known for this family include identifying generalized Petersen graphs that are Hamiltonian \cite{Alspach:1983,Alspach/Robinson/Rosenfeld:1981,Castagna/Prins:1972},
hypo-Hamiltonian \cite{Bondy:1972}, Cayley \cite{Nedela/Skoviera:1995,Lovrecic:1997}, or partial cubes \cite{Klavzar/Lipovec:2003}, and finding their automorphism group \cite{Frucht/Graver/Watkins:1971} or determining isomorphic members of the family \cite{Petkovsek/Zakrajsek:2009}. Additional aspects of the above-mentioned family are well surveyed in \cite{Chartrand/Hevia/Wilson:1992,Holton/Sheehan:1993}.
For this family, Watkins \cite{Watkins:1969} had studied the structure of $8$-cycles in 1969, while linear recognition was settled by Krnc and Wilson \cite{Krnc/Wilson:2017}.

Two natural generalizations of generalized Petersen graphs are $I$-graphs and double generalized Petersen graphs.
The family of $I$-graphs has been studied extensively with respect to their
automorphism group and isomorphisms \cite{Petkovsek/Zakrajsek:2009,Boben/Pisanski/Zitnik:2005,Horvat/Pisanski/Zitnik:2012}.
In addition to that, Horvat et. al. \cite{Horvat/Pisanski/Zitnik:2012} described their connectedness and girth, while Boben et al. \cite{Boben/Pisanski/Zitnik:2005} studied their bipartiteness, vertex transitivity, and also considered some configurations which arise from bipartite $I$-graphs.
For double generalized Petersen graphs, Kutnar and Petecki \cite{Kutnar/Petecki:2016} characterized their automorphism group and also consider their hamiltonicity, vertex-coloring and edge-coloring.

Some papers, e.g. \cite{Mirafzal:2016,Ming/Meijie:2006}, attribute first studies of folded cubes to El-Amawy and Latifi \cite{Latifi/Amawy:1991}, where they used their structure to develop efficient routing algorithms for broadcasting.
However, they were mentioned already by  Brouwer \cite{Brouwer:1983} and Terwilliger \cite{Terwilliger:1988},
where their  distance regularity structure was studied.
Some known results of these graphs include characterizing
cyclic structure \cite{Ming/Meijie:2006}, edge-fault-tolerant properties \cite{Ming/Meijie/Zhengzhong:2006}, hamiltonian-connectivity, strongly Hamiltonian-laceability \cite{Hsieh/Kuo:2007} and their automorphism group \cite[pg.~265]{Brouwer/Cohen/Neumaier:1989} (see also \cite{Mirafzal:2016}).

\subsection{Our contributions}

We study the cycle structure of families of $I$-graphs, double generalized Petersen graphs and folded cubes and discuss their  $[1,\lambda,8]$ or $[1,\lambda,4]$, $[1,\lambda,6]$, $[2, \lambda, 6]$-cycle regularity. Our results are summarized below.
\begin{thm}\label{thm:lin-Igr}
  An arbitrary $I$-graph  is never $[1,\lambda,8]$-cycle regular, except when isomorphic to $I(n,j,k)$ where $j=1$ and
  \[(n,k)\in \lbrace (3,1), (4,1),(5,2),(8,3), (10,2),(10,3), (12,5),(13,5),(24,5), (26,5) \rbrace.\]
\end{thm}
\begin{thm}\label{thm:lin-DP}
  A double generalized Petersen graph is never $[1,\lambda,8]$-cycle regular, except when isomorphic to $\DP(n,k)$ where $(n,k)\in \lbrace (5,2), (10,2) \rbrace$.
\end{thm}
We also observe the following isomorphism property for double generalized Petersen graphs.
\begin{restatable}{thm}{isomorphDP}
  \label{thm:isomorphDP}
  Let $n,k$ be positive integers, where $n$ is even and $k < n/2$. Then the graph $\DP(n,k)$ is isomorphic to $\DP(n, n/2 - k)$.
\end{restatable}
For all three families we devise efficient recognition algorithms, which are robust in the sense of Spinard \cite{Spinrad/Robust:2003}.
\begin{coro}\label{cor:lin}
  $I$-graphs and double generalized Petersen graphs can be recognized in linear time. Folded cubes can be recognized in $o(N \log N)$ time, where $N=|V(G)|+|E(G)|$.
\end{coro}

We also study the $4$ and $6$-cycle structure of folded cubes and get the following results.
\begin{restatable}{thm}{thmFQncycles}
\label{thm:FQ4-cycles}
Folded cubes $\FQ_1$ and $\FQ_2$  are $[1,0,4]$-cycle regular. Folded cube $\FQ_4$ is $[1,9,4]$-cycle regular. Any other folded cube $\FQ_n$ is $[1,n-1,4]$-cycle regular.
\end{restatable}
\begin{restatable}{thm}{cyclesFQ}
\label{thm:FQ6-cycles}
Folded cubes $\FQ_1, \FQ_2$ and $\FQ_3$ are $[1,0,6]$-cycle regular. Folded cubes $\FQ_4$ and $\FQ_6$ are $[1,36,6]$ and $[1,200,6]$-cycle regular. Any other folded cube $\FQ_n$ is $[1,4(n-2)(n-1),6]$-cycle regular.
\end{restatable}
\begin{restatable}{thm}{cyclesFQpathOfLengthTwo}
\label{thm:FQ2-6-cycles}
Folded cube $\FQ_4$ is not $[2,\lambda,6]$-cycle regular. Folded cubes $\FQ_1, \FQ_2$ and $\FQ_3$ are $[2,0,6]$-cycle regular. Folded cube $\FQ_6$ is $[2,2,6]$-cycle regular. Any other folded cube $\FQ_n$ is $[2,4(n-2),6]$-cycle regular.
\end{restatable}
\begin{restatable}{conjecture}{EightCyclesFQ}
\label{con:FQ1-8-cycles}
Folded cubes $\FQ_1, \FQ_2$ and $\FQ_3$ are $[1,0,8]$-cycle regular.
Folded cubes $\FQ_4, \FQ_6$ and $\FQ_8$ are $[1,36,8], [1,3580,8]$ and $[1,10794,8]$-cycle regular. Any other folded cube $\FQ_n$ is $[1,27n^3 - 133n^2 + 210 n - 104,8]$-cycle regular.
\end{restatable}

Unless specified otherwise, all graphs in this paper are finite, simple, undirected and connected.
For a given graph $G$ we use a standard notation for a set of vertices $V(G)$ and a set of edges $E(G)$.
A $k$-cycle $C$ in $G$ is a cycle on vertices $v_1, v_2, \dots , v_k$ from $V(G)$, using edges $e_1, e_2, \dots, e_k$ from $E(G)$; we will write it in two ways: as $( v_1, \dots , v_k )$, when we are discussing $I$-graphs and double generalized Petersen graphs, or $(e_1, \dots, e_k)$, with folded cubes.
We use the notation $G(n,k)$ for a generalized Petersen graph. For integers $a$ and $b$ we denote with $\gcd(a,b)$ the greatest common divisor of $a$ and $b$ respectively.

A definition of a \emph{robust algorithm} was introduced by Spinrad in \cite{Spinrad/Robust:2003} and states that an algorithm is robust if it takes as an input an arbitrary graph and it either determines the parameters of the graph belonging to our family or concludes that an input does not belong to the desired graph family.

This paper is structured as follows.
In \cref{section-Igraphs} we observe basic properties of $I$-graphs and study their $8$-cyclic structure, more precisely we characterize all $[q, \lambda, 8]$-cycle regular members of the family. In \cref{section-DPgraphs} we do the same for double generalized Petersen graphs. In \cref{section-FodledCubes} we focus our study on folded cubes and inspect their $4, 6$ and $8$-cyclic structure. At the end in \cref{section-RecognitionAlgorithms} we provide recognition algorithms for all of the previously mentioned families.

\section{\texorpdfstring{\boldmath$I$}{I}-graphs\label{section-Igraphs}}
For integers $n,j,k$, where $n \geq 3$ and $n \geq j,k \geq 1$, an {$I$-graph} $I(n,j,k)$, is a graph with the vertex set $\lbrace u_0, u_1, \dots, u_{n-1}, w_0, w_1, \dots, w_{n-1} \}$ and the edge set consisting of outer edges $u_i u_{i+j}$, inner edges $w_i w_{i + k}$ and spoke edges $u_i w_i$, where the subscripts are taken modulo $n$.
Without loss of generality we always assume that $j,k < n / 2$. Since $I(n,j,k)$ is isomorphic to $I(n,k,j)$, we restrict ourselves to cases when $j \leq k$.
Generalized Petersen graphs form a subclass of $I$-graphs, where parameter $j$ has value $1$.

It is well known \cite{Boben/Pisanski/Zitnik:2005} that an $I$-graph $I(n,j,k)$ is disconnected whenever $d = \gcd(n,j,k)>1$. In this case it consists of $d$ copies of $I(n/d,j/d,k/d)$. Therefore, throughout the paper we consider only graphs $I(n,j,k)$ where $\gcd(n,j,k)=1$.
We also know \cite{Horvat/Pisanski/Zitnik:2012} that two $I$-graphs $I(n,j,k)$ and $I(n, j', k')$ are isomorphic if and only if there exists an integer $a$, which is relatively prime to $n$, for which either $\{ j', k' \} = \{ a j \pmod n, ak \pmod n \}$ or $\{j', k' \} = \{aj \pmod n, -ak \pmod n \}$.
Throughout the paper, whenever we discuss $I$-graphs with certain parameters, we consider only the lexicographically smallest possible parameters by which the graph is uniquely determined.

\subsection{Equivalent \texorpdfstring{\ifbig \boldmath \fi $8$}{8}-cycles}
Every $I$-graph admits a rotation $\rho$ defined in a natural way: $\rho(u_i) = u_{i +1}$, $\rho(w_i) = w_{i+1}$.
Clearly, applying $n$ times the rotation $\rho$ yields an identity automorphism.
When acting on $I$-graphs with $\rho$  we get $3$ edge orbits: orbit of outer edges $E_J$, orbit of spoke edges $E_S$ and orbit of inner edges $E_I$.
Edges from the same orbit $E_J, E_S,$ or $E_I$ have the same octagon value, which we denote by $\sigma_J, \sigma_S$ and  $\sigma_I$, respectively. Therefore the \emph{octagon value of an $I$-graph} is said to be a triple $(\sigma_J, \sigma_S, \sigma_I)$.

In the $I$-graph $I(n,j,k)$, outer edges $E_J$ induce $\gcd(n,j)$ cycles of length $n / \gcd(n,j)$, inner edges $E_I$ induce $\gcd(n,k)$ cycles of length $n / \gcd(n,k)$ and spoke edges $E_S$ induce a perfect matching.

We say that two $8$-cycles of an $I$-graph are \emph{equivalent} if we can map one into the other using only rotation $\rho$.

Let $G\simeq I(n,j,k)$ be an arbitrary $I$-graph and let $C$ be one of its $8$-cycles. With $\gamma(C)$ we denote the number of equivalent $8$-cycles to $C$ in $G$. Each $8$-cycle contributes to the octagon value of a double generalized Petersen graph. We denote the contributed amount with $\tau(C)$, defined as the triple $(\delta_j, \delta_s, \delta_i)$, where we calculate $\delta_j, \delta_s, \delta_i$ by counting the number of outer, spoke and inner edges of a cycle and multiply these numbers with $\gamma / n$.
If graph $G$ admits $m$ non-equivalent $8$-cycles, one may calculate its octagon value $(\sigma_J, \sigma_S, \sigma_I)$ as
\begin{equation*}\label{sigma values of graph eq}
  (\sigma_J, \sigma_S, \sigma_I) = \sum _{i=1}^m  \tau(C_i).
\end{equation*}

The following claim serves also as an example of the above-mentioned definitions. Keep in mind that we are considering only $I$-graphs with $\gcd(n,j,k)=1$.
\begin{claim}
  For $I(n,j,k)$ where $n > 3$ and integers $k,j < n / 2$ there always exists an $8$-cycle.
\end{claim}
Indeed,
if $k \neq j$ it is of the form
$$C^* =(w_0, w_{\pm k}, u_{\pm k}, u_{\pm k \pm j}, w_{\pm k \pm j}, w_{\pm j}, u_{\pm j},u_0). $$
If $k = j$ it is of the form
$$C_7 =(u_0, u_k, u_{2k}, u_{3k}, w_{3k}, w_{2k}, w_{k},w_{0}).$$

\subsection{Characterization of non-equivalent \texorpdfstring{\ifbig \boldmath \fi $8$}{8}-cycles}
The purpose of this section is to provide a complete list of all possible non-equivalent $8$-cycles that can appear in an $I$-graph.
For each such $8$-cycle we additionally determine the contribution towards the $(\sigma_j,\sigma_i,\sigma_s)$ values. Those results are summarized in \cref{tab:I8conditions,tab:I8cycles} and \cref{fig:cyclesInIGraphs}.

\begin{figure}[ht!]
  \ifsmall \centering \fi
  \begin{subfigure}{\ifbig .245 \fi \ifsmall .18 \fi \textwidth}
    \centering
    \includegraphics[width=\linewidth]{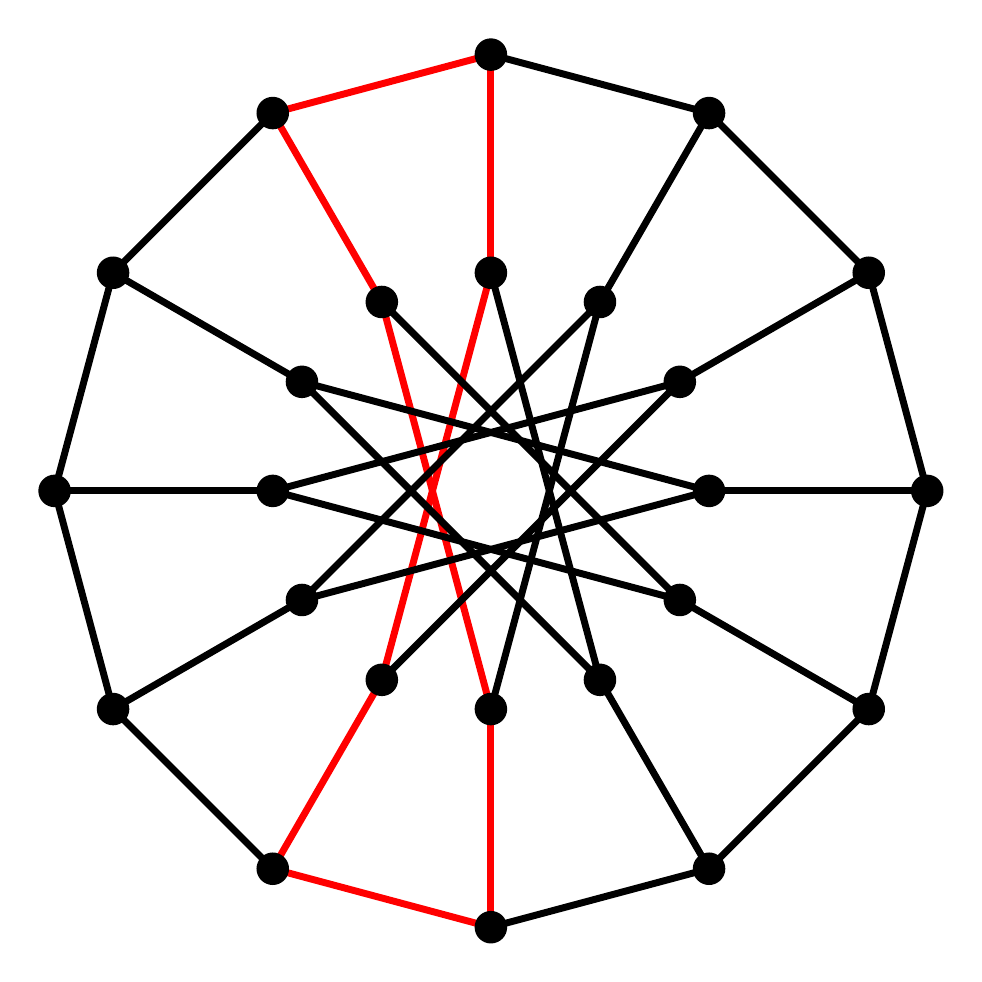}
    \caption{Cycle $C^*$.}
    \label{fig:IC*}
  \end{subfigure}
  \begin{subfigure}{\ifbig .245 \fi \ifsmall .18 \fi \textwidth}
    \centering
    \includegraphics[width=\linewidth]{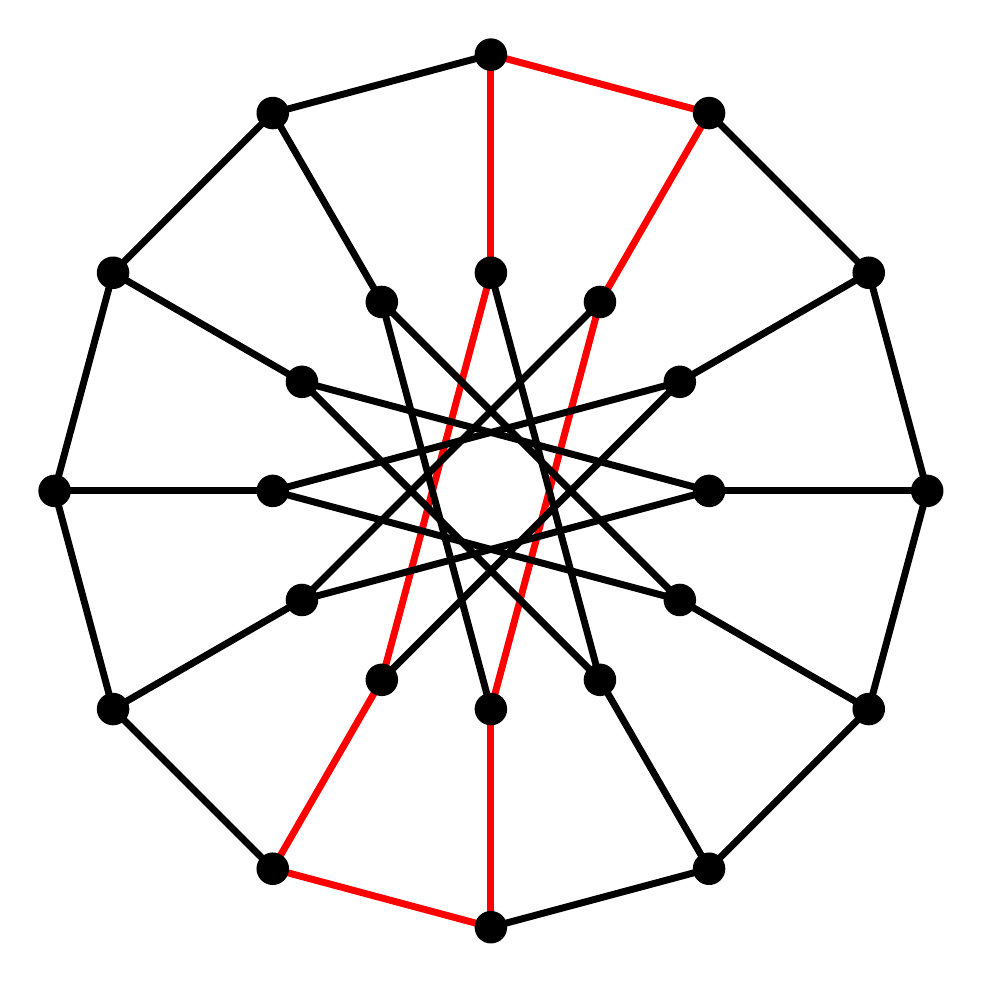}
    \caption{Cycle $C_0$.}
    \label{fig:IC0}
  \end{subfigure}
  \begin{subfigure}{\ifbig .245 \fi \ifsmall .18 \fi \textwidth}
    \centering
    \includegraphics[width=\linewidth]{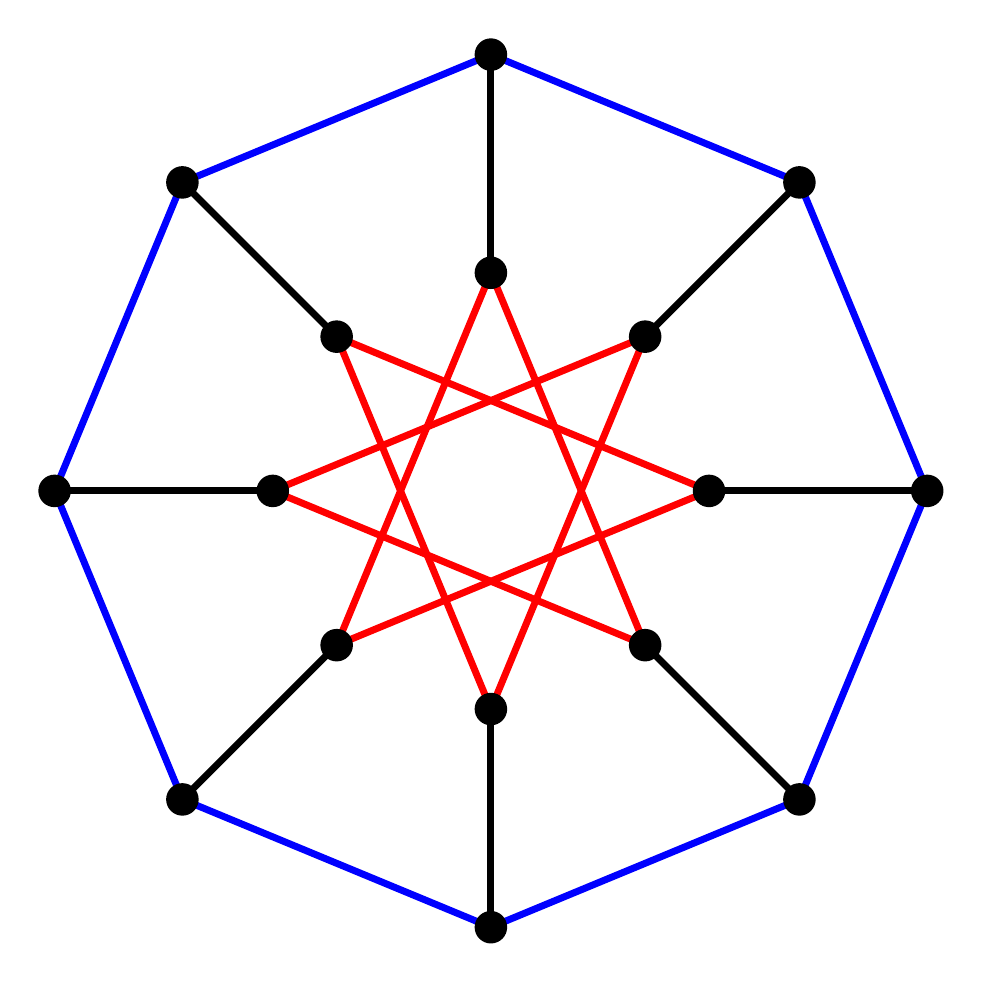}
    \caption{Cycle $C_1$ in blue and $C_2$ in red.}
    \label{fig:IC1C2}
  \end{subfigure}
  \begin{subfigure}{\ifbig .245 \fi \ifsmall .18 \fi \textwidth}
    \centering
    \includegraphics[width=\linewidth]{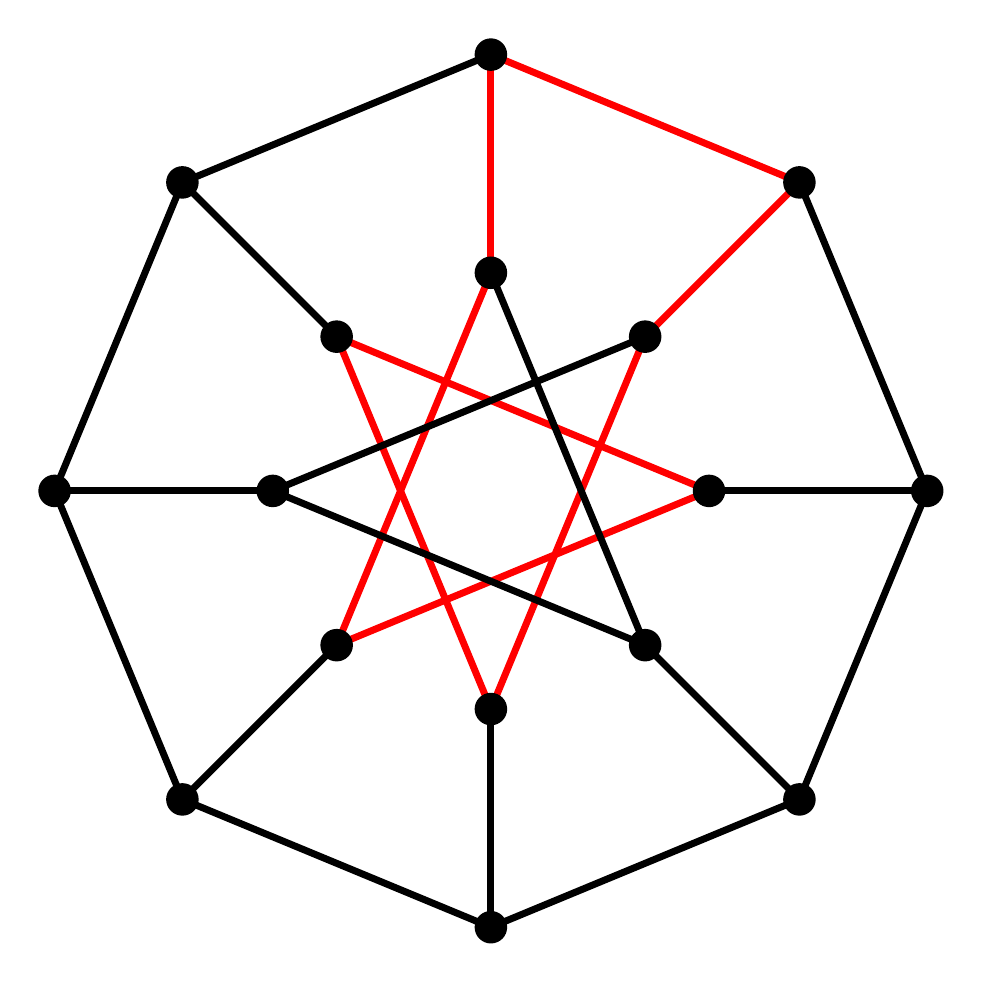}
    \caption{Cycle $C_3$.}
    \label{fig:IC3}
  \end{subfigure}
  \\
  \begin{subfigure}{\ifbig .245 \fi \ifsmall .18 \fi \textwidth}
    \centering
    \includegraphics[width=\linewidth]{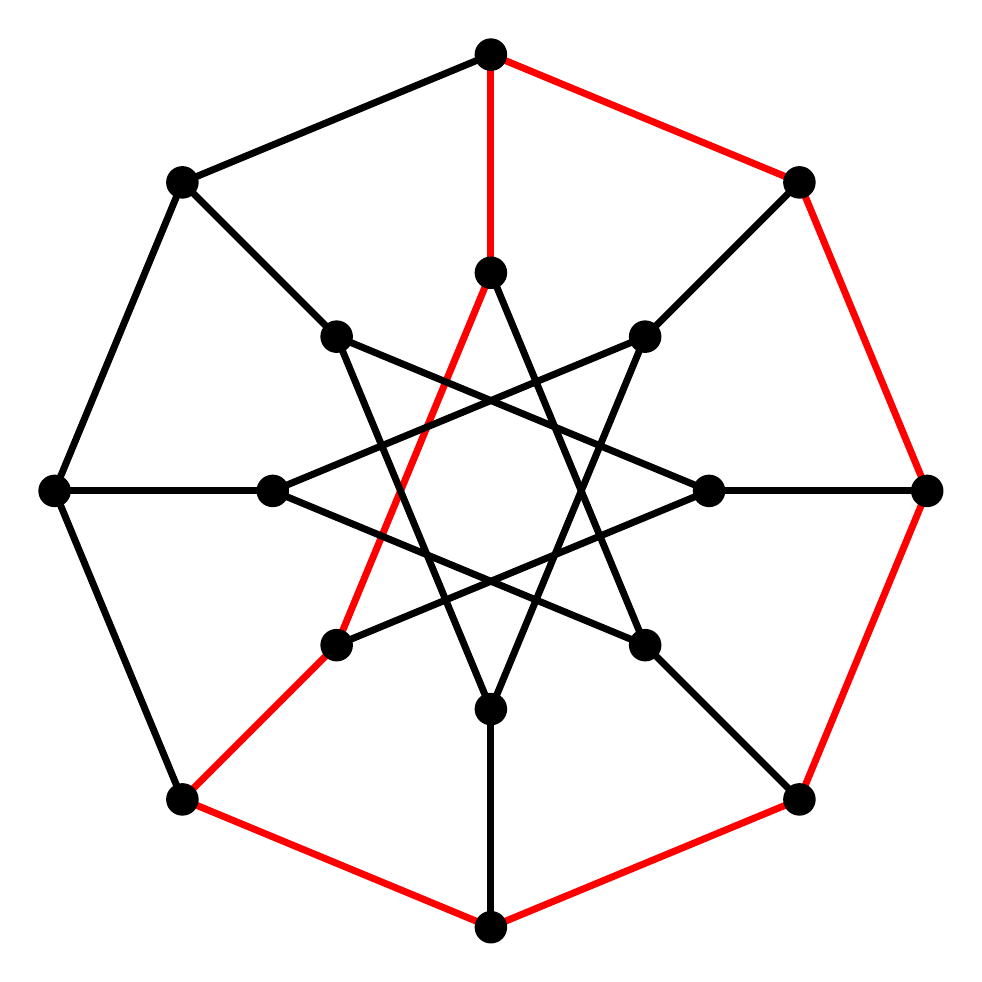}
    \caption{Cycle $C_4$.}
    \label{fig:IC4}
  \end{subfigure}
  \begin{subfigure}{\ifbig .245 \fi \ifsmall .18 \fi \textwidth}
    \centering
    \includegraphics[width=\linewidth]{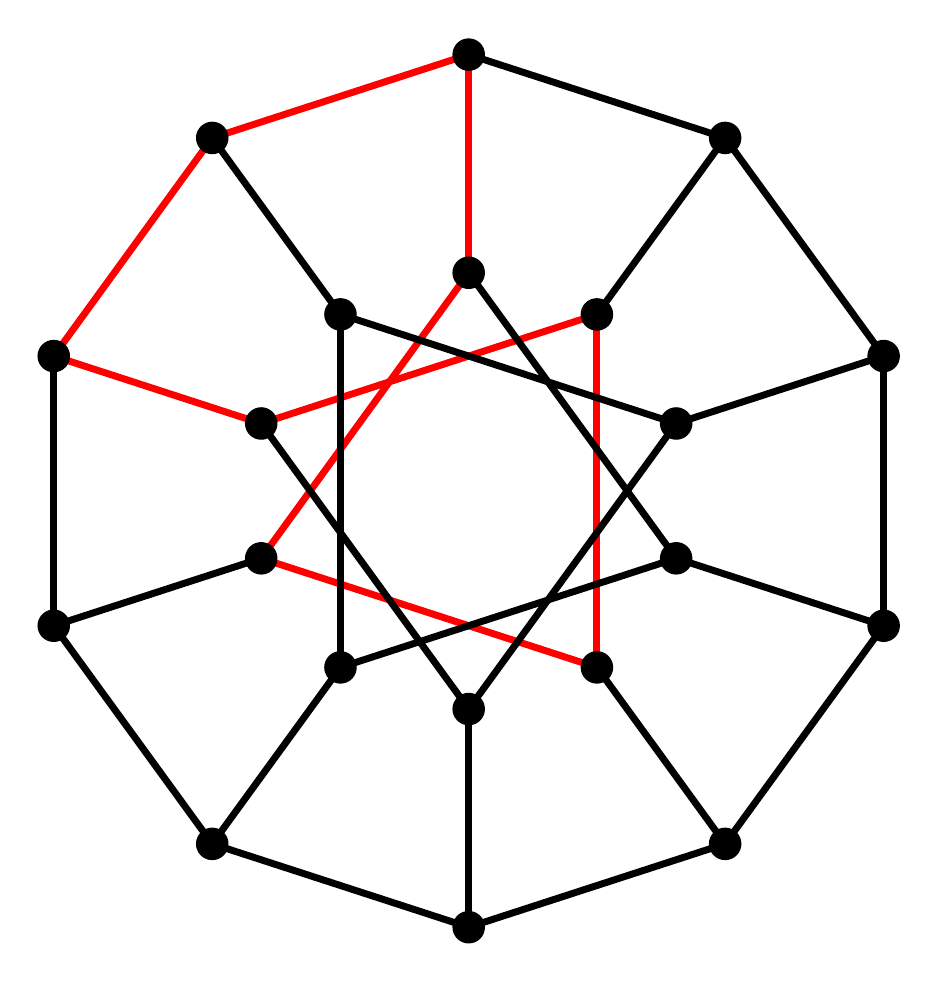}
    \caption{Cycle $C_5$.}
    \label{fig:IC5}
  \end{subfigure}
  \begin{subfigure}{\ifbig .245 \fi \ifsmall .18 \fi \textwidth}
    \centering
    \includegraphics[width=\linewidth]{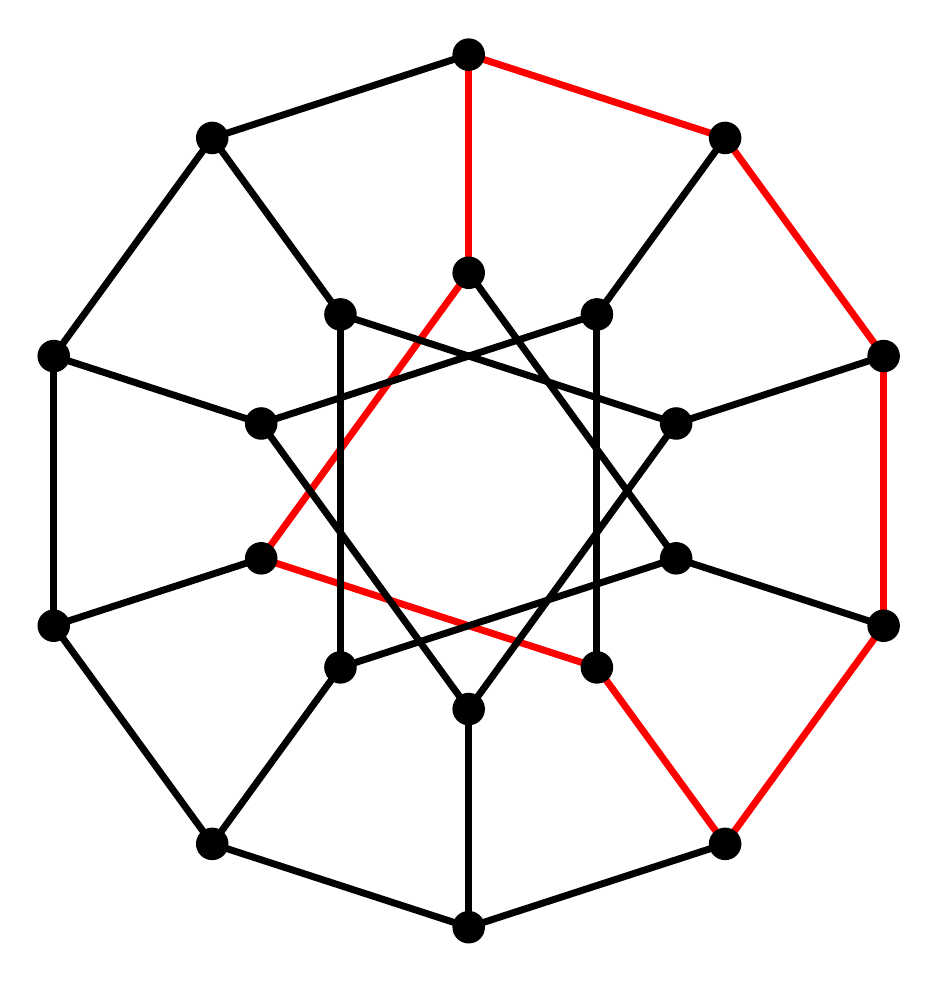}
    \caption{Cycle $C_6$.}
    \label{fig:IC6}
  \end{subfigure}
  \begin{subfigure}{\ifbig .245 \fi \ifsmall .18 \fi \textwidth}
    \centering
    \includegraphics[width=\linewidth]{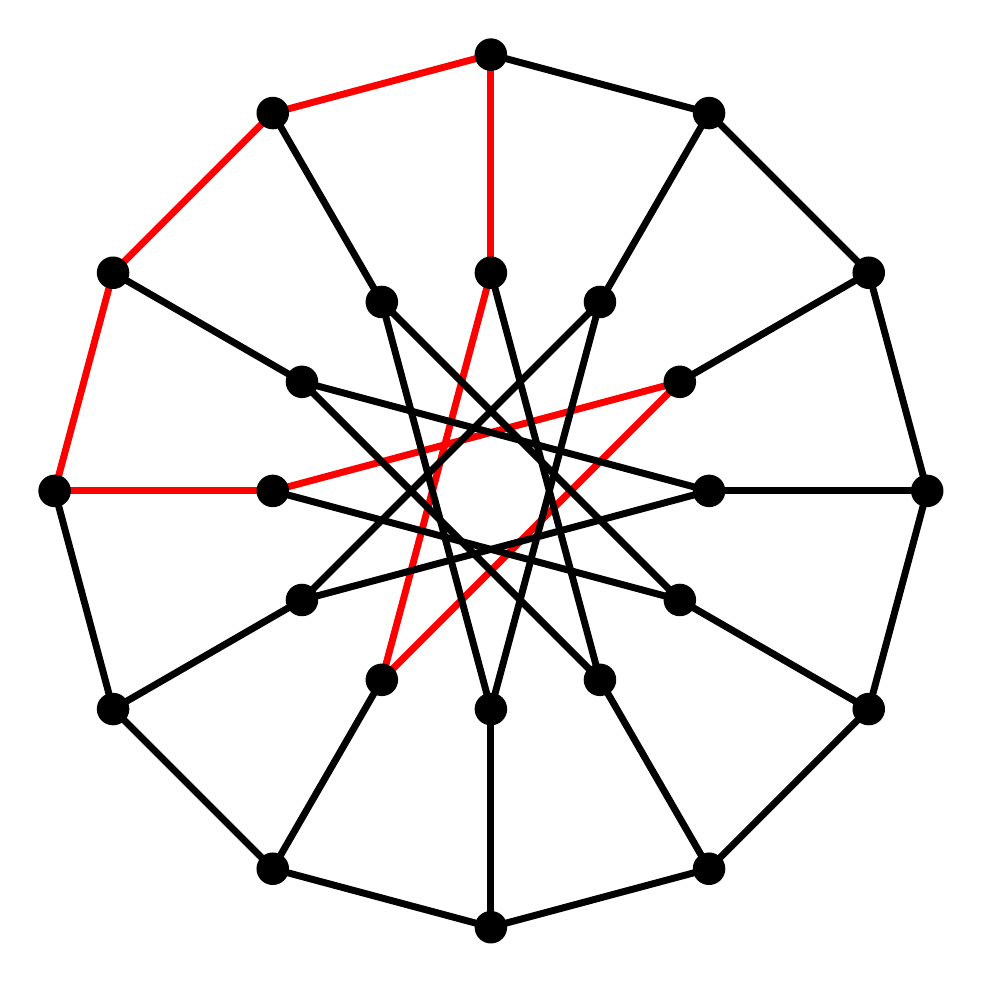}
    \caption{Cycle $C_7$.}
    \label{fig:IC7}
  \end{subfigure}
  \caption{Examples of non-equivalent $8$-cycles in $I$-graphs.}
  \label{fig:cyclesInIGraphs}
\end{figure}

The related case analysis is discussed in the subsections below, and organized as follows.
We first distinguish $8$-cycles by the number of spoke edges they admit. Indeed, it is easy to see that an arbitrary $8$-cycle can have either $4, 0$ or $2$ spoke edges. The first two cases correspond to \cref{sec:4spoke,sec:0spoke}, respectively.
For the last case we further distinguish cases by the number of outer and inner edges within a given $8$-cycle. Those cases are discussed in \cref{sec:outer1,sec:outer2,sec:outer3}.

\subsubsection{\texorpdfstring{\ifbig \boldmath \fi $8$}{8}-cycles with \texorpdfstring{\ifbig \boldmath \fi $4$}{4} spoke edges\label{sec:4spoke}}
In addition to $4$ spoke edges the $8$-cycle must have also two inner and two outer edges. When using the spoke edge there are two options for choosing an inner (outer) edge. After a thorough analysis of all possibilities it is easy to see that there can be just two such $8$-cycles, $C^*$ (see \cref{fig:IC*}), which exists whenever $j \neq k$, and $C_0$, which is of the following form:
$$C_0 =(w_0, w_{\pm k}, u_{\pm k}, u_{\pm k \pm j}, w_{\pm k \pm j}, w_{\pm 2k \pm j}, u_{\pm 2k \pm j},u_{\pm 2k \pm 2j}),$$
see \cref{fig:IC0}.
Cycle $C_0$ exists whenever $2k + 2j = n$.
One can verify easily, that $n$ applications of the rotation $\rho$ to $C^*$ and $n/2$ applications of the rotation $\rho$ to cycle $C_0$ maps the cycle back to itself. Therefore there are $n$ equivalent cycles to $C^*$ and $n/2$ equivalent cycles to $C_0$ in an $I$-graph $I(n,j,k)$ and they contribute $(2,4,2)$ and $(1,2,1)$, respectively, to the graph octagon value.

\begin{table}[t]
    \centering
    \resizebox{\ifsmall 0.70\fi \columnwidth}{!}{
    \begin{tabular}{c|c|c}
    \textbf{Label}          & \textbf{A representative of an \boldmath$8$-cycle}                                                                 & \textbf{Existence conditions} \\
    \hline
    $C^*$                   & $(w_0, w_{\pm k}, u_{\pm k}, u_{\pm k \pm j}, w_{\pm k \pm j}, w_{\pm j}, u_{\pm j},u_0)$                & $k \neq j$ and $n > 4$        \\
    \hline
    $C_0$                   & $ (w_0, w_{\pm k}, u_{\pm k}, u_{\pm k \pm j}, w_{\pm k \pm j}, w_{\pm 2k \pm j}, u_{\pm 2k \pm j}, u_{\pm 2k \pm 2j})$ & $2k + 2j = n$                 \\
    \hline
    $C_1$                   & $(u_0, u_j, u_{2j}, u_{3j}, u_{4j}, u_{5j}, u_{6j}, u_{7j})$                                              & $ 8j = n$ or  $3n$            \\
    \hline
    $C_2$                   & $(w_0, w_k, w_{2k}, w_{3k}, w_{4k}, w_{5k}, w_{6k}, w_{7k})$                                              & $ 8k = n$ or $3n$             \\
    \hline
    \multirow{2}{*}{$C_3$}  & $(w_0, w_k, w_{2k},  w_{3k}, w_{4k}, w_{5k}, u_{5k}, u_{5k+j})$                                           & $5k + j = n$ or $2n$          \\
    \cline{2-3}
                            & $(w_0, w_k, w_{2k},  w_{3k}, w_{4k}, w_{5k}, u_{5k}, u_{5k-j})$                                           & $5k - j = n$ or $2n$          \\
    \hline
    \multirow{2}{*}{$C_4$}  & $(u_0, u_j, u_{2j}, u_{3j}, u_{4j}, u_{5j}, w_{5j}, w_{5j + k})$                                          & $k + 5j = n$ or $ 2n$         \\
    \cline{2-3}
                            & $(u_0, u_j, u_{2j}, u_{3j}, u_{4j}, u_{5j}, w_{5j}, w_{5j - k})$                                          & $5j - k = 2n$ or $n$ or $0$   \\
    \hline
    \multirow{2}{*}{$C_5$}  & $(w_0, w_k, w_{2k},  w_{3k}, w_{4k}, u_{4k}, u_{4k + j}, u_{4k + 2j})$                                    & $4k + 2j = n$ or $2k + j = n$ \\
    \cline{2-3}
                            & $(w_0, w_k, w_{2k},  w_{3k}, w_{4k}, u_{4k}, u_{4k - j}, u_{4k - 2j})$                                    & $4k - 2j = n$                 \\
    \hline
    \multirow{ 2}{*}{$C_6$} & $(u_0, u_j, u_{2j},  u_{3j}, u_{4j}, w_{4j}, w_{4j + k}, w_{4j + 2k})$                                    & $2k + 4j = n$ or $k + 2j = n$ \\
    \cline{2-3}
                            & $(u_0, u_j, u_{2j},  u_{3j}, u_{4j}, w_{4j}, w_{4j - k}, w_{4j - 2k})$                                    & $4j - 2k = n$ or $0$          \\
    \hline
    \multirow{2}{*}{$C_7$}  & $(w_0, w_k, w_{2k},  w_{3k}, u_{3k}, u_{3k + j}, u_{3k + 2j}, u_{3k + 3j})$                               & $3k + 3j = n$ or $2n$         \\
    \cline{2-3}
                            & $(w_0, w_ {k}, w_{2k},  w_{ 3k}, u_{3k}, u_{3k - j}, u_{3k - 2j}, u_{3k - 3j})$                           & $3k - 3j = n$ or $0$          \\
  \end{tabular}}
  \caption{All non-equivalent $8$-cycles of $I$-graphs.}
  \label{tab:I8conditions}
\end{table}

\subsubsection{\texorpdfstring{\ifbig \boldmath \fi$8$}{8}-cycles without spoke edges\label{sec:0spoke}} In this case we have $2$ options for such a cycle, either it lies completely in the outer rim or in the inner rim (see \cref{fig:IC1C2}). By the definition of $I$-graphs, cycles in the outer rim (inner rim) are of length $n / \gcd(n,j)$ ($n / \gcd(n,k)$), therefore this cycle exists when $8j \equiv 0 \pmod n$ ($8k \equiv 0 \pmod n$).
Using this equation and the fact that $j < n/2$ ($k < n/2$) we get the following conditions: either $8j = n$ ($8k = n $) or $8j = 3n$ ($8j = 3n$).
These cycles are of the following form:
\begin{align*}
  C_1 & =(u_0, u_j, u_{2j}, u_{3j}, u_{4j}, u_{5j}, u_{6j}, u_{7j}), \\
  C_2 & =(w_0, w_k, w_{2k}, w_{3k}, w_{4k}, w_{5k}, w_{6k}, w_{7k}),
\end{align*}
see \cref{fig:IC1C2}.
Each edge appears only in one cycle of this form, therefore there are $n/8$ equivalent such cycles. These cycles contribute $(1,0,0)$, $(0,0,1)$ respectively, to the octagon value of an $I$-graph.

\subsubsection{\texorpdfstring{\ifbig \boldmath \fi$8$}{8}-cycles with \texorpdfstring{\ifbig \boldmath \fi$1$}{1} edge in the outer or inner rim \label{sec:outer1}}
In this case cycles are of form
\begin{align*}
  C_3 & =(w_0, w_k, w_{2k}, w_{3k}, w_{4k}, w_{5k}, u_{5k}, u_{5k \pm j}) \quad \text{or} \\
  C_4 & =(u_0, u_j, u_{2j}, u_{3j}, u_{4j}, u_{5j}, w_{5j}, w_{5j \pm k}),
\end{align*}
see also \cref{fig:IC3,fig:IC4}.
This happens when $5k + j \equiv 0 \pmod n$ or $5k - j \equiv 0 \pmod n$ in the first case, $k + 5j \equiv 0 \pmod n$ in the second case.
Since there is only one edge in the outer or inner rim, we can see that the number of equivalent cycles equals $n$. Therefore they contribute to the octagon value of an $I$-graph $(1,2,5)$ and $(5,2,1)$ respectively.

\subsubsection{\texorpdfstring{\ifbig \boldmath \fi$8$}{8}-cycles with \texorpdfstring{\ifbig \boldmath \fi$2$}{2} edges in the outer or inner rim \label{sec:outer2}}
In this case cycles are of form:
\begin{align*}
  C_5 & =(w_0, w_k, w_{2k},  w_{3k}, w_{4k}, u_{4k}, u_{4k \pm j}, u_{4k \pm 2j}) \quad \text{or} \\
  C_6 & =(u_0, u_j, u_{2j},  u_{3j}, u_{4j}, w_{4j}, w_{4j \pm k}, w_{4j \pm 2k}),
\end{align*}
see also \cref{fig:IC5,fig:IC6}. This happens when $4k + 2j \equiv 0 \pmod n $ or $4k - 2j \equiv 0 \pmod n $ for the first case and $2k + 4j \equiv 0 \pmod n$ or $2k - 4j \equiv 0 \pmod n$ for the second one.
To determine their contribution to the octagon value of an $I$-graph, it is necessary to see that there are $n$ such equivalent $8$-cycles.
In the case of $C_5$, we prove this by taking vertices on the outer rim $u_{4k}, u_{4k \pm j}, u_{4k \pm 2j}$. Two of those vertices ($u_{4k}, u_{4k \pm 2j}$) are incident to a spoke edge and an outer edge, where $u_{4k \pm j}$ is incident only to outer edges. So if we want $\gamma(C_5)$ to be less than $n$, there would have to exist $i \in \mathbb{N}$ such that $\rho^i(u_{4k})=u_{4k \pm 2j}$ and $\rho^i(u_{4k \pm j})=u_{4k \pm j}$, which is not possible.
We can argue similarly for $C_6$. Therefore $C_5$ and $C_6$ occur $n$ times and contribute to the octagon value of an $I$-graph $(2,2,4), (4,2,2)$ respectively.

\subsubsection{\texorpdfstring{\ifbig \boldmath \fi$8$}{8}-cycles with \texorpdfstring{\ifbig \boldmath \fi $3$}{3} edges in the outer and inner rim \label{sec:outer3}}
In this case cycles are of form:
$$C_7 =(u_0, u_j, u_{2j}, u_{3j}, w_{3j}, w_{3j \pm k}, w_{3j \pm 2k},w_{3j \pm 3k}),$$
see also \cref{fig:IC7}. This happens when $3k + 3j \equiv 0 \pmod n $ or $3k - 3j \equiv 0 \pmod n$.
Similarly as before, we can show that $C_7$ occurs $n$ times.
If we observe vertices of the outer rim $u_0, u_j, u_{2j}, u_{3j}$, we see that the first and last vertex $u_0,u_{3j}$ are incident to one outer and one spoke edge, where $u_j, u_{2j}$ are incident to two outer edges. In this case the rotation $\rho^i$ would need to map $u_0$ to $u_{3j}$ and $u_j$ to $u_{2j}$, which is impossible. Therefore $C_7$ occurs $n$ times and contributes $(3,2,3)$ to the octagon value of an $I$-graph.

All possible $8$-cycles with their existence conditions are summed up in \cref{tab:I8conditions}, whereas their contributions to graph octagon value and number of equivalent cycles are summed up in \cref{tab:I8cycles}.

\begin{table}[ht!]
  \centering
     \resizebox{\ifsmall 0.75 \fi \columnwidth}{!}{%
  \begin{tabular}{c|c|c|c|c|c|c|c|c|c}
    \textbf{Label}       & $C^*$     & $C_0$     & $C_1$     & $C_2$     & $C_3$     & $C_4$     & $C_5$     & $C_6$     & $C_7$     \\
    \hline
    \boldmath$\tau(C)$   & $(2,4,2)$ & $(1,2,1)$ & $(0,0,1)$ & $(1,0,0)$ & $(5,2,1)$ & $(1,2,5)$ & $(4,2,2)$ & $(2,2,4)$ & $(3,2,3)$ \\

    \boldmath$\gamma(C)$ & $n$       & $n / 2$   & $n / 8$   & $n / 8$   & $n$       & $n$       & $n$       & $n$       & $n$
  \end{tabular}}
  \caption{Contribution of $8$-cycles to octagon value and their occurrences in $I$-graphs.}
  \label{tab:I8cycles}
\end{table}

\subsection{Obtaining constant octagon value}
As we observed, every $8$-cycle of an $I$-graph contributes to the octagon value of each edge partition. It will turn out that if we can identify at least one edge partition of a graph, we can easily determine its parameters. Therefore, we want to find graphs with constant octagon value, these are graphs for which all edges touch  the same number of $8$-cycles. They are called $[1,\lambda,8]$-cycle regular.
To compute their parameters we use calculations from previous sections.

We consider all possible collections of $8$-cycles and determine octagon values of $I$-graphs admitting those $8$-cycles. Since $I$-graphs are defined with $3$ parameters and all $8$-cycles give constraints for these parameters, it is enough to consider collections of at most $4$ cycles, to uniquely determine all $[1,\lambda,8]$-cycle regular graphs.
Some $8$-cycles have more than one existence condition, so it can happen that the graph admits multiple $8$-cycles of the same type. This happens with cycles $C_5, C_6$ where such a graph is isomorphic to $G(8,2)$ and with $C_7$ where a graph is isomorphic to $G(6,1)$. In both cases the octagon value of an $I$-graph is not constant. With other cycles existence conditions exclude each other, therefore there is at most one occurrence of each non-equivalent $8$-cycle in an arbitrary $I$-graph.

\subsubsection{Coexistence  of \texorpdfstring{\ifbig \boldmath \fi$C_0$ }{C0} and \texorpdfstring{\ifbig \boldmath \fi$C_7$}{C7}.}
If an $I$-graph $G$ admits only two $8$-cycles, one of form $C_0$ and the other of form $C_7$, then it has a constant octagon value $(4,4,4)$.\\
In this case $G$ does not admit any other cycle, neither $C^*$, which holds whenever $k=j$. Combining this with conditions for $C_0$ and $C_7$ we can see that there is exactly one such graph, namely the Cubical graph $G(4,1)$.

\subsubsection{Coexistence  of \texorpdfstring{\ifbig \boldmath \fi$C_0$}{C0}, \texorpdfstring{\ifbig \boldmath \fi$C_1$}{C1} and \texorpdfstring{\ifbig \boldmath \fi$C_2$}{C2}.}
An $I$-graph $G$ which admits exactly these three $8$-cycles has the octagon value $(2,2,2)$.
After observing the existence conditions from \cref{tab:I8cycles} for $C_0, C_1, C_2$ and taking into consideration that $C^*$ does not exist, we conclude that such an $I$-graph does not exist.

\subsubsection{Coexistence  of \texorpdfstring{\ifbig \boldmath \fi$C_3$}{C3}, \texorpdfstring{\ifbig\boldmath \fi$C_4$}{C4} and \texorpdfstring{\ifbig \boldmath \fi$C^*$}{C*}.}
In this case the octagon value of an $I$-graph equals $(8,8,8)$. Possible $I$-graphs containing these three $8$-cycles are calculated in \cref{tab:C3/C4}, using existence conditions of $C_3, C_4$ and considering that an $I$-graph admits $C^*$ whenever $j \neq k$.

\begin{table}[h!]
  \centering
  \resizebox{\ifsmall 0.55 \fi\columnwidth}{!}{
  \begin{tabular}{c||c|c|c|c|c}
    \backslashbox{$C_3$}{$C_4$} & $k + 5j = n$ & $k + 5j = 2n$ & $5j- k = 2n$ & $5j- k = n$  & $5j- k = 0$ \\
    \hline \hline
    $5k + j = n$                & not exist    & $G(8,3))$    & $G(26,5)$ & $G(13,5)$  &$G(26,5)$ \\
    \hline
    $5k + j = 2n$               & $G(8,2)$   & not exist     & $G(13,5)$  & $G(26,5)$  & $G(13,5)$ \\
    \hline
    $5k - j = 2n$               & $G(26,5)$ & $G(13,5)$   & not exist    & $G(24,5)$ & $G(12,5)$ \\
    \hline
    $5k - j = n$                & $G(13,5)$  & $G(26,5)$   & $G(24,5)$ & not exist    & $G(24,5)$ \\
  \end{tabular}}
  \caption{$I$-graphs containing $C_3, C_4$ and $C^*$.}
  \label{tab:C3/C4}
\end{table}

\subsubsection{Coexistence  of \texorpdfstring{\ifbig \boldmath \fi$C_5$}{C5}, \texorpdfstring{\ifbig\boldmath\fi$C_6$}{C6} and \texorpdfstring{\ifbig\boldmath\fi$C^*$}{C*}}
Also in this case the octagon value of an $I$-graph equals to $(8,8,8)$. Using similar approaches as before we get results that are summed up in \cref{tab:C2/C4}.

\begin{table}[h!]
\centering
\resizebox{\ifsmall 0.48 \fi\columnwidth}{!}{
  \begin{tabular}{c||c|c|c|c}
    \backslashbox{$C_5$}{$C_6$} & $2k + 4j = n$ & $k + 2j = n$ & $4j - 2k = n$ & $4j - 2k= 0$ \\
    \hline \hline
    $4k + 2j = n$               & not exist     & not exist    & $G(10,3)$   & $G(10,2)$  \\
    \hline
    $2k + j = n$                & not exist     & not exist    & $G(10,2)$   & $G(5,2)$   \\
    \hline
    $4k - 2j = n$               & $G(10,3)$   & $G(10,2)$  & not exist     & not exist    \\
  \end{tabular}}
  \caption{$I$-graphs containing $C_5, C_6$ and $C^*$.}
  \label{tab:C2/C4}
\end{table}

\subsubsection{No \texorpdfstring{\ifbig\boldmath\fi$8$}{8}-cycles.}
In this case the octagon value is $(0,0,0)$. A triangular prism graph or $G(3,1)$ is the only such graph from the family of $I$-graphs.

\begin{figure}[ht!]
  \begin{subfigure}{\ifbig.245\fi \ifsmall .18 \fi \textwidth}
    \centering
    \includegraphics[width=\linewidth]{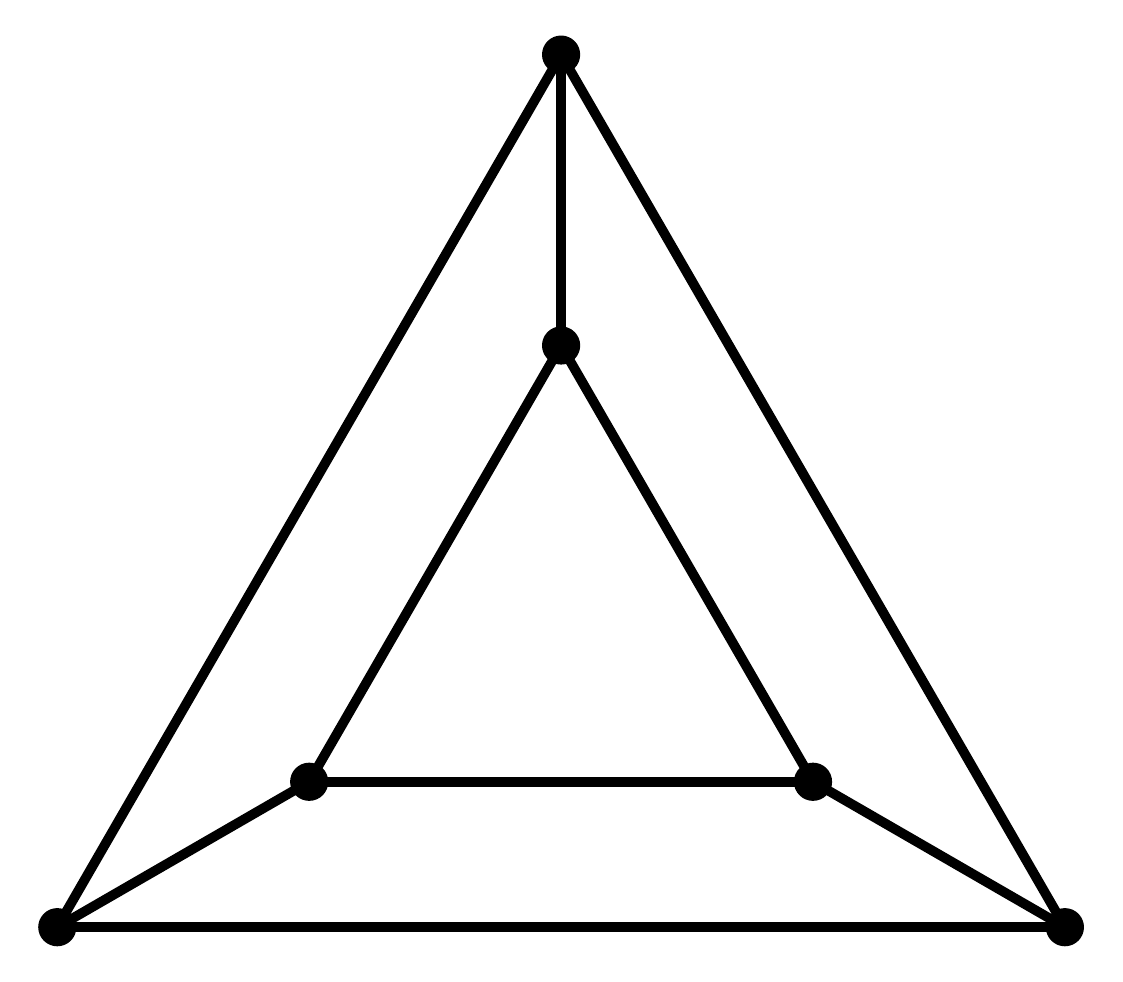}
    \caption{Triangular prism\\ graph.}
  \end{subfigure}
  \begin{subfigure}{\ifbig.245\fi \ifsmall .18 \fi \textwidth}
    \centering
    \includegraphics[width=\linewidth]{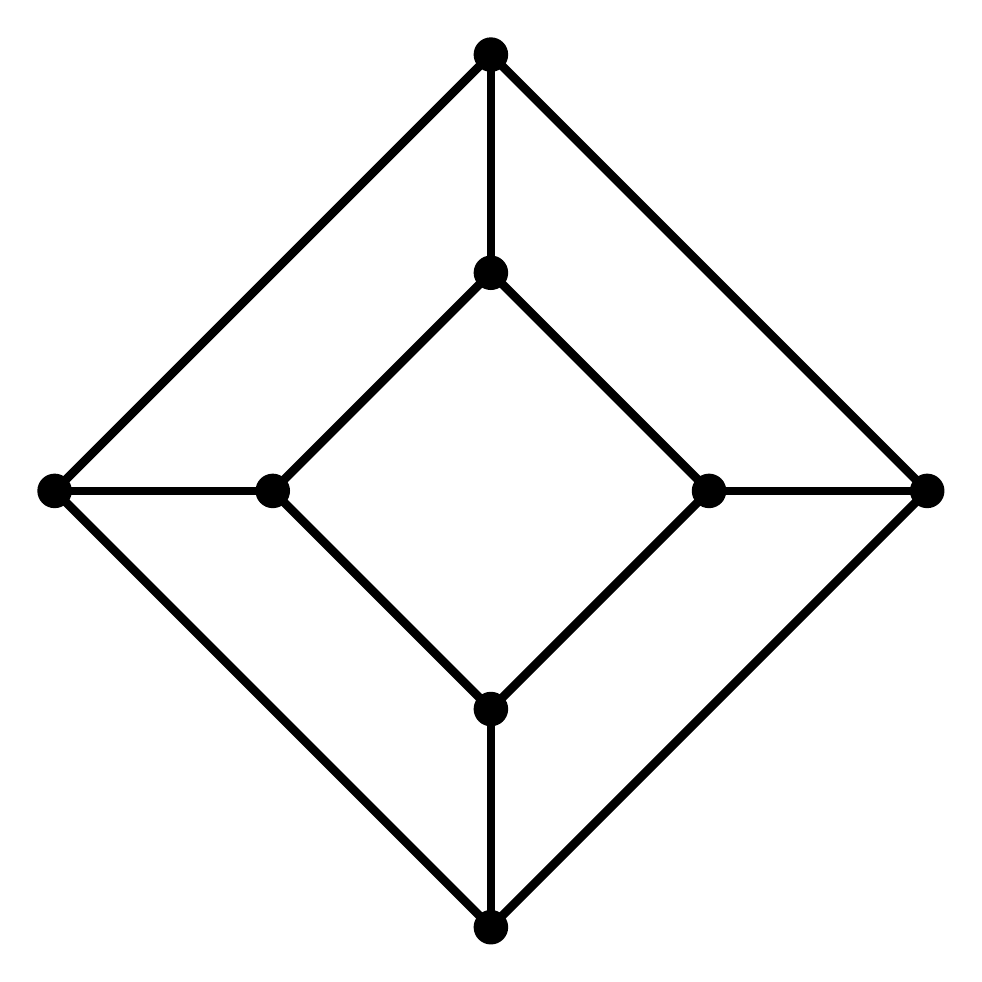}
    \caption{$3$-cube graph.}
  \end{subfigure}
  \begin{subfigure}{\ifbig.245\fi \ifsmall .18 \fi \textwidth}
    \centering
    \includegraphics[width=\linewidth]{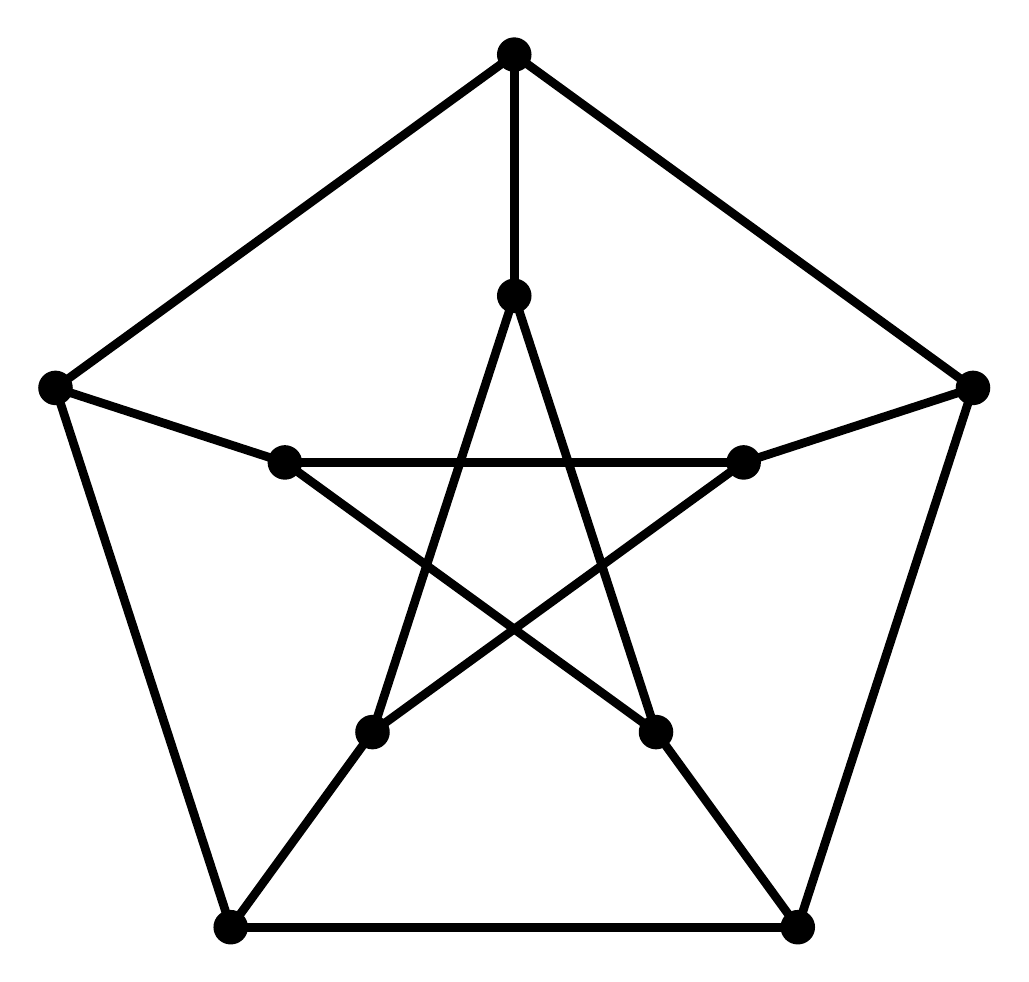}
    \caption{Petersen graph.}
    \label{GP5-2}
  \end{subfigure}
  \begin{subfigure}{\ifbig.245\fi \ifsmall .18 \fi \textwidth}
    \centering
    \includegraphics[width=\linewidth]{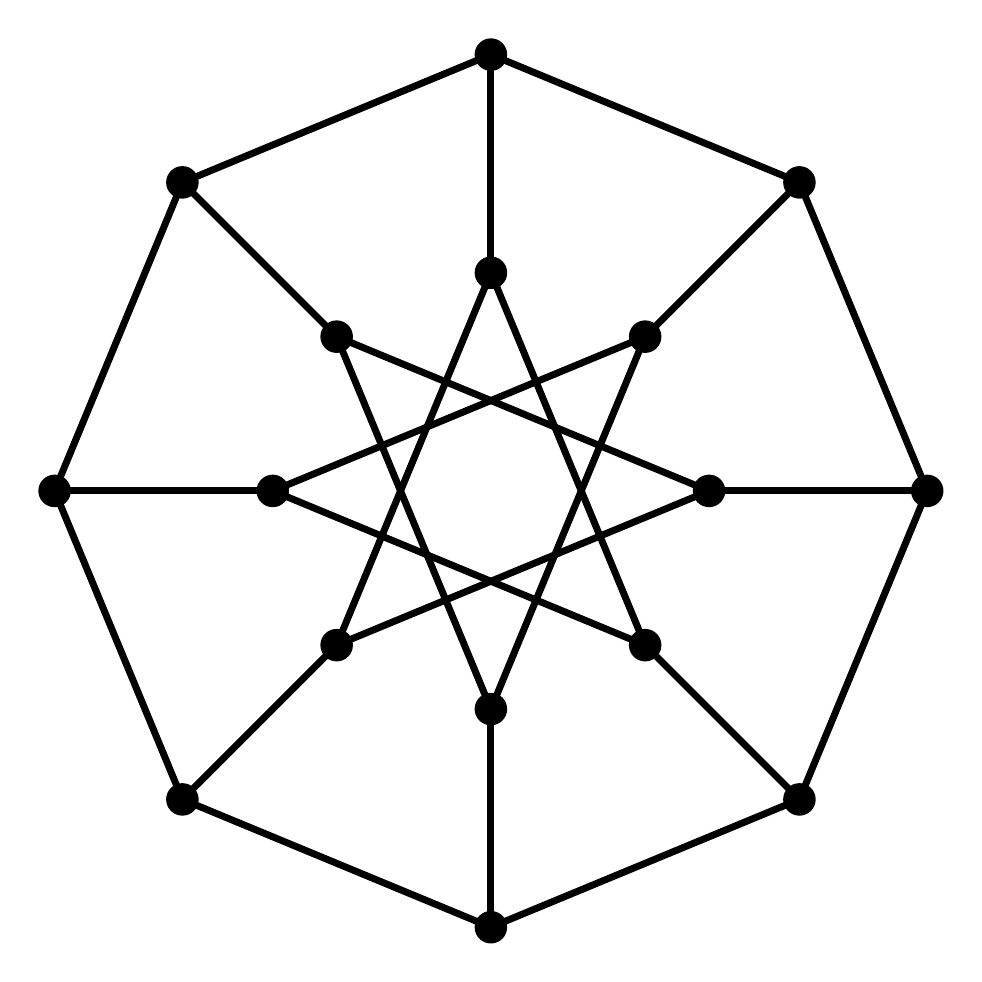}
    \caption{M\"{o}bius-Kantor\\ graph.}
    \label{GP8-3}
  \end{subfigure}
  \ifbig \\ \fi
  \begin{subfigure}{\ifbig.245\fi \ifsmall .18 \fi \textwidth}
    \centering
    \includegraphics[width=\linewidth]{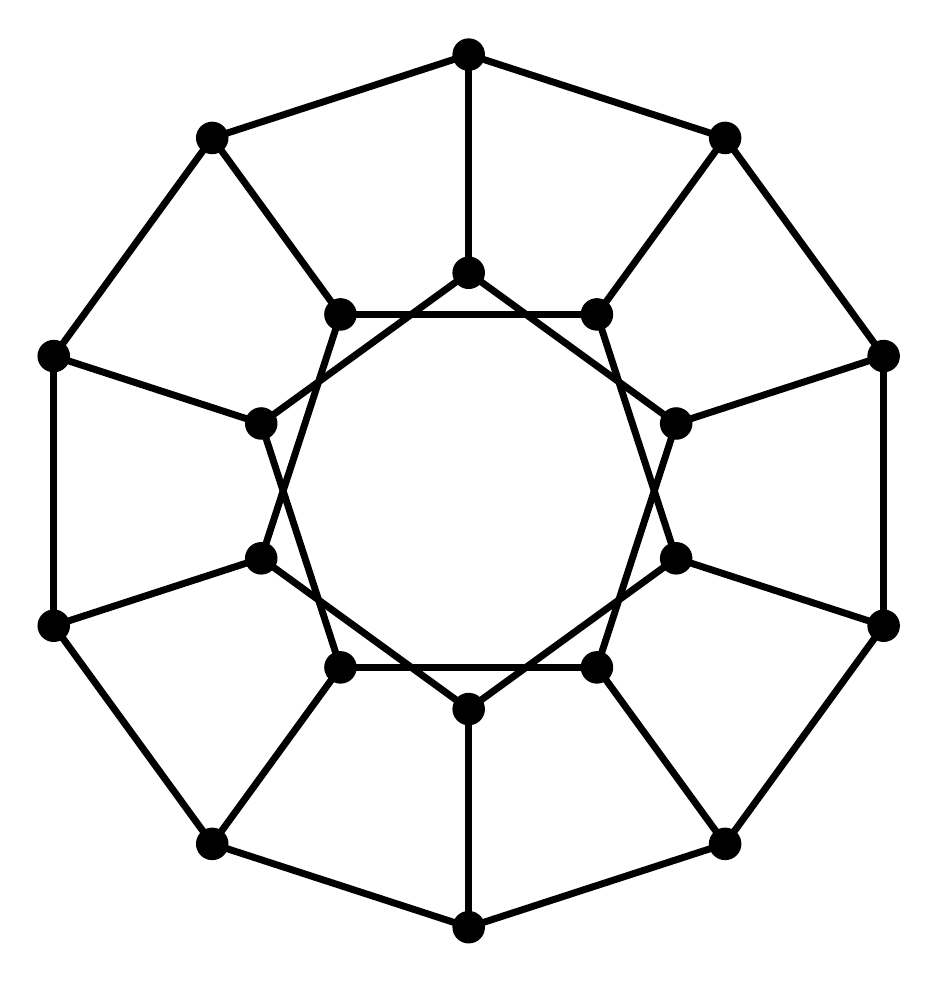}
    \caption{Dodecahedral graph.}
    \label{ForbiddenGP10,2}
  \end{subfigure}
  \begin{subfigure}{\ifbig.245\fi \ifsmall .19 \fi \textwidth}
    \centering
    \includegraphics[width=\linewidth]{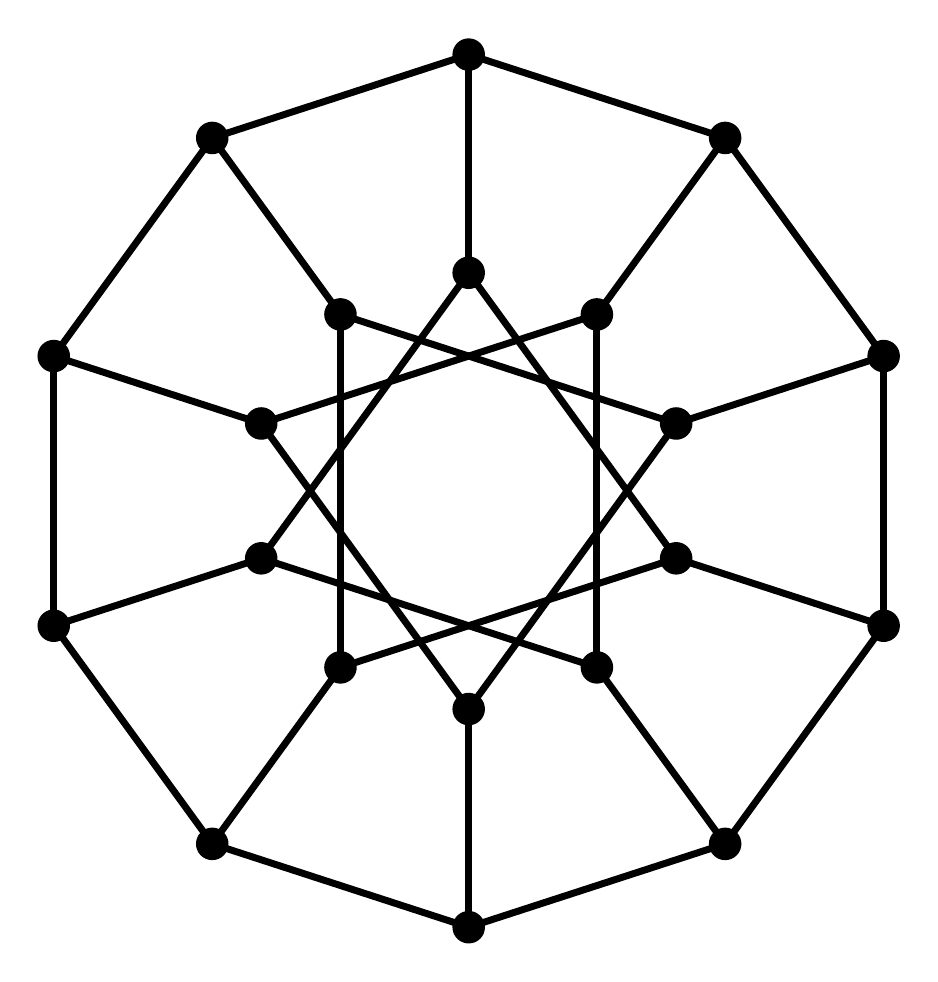}
    \caption{Desargues graph.}
    \label{ForbiddenGP10,3}
  \end{subfigure}
  \begin{subfigure}{\ifbig.245\fi \ifsmall .19 \fi \textwidth}
    \centering
    \includegraphics[width=\linewidth]{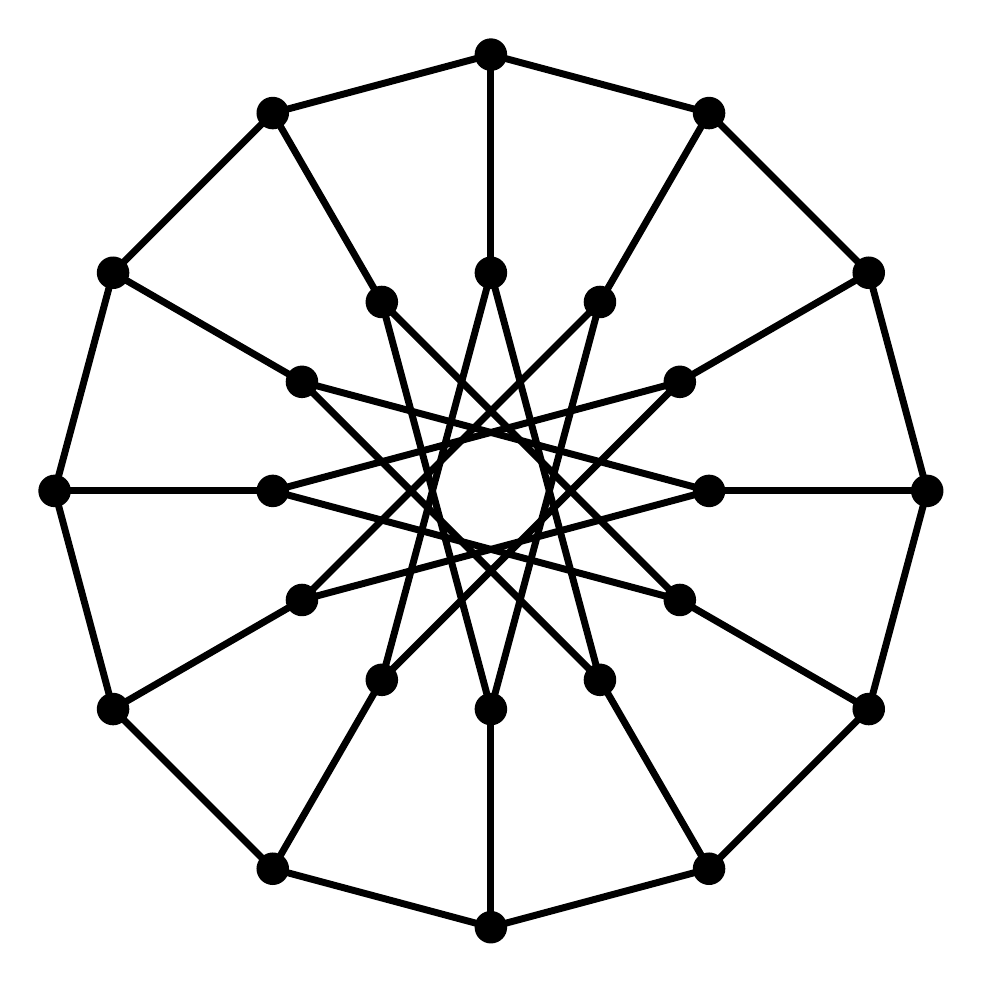}
    \caption{Nauru graph.}
    \label{ForbiddenGP12,5}
  \end{subfigure}
  \begin{subfigure}{\ifbig.245\fi \ifsmall .19 \fi \textwidth}
    \centering
    \includegraphics[width=\linewidth]{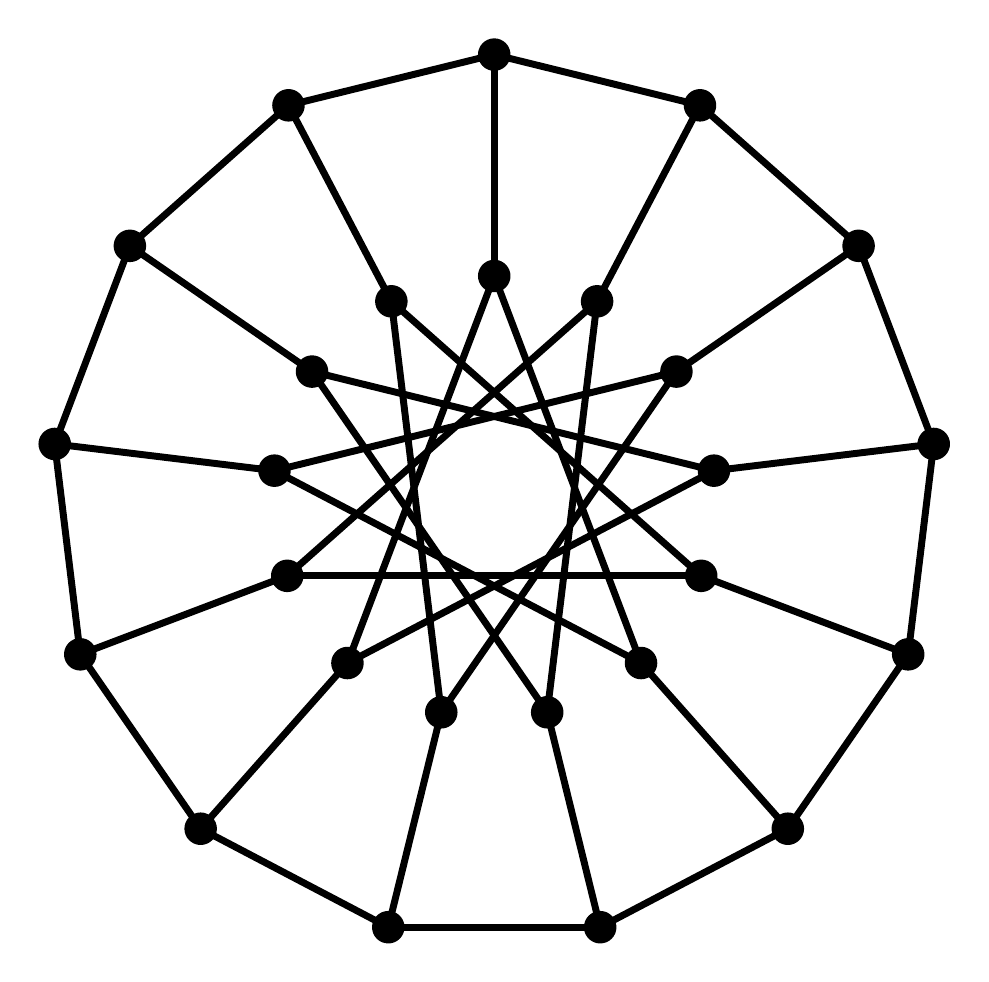}
    \caption{$G(13,5)$.}
    \label{ForbiddenGP13,5}
  \end{subfigure}
  \ifbig \\ 
  \begin{center} \fi
    \begin{subfigure}{\ifbig.245\fi \ifsmall .19 \fi \textwidth}
      \centering
      \includegraphics[width=\linewidth]{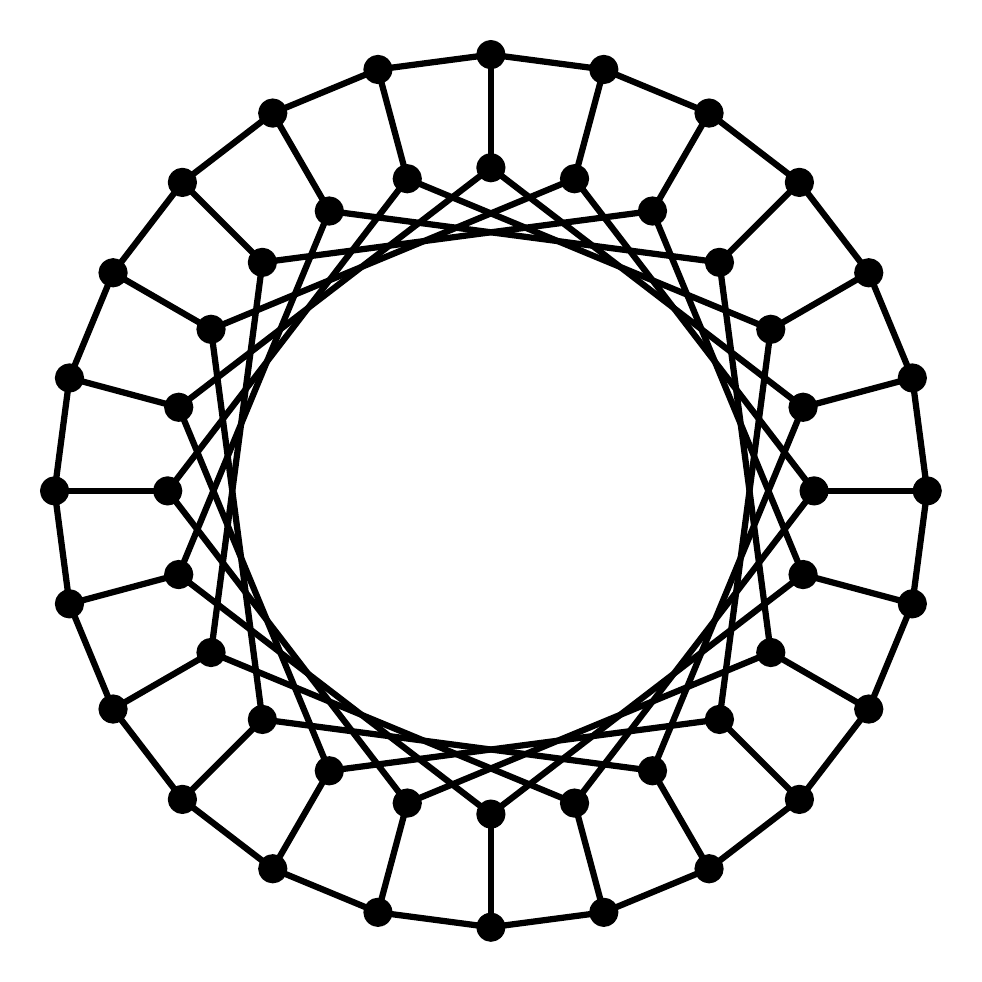}
      \caption{$F_{048} \cong G(24,5)$.}
      \label{ForbiddenGP24,5}
    \end{subfigure}
    \begin{subfigure}{\ifbig.245\fi \ifsmall .19 \fi \textwidth}
      \centering
      \includegraphics[width=\linewidth]{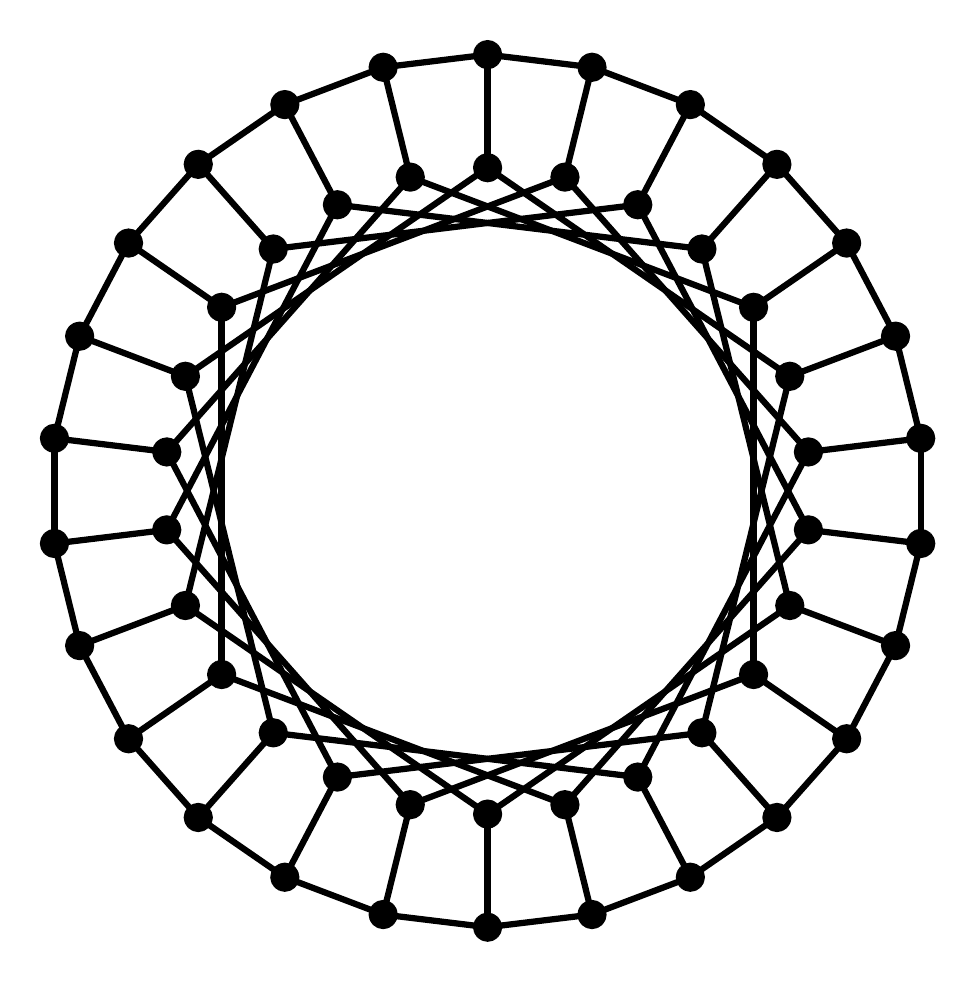}
      \caption{$G(26,5)$.}
      \label{ForbiddenGP26,5}
    \end{subfigure}
  \ifbig  \end{center} \fi
  \caption{All $[1,\lambda,8]$-cycle regular $I$-graphs.}
  \label{fig:specialIgraphs}
\end{figure}

\subsection{\texorpdfstring{\ifbig \boldmath \fi$[1,\lambda,8]$-cycle regular $I$-graphs}{[1,lambda,8]-cycle regular I-graphs}}
After considering all the cases we see that there is a finite list of $[1,\lambda,8]$-cycle regular $I$-graphs. Using isomorphism properties of $I$-graphs we can determine which graphs from \cref{tab:C3/C4,tab:C2/C4} are isomorphic. It turns out that they are all in the family of generalized Petersen graphs.
All $[1,\lambda,8]$-cycle regular graphs are listed in \cref{tab:IForbidden} together with their well-known names and are depicted in \cref{fig:specialIgraphs}.

\begin{table}[h!]
  \centering
  \resizebox{\ifsmall 0.40 \fi \ifbig 0.6 \fi \columnwidth}{!}{
  \begin{tabular}{c|c|c}
    \textbf{Graph} & \textbf{Value of} \boldmath $\lambda$ & \textbf{Common name}     \\
    \hline
    $G(3,1)$       & $0$                                   & Triangular prism graph   \\
    \hline
    $G(4,1)$       & $4$                                   & $3$-cube graph           \\
    \hline
    $G(5,2)$       & $8 $                                  & Petersen graph           \\
    \hline
    $G(8,3)$       & $8 $                                  & M\"{o}bius-Kantor graph  \\
    \hline
    $G(10,2)$      & $8 $                                  & Dodecahedral graph       \\
    \hline
    $G(10,3)$      & $8 $                                  & Desargues graph          \\
    \hline
    $G(12,5)$      & $8 $                                  & Nauru graph ($F_{024}A$) \\
    \hline
    $G(13,5)$      & $8 $                                  & $/$                      \\
    \hline
    $G(24,5)$      & $8 $                                  & $F_{048}A$               \\
    \hline
    $G(26,5)$      & $8 $                                  & $/$                      \\
  \end{tabular}}
  \caption{$[1,\lambda,8]$-cycle regular $I$-graphs. \label{tab:IForbidden}}
\end{table}

It is worth mentioning that, with the exception of $G(13,5)$ and $G(26,5)$ all other $[1,\lambda,8]$-cycle regular graphs are edge-transitive. In fact, these are the only edge transitive graphs in the family of generalized Petersen graphs (see \cite{Frucht/Graver/Watkins:1971}).

\section{Double generalized Petersen graphs \label{section-DPgraphs}}
\emph{Double generalized Petersen graphs} $\DP(n,k)$ represent another natural generalization of generalized Petersen graphs. They are defined for integers $n \geq 3, k < n / 2$, on the vertex set $\lbrace u_0, u_1, \dots ,u_{n-1}, w_0, w_1, \dots, w_{n-1}, x_0, x_1, \dots, x_{n-1},y_0, y_1,$ $\dots, y_{n-1} \rbrace$ and the edge set consisting of outer edges $u_i u_{i+1}, x_i x_{i+1}$, inner edges $w_i y_{i + k}, y_i w_{i + k}$ and spoke edges $u_i w_i, x_i y_i$, where the subscripts are taken modulo $n$.
To get a better feeling for the structure of this graph family we append two interesting observations.

\begin{thm}[Qin et al. \cite{Qin/Xia/Zhou:2018}] \label{claim:DPisom}
  $\DP(n,k)$ and $G(n',k')$ are isomorphic if and only if $n$ is an odd integer and $\gcd(n,k)=1$. In this case $n' = 2n$ and $k'$ is a unique even integer, such that ${k k' \equiv \pm 1 \pmod n}$.
\end{thm}

\isomorphDP*

\begin{proof}\label{DP-proof}
  All subscripts in the proof are taken modulo $n$.

  Let $u_i, w_i, x_i, y_i$, for $i = 0, \dots , n-1$ be vertices of graph $\DP(n,k)$ with edges
  \[u_i u_{i+1}, x_i x_{i+1}, w_i y_{i + k}, y_i w_{i + k}, u_i w_i, x_i y_i.\]
  Let $j = n/2 - k$.
  We define a permutation $\Pi$ on vertices of $\DP(n,k)$ as follows:
  \begin{align*}
    u_i & \mapsto u_i,           \\
    w_i & \mapsto w_i,           \\
    x_i & \mapsto x_{i + n/2},   \\
    y_i & \mapsto y_{i + n / 2},
  \end{align*}
  for $ i = 0, 1, \dots , n-1$.
  If we want $\Pi$ to be an isomorphism it needs to preserve all the edges set-wise.
  It is easy to see that the outer edges $u_i u_{i+1}, x_{i} x_{i+1}$ are mapped to outer edges and spoke edges $u_i v_{i}, x_{i} y_{i}$ to spoke edges.
  In the case of inner edges, $\Pi$ maps $w_i y_{i + k}$ to $w_i y_{i + k + n/2}$. Since $k + n /2 \equiv n - (k + n/2) \equiv n/2 - k \equiv j \pmod n$, we can write this edge as $w_i y_{i + j}$. Similarly we check that $y_{n / 2 + i}  w_{n / 2 + i + k}$ is mapped to $y_i w_{i + j}$.
  Therefore $\DP(n,k) \cong \DP(n,j)$.
\end{proof}

\begin{figure}[ht!]
  \centering
  \begin{subfigure}[t]{\ifbig .28 \fi \ifsmall .22 \fi \textwidth}
    \centering
    \includegraphics[width=\linewidth]{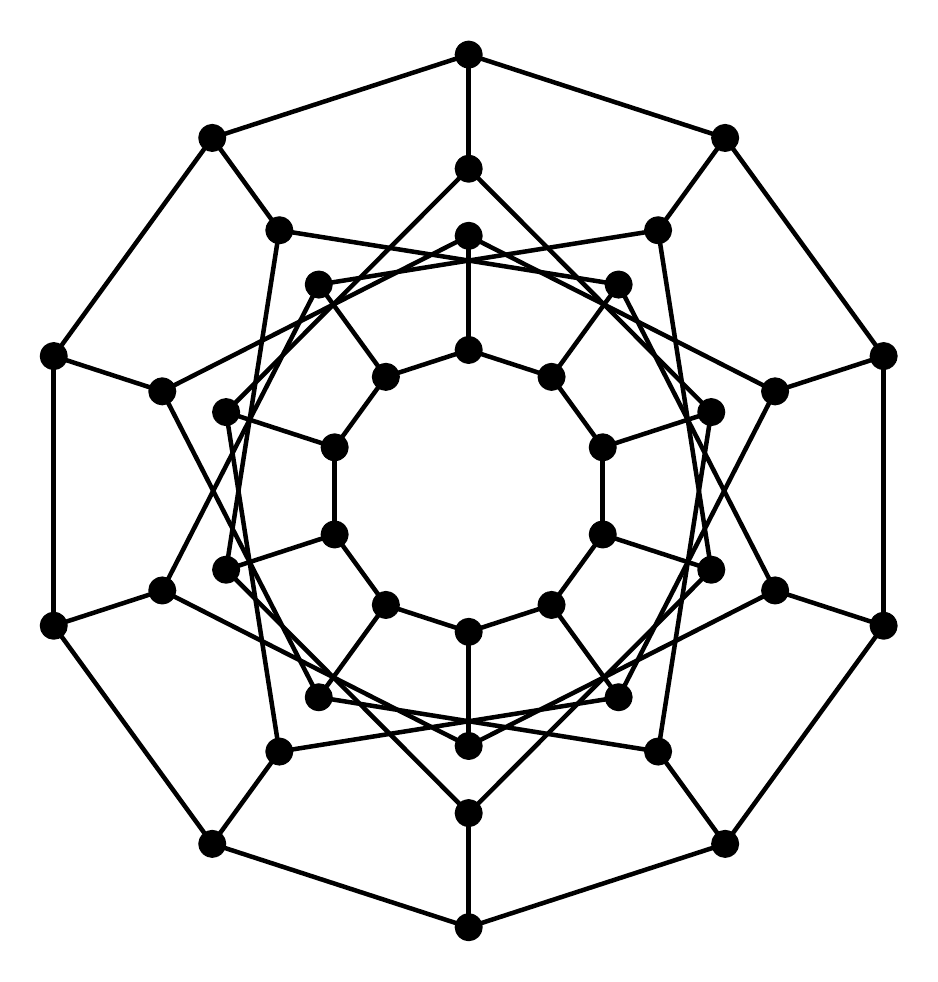}
  \end{subfigure}
  \quad \quad
  \begin{subfigure}[t]{\ifbig .28 \fi \ifsmall .22 \fi \textwidth}
    \centering
    \includegraphics[width=\linewidth]{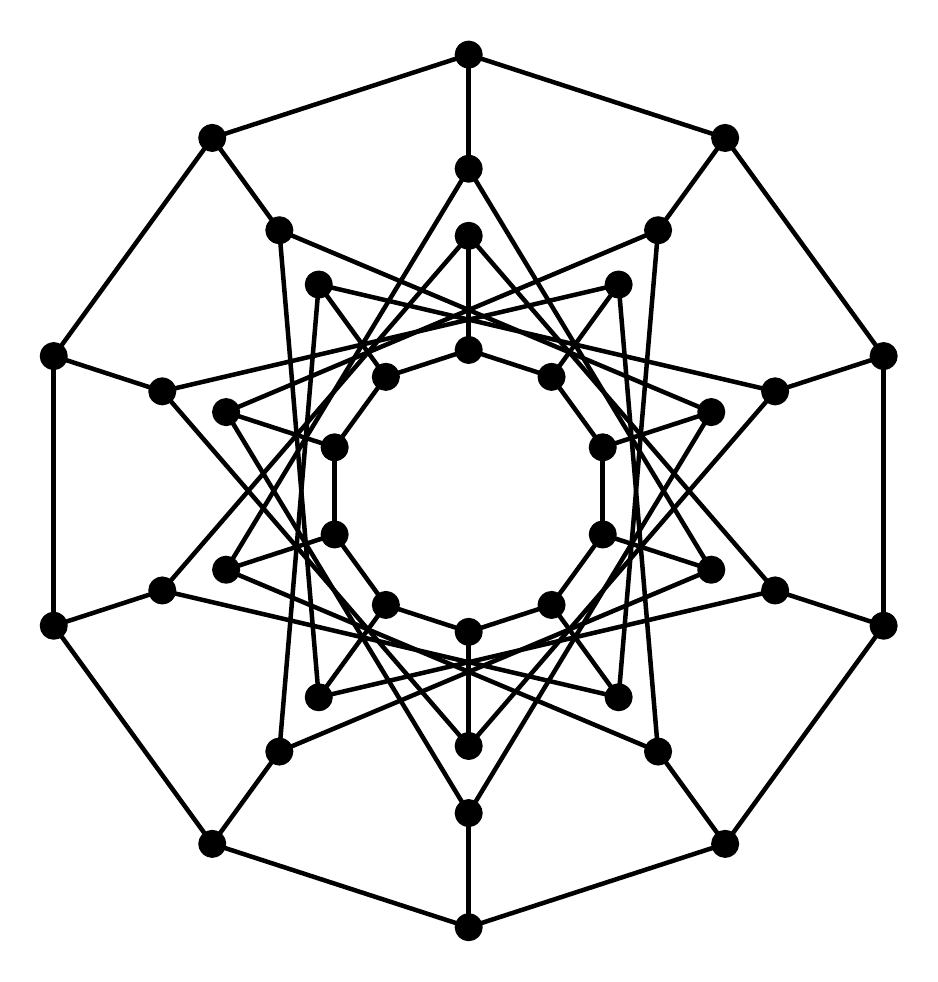}
  \end{subfigure}
  \caption{Isomorphic graphs $\DP(10,2)$ and $\DP(10,3)$.}
  \label{fig:DPisomGraphs}
\end{figure}

\subsection{Equivalent \texorpdfstring{\ifbig \boldmath \fi$8$}{8}-cycles}
Every double generalized Petersen graph admits the rotation $\rho$, defined as: $\rho(u_i) = u_{i +1}$, $\rho(w_i) = w_{i+1}$, $\rho(x_i) = x_{i +1}$, $\rho(y_i) = y_{i+1}$ and an automorphism $\alpha$ that interchanges vertices of one copy of a generalized Petersen graph with another: $\alpha(u_i) = x_i, \alpha(w_i)=y_i, \alpha(x_i) = u_i$ and $\alpha(y_i) = w_i$.
When acting with $\rho$ and $\alpha$ on any double generalized Petersen graph we get $3$ edge orbits: orbit of outer edges $E_J$, orbit of spoke edges $E_S$ and orbit of inner edges $E_I$.
Orbit of outer edges induces two cycles of length $n$, orbit of spoke edges induces a perfect matching and orbit of inner edges induces either $2 \cdot \gcd(n,k)$ cycles of length $n / \gcd(n,k)$ if $n/ \gcd(n,k)$ is even and $\gcd(n,k) > 1$, or $\gcd(n,k)$ cycles of length $2n / \gcd(n,k)$ otherwise.

Each double generalized Petersen graph admits at least one $8$-cycle.
\begin{claim}
  For positive integers $n,k$ where $n \geq 3$ and $k < n/2$, graph $\DP(n,k)$ always admits an $8$-cycle
  $$C^* =(w_0, y_{\pm k}, x_{\pm k}, x_{\pm k \pm 1}, y_{\pm k \pm 1}, w_{\pm 1}, u_{\pm 1},u_0).$$
\end{claim}
To study the structure of $8$-cycles in double generalized Petersen graphs we use the same approaches as in the case of $I$-graphs, therefore some notation and calculation is omitted.
We say that two $8$-cycles are equivalent if we can map one into the other not only using the rotation $\rho$ but also automorphism $\alpha$.
To calculate the contribution $\tau$ of an $8$-cycle to the octagon value of a double generalized Petersen graph, we again calculate $\delta_j, \delta_s, \delta_i$ by counting the number of outer, spoke and inner edges but in this case we multiply these numbers with $\gamma / 2n$.

\subsection{Characterization of non-equivalent \texorpdfstring{\ifbig \boldmath \fi$8$}{8}-cycles}
In this section we provide a list of all possible $8$-cycles in double generalized Petersen graphs and determine their contribution to the octagon value of a double generalized Petersen graph. The results are summarized in \cref{tab:DP8conditions,tab:DP8cycles} and shown in \cref{fig:DP-graphs}.

\begin{figure}[ht!]
  \centering
  \begin{subfigure}{\ifbig .25 \fi \ifsmall .22 \fi\textwidth}
    \centering
    \includegraphics[width=\linewidth]{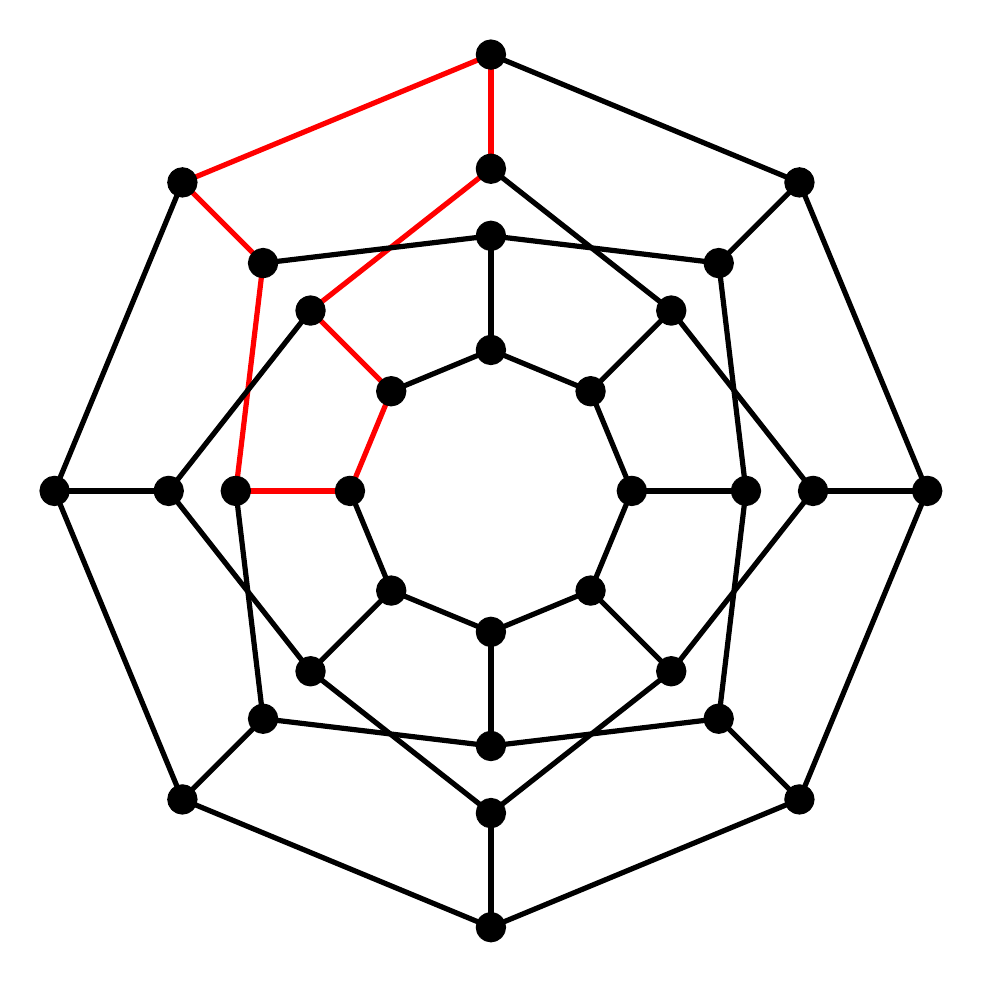}
    \caption{Cycle $C^*$.}
    \label{fig:DPC*}
  \end{subfigure}
  \quad
  \begin{subfigure}{\ifbig .25 \fi \ifsmall .22 \fi\textwidth}
    \centering
    \includegraphics[width=\linewidth]{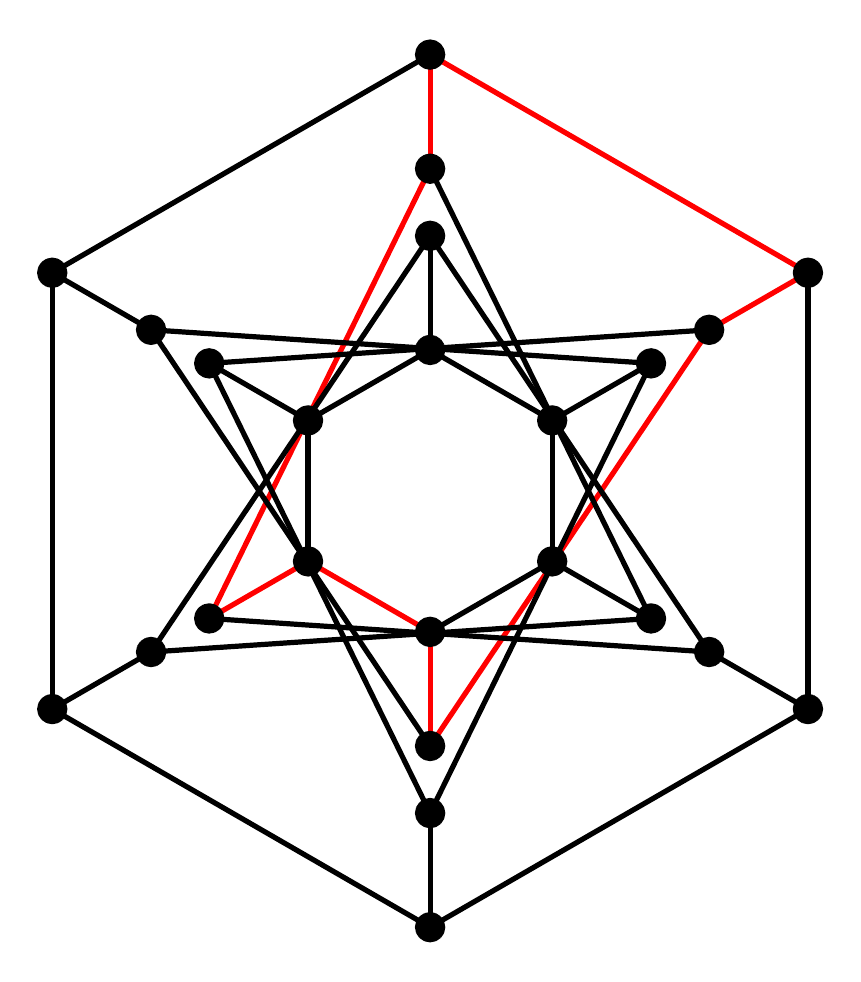}
    \caption{Cycle $C_0$.}
    \label{fig:DPC0}
  \end{subfigure}
  \quad
  \begin{subfigure}{\ifbig .25 \fi \ifsmall .22 \fi\textwidth}
    \centering
    \includegraphics[width=\linewidth]{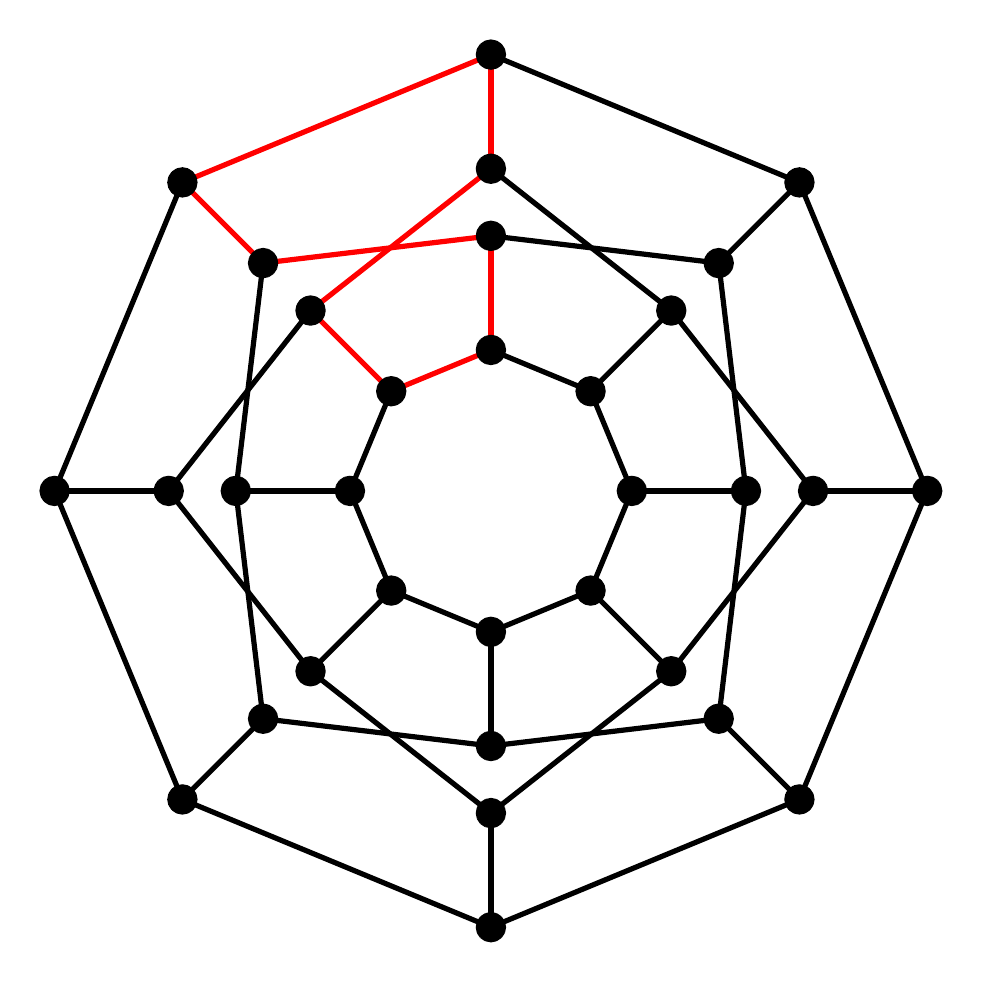}
    \caption{Cycle $C_1$.}
    \label{fig:DPC1}
  \end{subfigure}
  \\
  \begin{subfigure}{\ifbig .25 \fi \ifsmall .22 \fi\textwidth}
    \centering
    \includegraphics[width=\linewidth]{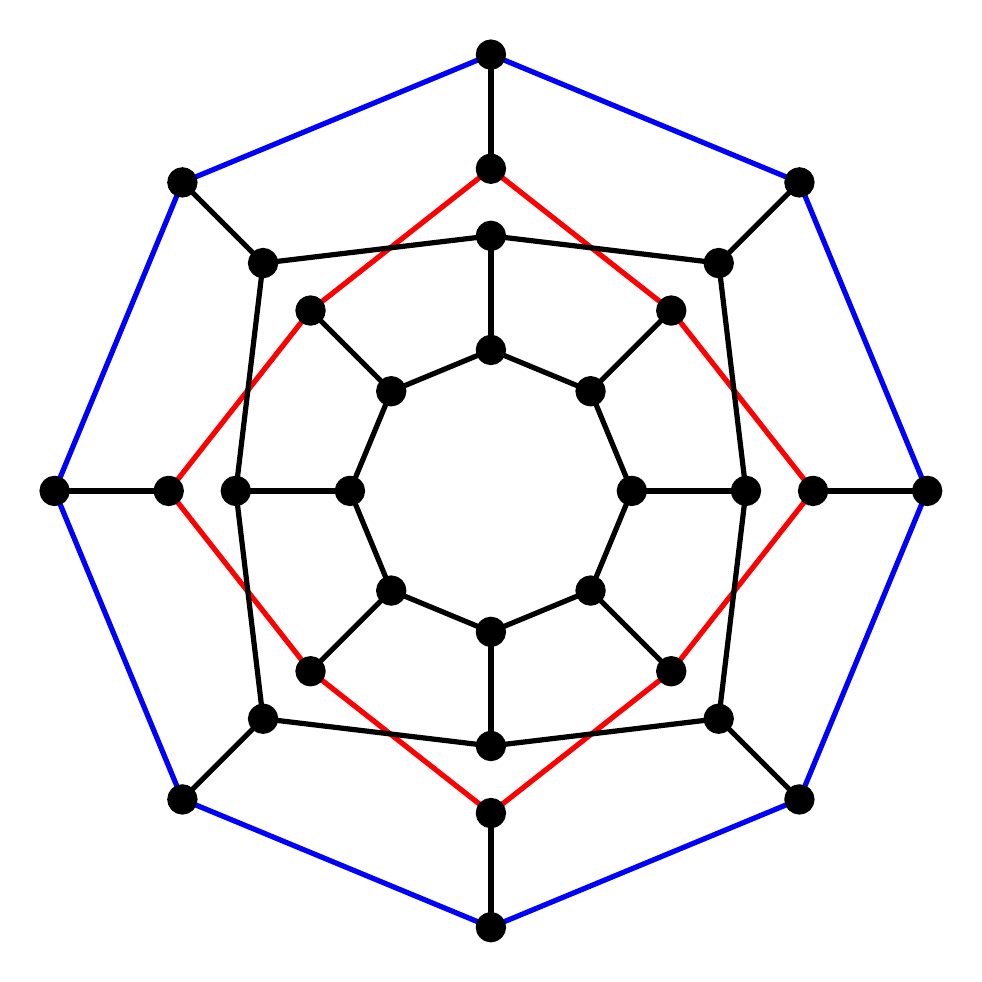}
    \caption{Cycle $C_2$ in blue and $C_3$ in red.}
    \label{fig:DPC2C3}
  \end{subfigure}
  \quad
  \begin{subfigure}{\ifbig .25 \fi \ifsmall .22 \fi\textwidth}
    \centering
    \includegraphics[width=\linewidth]{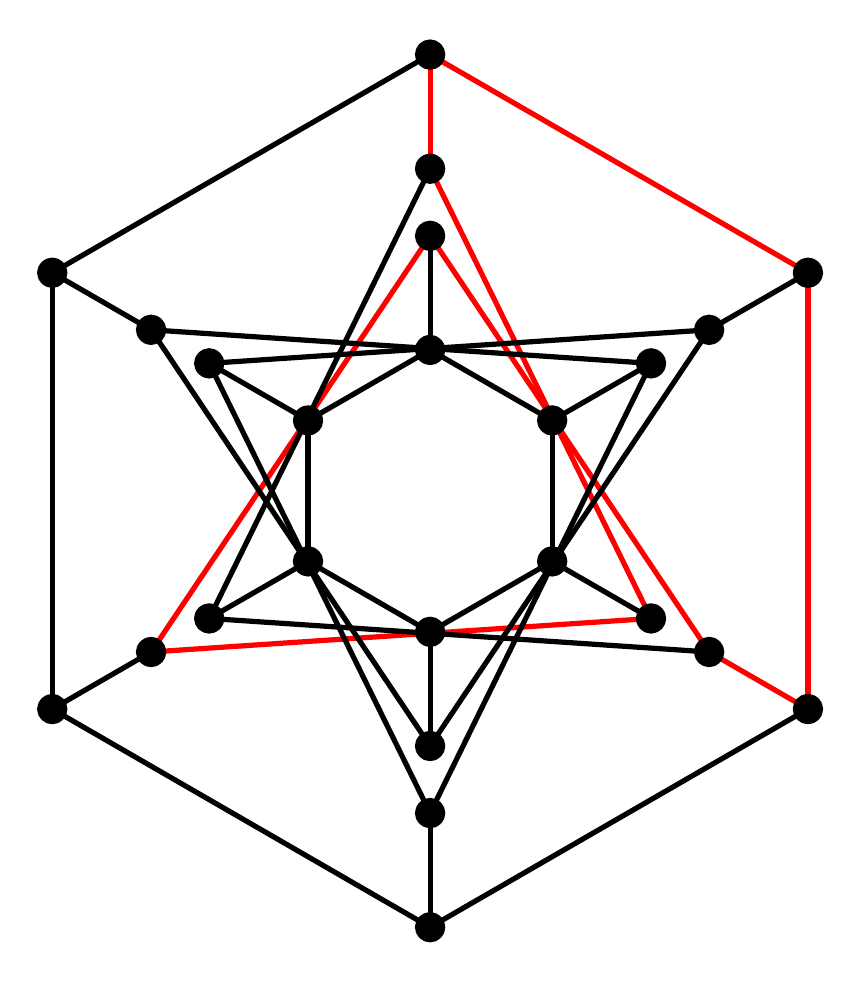}
    \caption{Cycle $C_4$.}
    \label{fig:DPC4}
  \end{subfigure}
  \quad
  \begin{subfigure}{\ifbig .25 \fi \ifsmall .22 \fi\textwidth}
    \centering
    \includegraphics[width=\linewidth]{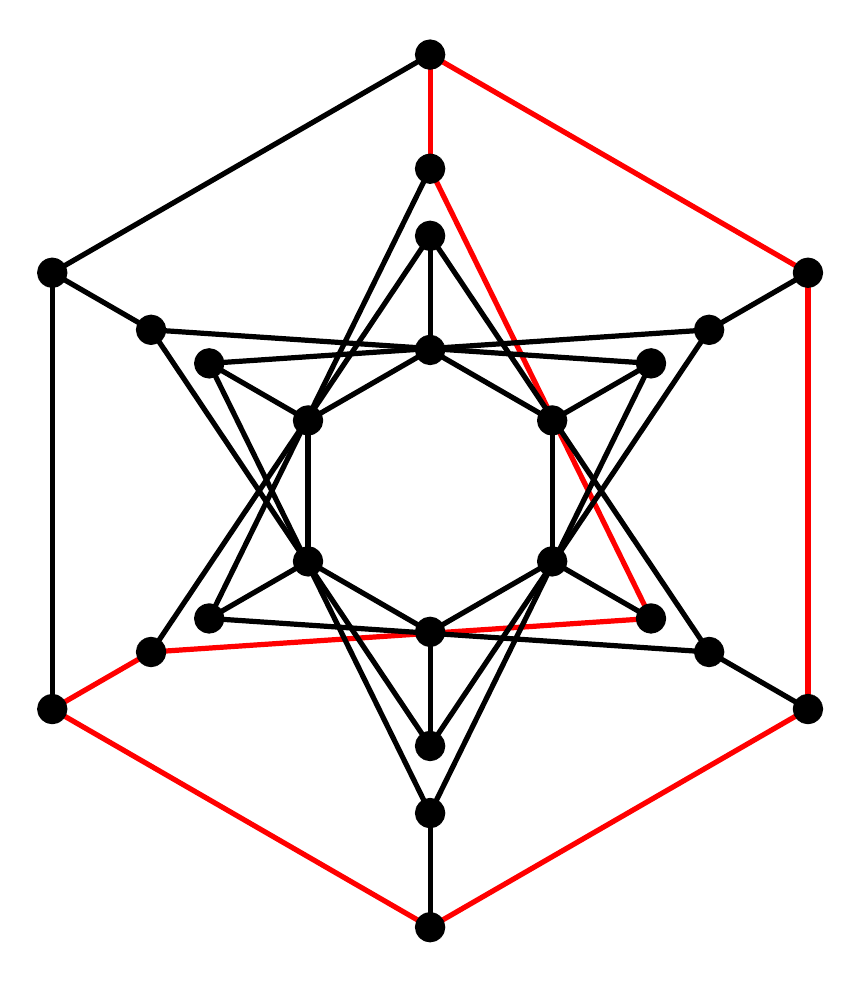}
    \caption{Cycle $C_5$.}
    \label{fig:DPC5}
  \end{subfigure}
  \caption{Examples of all non-equivalent $8$-cycles in $DP$-graphs.}
\end{figure}

The related case analysis is discussed in the subsections below, and organized as follows.
We first distinguish $8$-cycles by the number of spoke edges they admit. Indeed, it is easy to see that an arbitrary $8$-cycle can have either $4, 0$ or $2$ spoke edges. The first two cases correspond to \cref{sec:DPFourSpokes,sec:DPnoSpokes}, respectively.
For the last case we further distinguish cases by the number of outer and inner edges within a given $8$-cycle. Those cases are discussed in \cref{sec:DP2spokes}.

Even though double generalized Petersen graphs are undirected, we use the orientation of edges in the following $8$-cycle analysis, to make the process easier. We say that an edge $e=v_iv_j$ of a graph $\DP(n,k)$ is oriented \emph{positively} if $i + 1 \equiv j \pmod n$ or $i + k  \equiv j \pmod n$ and is oriented \emph{negatively} if $i - 1 \equiv j \pmod n$ or $i - k  \equiv j \pmod n$.


\subsubsection{\texorpdfstring{\ifbig \boldmath \fi$8$}{8}-cycles with four spoke edges} \label{sec:DPFourSpokes}
In this case the $8$-cycle admits also two inner and two outer edges. 
We distinguish $3$ different such $8$-cycles regarding the orientation of outer and inner edges. If one of the outer edges and one of the inner edges have positive orientation and the other outer and inner edge have negative orientation, we have a cycle of the form $C^*$ (see \cref{fig:DPC*}). If all edges have the same orientation (\cref{fig:DPC0}) the cycle is of the following form:
$$C_0 =(w_0, y_{\pm k}, x_{\pm k}, x_{\pm k \pm 1}, y_{\pm k \pm 1}, w_{\pm 2k \pm 1}, u_{\pm 2k \pm 1},u_{\pm 2k \pm 2}).$$
If outer edges are positively oriented and inner edges negatively, or vice versa, the cycle is of form
$$ C_1 = (w_0, y_{\pm k}, x_{\pm k}, x_{\pm k \mp 1}, y_{\pm k \mp 1}, w_{\pm 2k \mp 1}, u_{\pm 2k \mp 1},u_{\pm 2k \mp 2}),$$
see also \cref{fig:DPC1}.
\subsubsection{\texorpdfstring{\ifbig \boldmath \fi$8$}{8}-cycle with no spoke edges} \label{sec:DPnoSpokes}
There are two such non-equivalent $8$-cycles, one on outer edges $$C_2 = (u_0,u_1,u_2,u_3,u_4,u_5,u_6,u_7)$$ and one on inner edges
$$C_3 =(w_0, y_k, w_{2k}, y_{3k}, w_{4k}, y_{5k}, w_{6k}, y_{7k}),$$
see \cref{fig:DPC2C3}.

\subsubsection{\texorpdfstring{\ifbig \boldmath \fi$8$}{8}-cycles with two spoke edges} \label{sec:DP2spokes}
It is not hard to see that the $8$-cycle of this form admits an even number of inner edges, i.e. either $2$ or $4$. Therefore we have two possibilities for these $8$-cycles.
\begin{align*}
  C_4 & =(w_0, y_k, w_{2k},  y_{3k}, w_{4k}, u_{4k}, u_{4k \pm 1}, u_{4k \pm 2}) \quad \text{or} \\
  C_5 & =(u_0, u_1, u_{2},  u_{3}, u_{4}, w_{4}, y_{4 \pm k}, w_{4 \pm 2k}),
\end{align*}
see \cref{fig:DPC5,fig:DPC4}.

All possible $8$-cycles that can be found in double generalized Petersen graphs are listed in \cref{tab:DP8conditions,tab:DP8cycles}, together with their existence conditions, contribution to graph octagon value and number of equivalent $8$-cycles.

\begin{table}[ht!]
  \centering
  \resizebox{\ifbig \fi \ifsmall 0.65 \fi \columnwidth}{!}{
  \begin{tabular}{c|c|c|c|c|c|c|c}
    \textbf{Label}       & $C^*$       & $C_0$             & $C_1$             & $C_2$     & $C_3$             & $C_4$       & $C_5$       \\
    \hline
    \boldmath$\tau(C)$   & $(2,4,2)$   & $(1,2,1)$         & $(1,2,1)$         & $(1,0,0)$ & $(0,0,1)$         & $(2,2,4)$   & $(4,2,2)$   \\

    \boldmath$\gamma(C)$ & $n \cdot 2$ & $(n / 2) \cdot 2$ & $(n / 2) \cdot 2$ & $2$       & $(n / 8) \cdot 2$ & $n \cdot 2$ & $n \cdot 2$
  \end{tabular}}
  \caption{Contribution of $8$-cycles to octagon value and their occurrences in $\DP$-graphs.}
  \label{tab:DP8cycles}
\end{table}

\begin{table}[ht!]
  \centering
  \resizebox{\ifbig \fi \ifsmall 0.7 \fi \columnwidth}{!}{
  \begin{tabular}{c|c|c}
    \textbf{Label}         & \textbf{A representative of an \boldmath$8$-cycle}                                                                & \textbf{Existence conditions} \\
    \hline
    $C^*$                  & $(w_0, y_{\pm k}, x_{\pm k}, x_{\pm k \pm 1}, y_{\pm k \pm 1}, w_{\pm 1}, u_{\pm 1},u_0)$               & $n \geq 3$                    \\
    \hline
    $C_0$                  & $(w_0, y_{\pm k}, x_{\pm k}, x_{\pm k \pm 1}, y_{\pm k \pm 1}, w_{\pm 2k \pm 1}, u_{\pm 2k \pm 1},u_{\pm 2k \pm 2})$ & $2k + 2 = n$                  \\
    \hline
    $C_1$                  & $(w_0, y_{\pm k}, x_{\pm k}, x_{\pm k \mp 1}, y_{\pm k \mp 1}, w_{\pm 2k \mp 1}, u_{\pm 2k \mp 1},u_{\pm 2k \mp 2})$ & $k=1$                         \\
    \hline
    $ C_2$                 & $(u_0,u_1,u_2,u_3,u_4,u_5,u_6,u_7)$                                                                      & $n = 8$                       \\
    \hline
    $C_3$                  & $(w_0, y_k, w_{2k}, y_{3k}, w_{4k}, y_{5k}, w_{6k}, y_{7k})$                                             & $ 8k = n$ or  $3n$            \\
    \hline
    \multirow{2}{*}{$C_4$} & $(w_0, y_k, w_{2k},  y_{3k}, w_{4k}, u_{4k}, u_{4k + 1}, u_{4k + 2})$                                    & $4k + 2 = n$ or $2k + 1 = n$  \\
    \cline{2-3}            & $(w_0, y_k, w_{2k},  y_{3k}, w_{4k}, u_{4k}, u_{4k - 1}, u_{4k - 2})$                                    & $4k - 2 = n$                  \\
    \hline
    \multirow{2}{*}{$C_5$} & $(u_0, u_1, u_{2},  u_{3}, u_{4}, w_{4}, y_{4 + k}, w_{4 + 2k})$                                     & $2k + 4 = n$  \\
    \cline{2-3}            & $(u_0, u_1, u_{2},  u_{3}, u_{4}, w_{4}, y_{4 - k}, w_{4 - 2k})$                                   & $2k - 4= 0$           \\
  \end{tabular}}
  \caption{Characterization of non-equivalent $8$-cycles of $\DP$-graphs.}
  \label{tab:DP8conditions}
\end{table}

\subsection{Constant octagon value}
Similarly as in the case of $I$-graphs, we are interested in double generalized Petersen graphs with constant octagon value.
Here we also have some $8$-cycles with more than one existence condition, so it can happen that a double generalized Petersen graph admits two $8$-cycles of the same type. This happens only in the case of cycle $C_5$ and with the graph $\DP(8,2)$, which has a non-constant octagon value, since it admits also cycles $C^*$ and $C_2$. Therefore we can assume that an arbitrary double generalized Petersen graph admits at most one $8$-cycle of each form.

Note that $C^*$ always exists. Therefore, the only possibility to have a double generalized Petersen graph with a constant octagon value is in the case when a graph admits only $C^*, C_4$ and $C_5$.
The whole calculation is summarized in \cref{tab:DP}.
Analyzing obtained graphs, we determine that $\DP(10,3)$ and $\DP(10,2)$ are isomorphic (see \cref{thm:isomorphDP}) and that the octagon value of $\DP(6,1)$ and $\DP(6,2)$ is not constant 
(each of them admits one more cycle, $C_1$ or $C_5$).
Hence, there are just two $[1,\lambda,8]$-cycle regular double generalized Petersen graphs, with the value of $\lambda$ being $8$. They are depicted in \cref{fig:DP-graphs}.
\begin{table}[hb]
  \centering
  \resizebox{\ifbig 0.6 \fi \ifsmall 0.35 \fi \columnwidth}{!}{
  \begin{tabular}{c||c|c}
    \backslashbox{$C_4$}{$C_5$} & $2k + 4 = n$ & $2k - 4 = 0$ \\
    \hline \hline
    $4k + 2 = n$                                  & $\DP(6,1)$   & $\DP(10,2) $ \\
    \hline
    $2k + 1 = n$                                  & not exist    & $\DP(5,2) $  \\
    \hline
    $4k - 2 = n$                                  & $\DP(10,3)$  & $\DP(6,2)$   \\
  \end{tabular}}
  \caption{$\DP$ graphs containing $C_4, C_5$ and $C^*$.}
  \label{tab:DP}
\end{table}

It is worth mentioning that these $[1,\lambda,8]$-cycle regular graphs are edge transitive, in fact they are the only edge transitive double generalized Petersen graphs (see Kutnar and Petecki \cite{Kutnar/Petecki:2016}).
Notice, \cref{claim:DPisom} implies that the graph $\DP(5,2)$ is isomorphic to the Dodecahedral graph, i.e. $G(10,2)$. The other $[1,\lambda,8]$-cycle regular double generalized Petersen graph is not isomorphic to a generalized Petersen graph.

\begin{figure}[htb]
  \centering
  \begin{subfigure}[t]{\ifbig .29 \fi \ifsmall .22 \fi\textwidth}
    \centering
    \includegraphics[width=\linewidth]{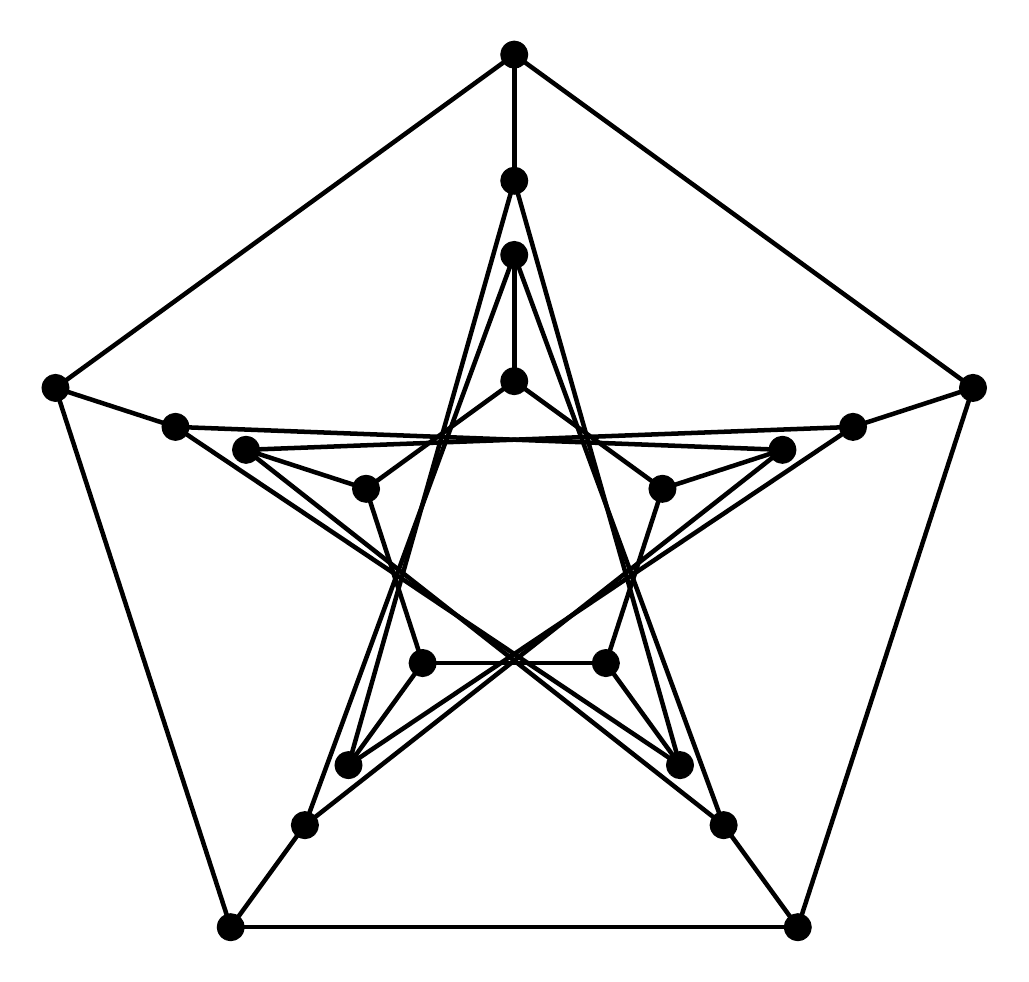}
    \caption{Dodecahedral graph.}
  \end{subfigure}
  \quad \quad
  \begin{subfigure}[t]{\ifbig .28 \fi \ifsmall .22 \fi\textwidth}
    \centering
    \includegraphics[width=\linewidth]{fig/DP10-2}
    \caption{$\DP(10,2)$.}
  \end{subfigure}
  \caption{All $[1,\lambda,8]$-cycle regular double generalized Petersen graphs.}
  \label{fig:DP-graphs}

\end{figure}

\section{Folded cubes \label{section-FodledCubes}}
An $n$-dimensional hypercube $Q_n$ is a graph on $2^{n}$ vertices, where vertices are represented with a binary string  $(x_1 x_2 \dots x_{n})$, while two vertices are adjacent whenever they differ in exactly one bit. 
For a binary value $x$ we use $\overline{x}$ to denote $1 - x$.
The $n$-dimensional folded cube $\FQ_n$ is constructed from an $n$-dimensional cube by identifying pairs of antipodal vertices, these are vertices which are exactly $n$ apart.
It can be formed also by adding \emph{complementary edges}, i.e. edges between $(x_1 x_2 \dots x_{n-1})$ and $(\overline{x}_1 \overline{x}_2 \dots \overline{x}_{n-1})$, to  the hypercube $Q_{n-1}$.
Edges of a folded cube are partitioned into two sets: set $D$ of complementary edges, also called \emph{diagonal edges} and the set $H$ of edges of a hypercube, called \emph{hypercube edges}.

Every hypercube edge of $\FQ_n$ gets a label $i$ from $ \{ 1, 2, \dots, n-1 \}$, where its value represents the bit at which adjacent vertices differ. Every diagonal edge is labeled with $d$.
With $d^*$ we denote the diagonal edge between vertices $(0, 0, \dots, 0)$ and $(1, 1, \dots, 1)$.
Cycles in this section are labeled using a sequence of edges, instead of vertices, as it was done in previous sections.

One should note that there are multiple edges with the same label, but they can never be incident, since each vertex is incident to exactly one edge of each label.
For example, see \cref{fig:4cycleFQ3}.
\begin{figure}[ht!]
    \centering
    \includegraphics[width= \ifbig 0.4 \fi \ifsmall 0.28 \fi \linewidth]{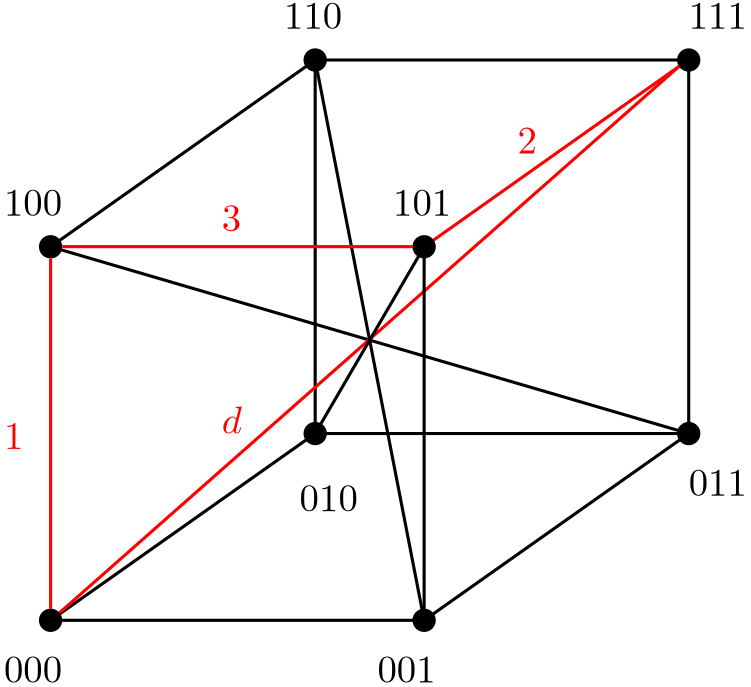}
    \caption{Example of a $4$-cycle $(1,3,2,d)$ in $\FQ_4$. \label{fig:4cycleFQ3}}
\end{figure}

\subsection{Cyclic structure of folded cubes}
Since folded cubes are distance-transitive, see for example \cite{Bon:2007}, they are arc-transitive, so they are $[1, \lambda, m]$-cycle regular for all $m$.
In this section we determine the value $\lambda$ in $[1, \lambda, 4]$ and $[1, \lambda, 6]$-cycle regularity of folded cubes.
It turns out that folded cubes are also $[2, \lambda, 6]$-cycle regular.
Also for this case we determine the value of $\lambda$.
In our study we rely on the following results.
\begin{remark}
\label{rem:FQ}
Let $C$ be an arbitrary cycle in a folded cube $\FQ_n$. Then all edge labels $\{1,2, \dots, n-1, d\}$ in $C$ appear in $C$ either all even or odd times. 
\end{remark}

\begin{thm}[Brouwer et.~al \cite{Brouwer/Cohen/Neumaier:1989}, Xu, Ma \cite{Ming/Meijie:2006}]
\label{thm:FQ-bipartite}
$\FQ_n$ is a bipartite graph if and only if $n$ is even.
\end{thm}

\begin{thm}[Xu, Ma \cite{Ming/Meijie:2006}]
Every edge of $\FQ_n$ lies on every cycle of even length from $4$ to $2^{n-1}$. Moreover, if $n$ is odd then every edge lies also on every cycle of odd length from $n$ to $2^{n-1}-1$.
\end{thm}

\subsection{\texorpdfstring{\ifbig \boldmath \fi$4$}{4}-cycles in folded cubes}
First we focus our study on $4$-cycles.
There are already some results known in this area.

\begin{thm}[Mirafzal \cite{Mirafzal:2016}]
Folded cube $\FQ_5$ is $[2, 3, 4]$-cycle regular. Any other folded cube $\FQ_n$ is $[2, 1, 4]$-cycle regular.
\end{thm}

We settle the following result on $[1, \lambda, 4]$-cycle regularity of folded cubes.

\thmFQncycles*

\begin{proof}
Since $\FQ_1 \simeq K_1$, $\FQ_2 \simeq K_2$ and $\FQ_4 \simeq K_{4,4}$, the claim trivially holds for $n\in \{1, 2, 4\}$.
So assume $n \notin \{1, 2, 4\}$ and observe that two diagonal edges cannot be consecutive, therefore an arbitrary $4$-cycle admits at most two of them.
In fact, with the exception of $\FQ_4$, $4$-cycles in all folded cubes admit either $2$ or $0$ diagonal edges (see \cref{rem:FQ}).

To prove that an $\FQ_n$ is $[1,n-1,4]$-cycle regular, when $n \notin \{2, 4\}$, we need to take an arbitrary edge, denote it $e$, and check in how many $4$-cycles it lies. We have two possibilities, either $e \in D$ or $e \in H$.
In the case $e \in D$, the $4$-cycle consists of an additional diagonal edge and two hypercube edges.
Let us label $e$ with $d$ and an arbitrary edge from $H$ with $i$.
There are $n-1$ such cycles and they are of the form $(d,i,d,i)$.
In the case when $e \in H$, the $4$-cycle admits also either $2$ or $0$ diagonal edges. Let $d$ be an edge from $D$ and $i$ an edge from $H$.
If there are $2$ diagonal edges the cycle is uniquely determined and is of form $(e,d,e,d)$. If there are no diagonal edges the cycle is of form $(e,i,e,i)$. Since there are $n-2$ choices for $i$, there are $n-2$ such cycles.
\end{proof}

\subsection{\texorpdfstring{\ifbig \boldmath \fi$6$}{6}-cycles in folded cubes}

In this section we settle $[1,\lambda,6]$ and $[2, \lambda, 6]$-cycle regularity of folded cubes.
Before we start listing $6$-cycles we need the following definition.
\begin{defin}
Let $\FQ_n$ be a folded cube of dimension $n$ and let $F_n = \{1, 2, \dots, n-1\} \cup \{ d\}$.
Two $6$-cycles, $C_1 = (a_1, a_2, a_3, a_4, a_5, a_6)$ and $C_2 = (b_1, b_2, b_3, b_4, b_5, b_6)$, for $a_i, b_i \in F_n$ ($1 \leq i \leq 6$), are said to be equivalent,
if there exists a permutation $f : F_n \rightarrow F_n$, where $f(d) = d$ and $f(C_1) = C_2$, where $$f(C_1) := (f(a_1), f(a_2), f(a_3), f(a_4), f(a_5), f(a_6)) .$$
\end{defin}
For better understanding of the definition see the following example.
\begin{example}
Let $G = \FQ_7$ and let
$ C_1=(1,2,3,1,2,3),  C_2 = (4,5,6,4,5,6), C_3 = (4,d,6,4,d,6),$
be $6$-cycles in $G$.\\
Cycles $C_1, C_2$ are equivalent, since for $f: F_7 \rightarrow F_7$ given by
\begin{equation*}
f = 
\begin{pmatrix}
1 & 2 & 3 & 4 & 5 &6 & d \\
4 & 5 & 6 & 2 & 3 & 1 & d \\
\end{pmatrix}
\end{equation*}
holds $C_1(f) = C_2$.
However, $C_3$ is not equivalent to any of the given $6$-cycles, because there does not exist such a permutation on $F_7$ that would map $2$ or $5$ to $d$.
\end{example}


\subsection{Determining \texorpdfstring{\ifbig \boldmath \fi$[1,\lambda,6]$}{[1,lambda,6]}-cycle regularity}
To determine in how many $6$-cycles an edge lies, 
we need to fix a hypercube or diagonal edge and count the number of cycles going through it. 
By studying both possibilities we not only determine the $\lambda$ value, but also give a full characterization of all possible $6$-cycles in folded cubes.
\subsubsection{Fixing hypercube edge in a \texorpdfstring{\ifbig \boldmath \fi $6$}{6}-cycle}
In this part we write all $6$-cycles where the first edge is a hypercube edge.
Because of the equivalence of $6$-cycles, we can write them starting with the edge labeled $1$.
We distinguish these cycles by the number of diagonal edges they admit. Since two diagonal edges are never consecutive, a $6$-cycle can admit either $0, 1, 2$ or $3$ of them.
Throughout this section, $n$ is the dimension of the folded cube $\FQ_n$, diagonal edge is denoted by $d$ and $i,j,k,l > 1$ are distinct, positive integers, smaller than $n$.
All of the calculation is summarized in \cref{tab:FQ6-cycles1} and \cref{fig:cylesFQ_n-starting1}.

\begin{figure}[ht!]
  \centering
  \begin{subfigure}{.24\textwidth}
    \centering
    \includegraphics[width=\linewidth]{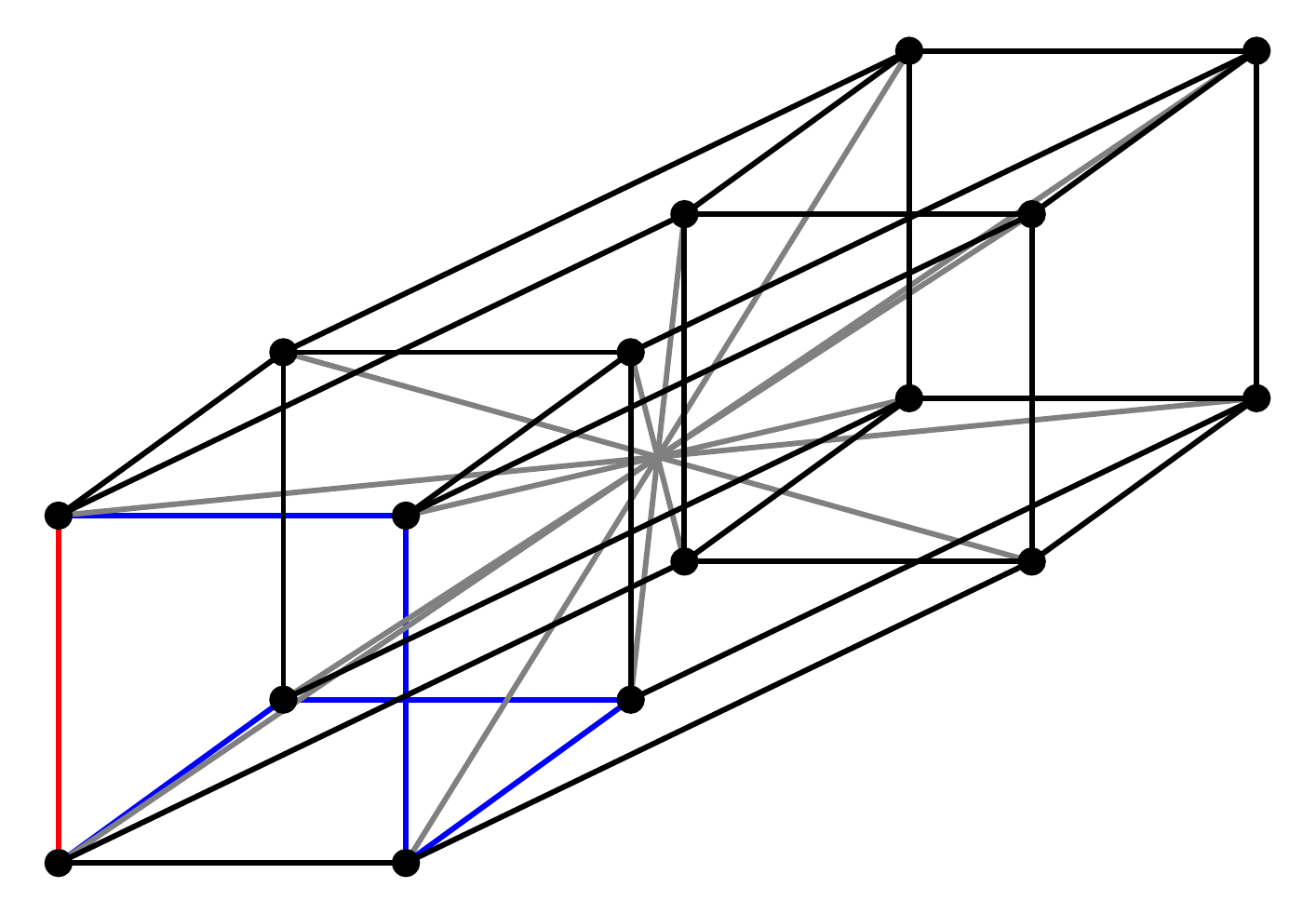}
    \caption{Cycle $C_0$.}
    \label{fig:FQnC00}
  \end{subfigure}
  \begin{subfigure}{.24\textwidth}
    \centering
    \includegraphics[width=\linewidth]{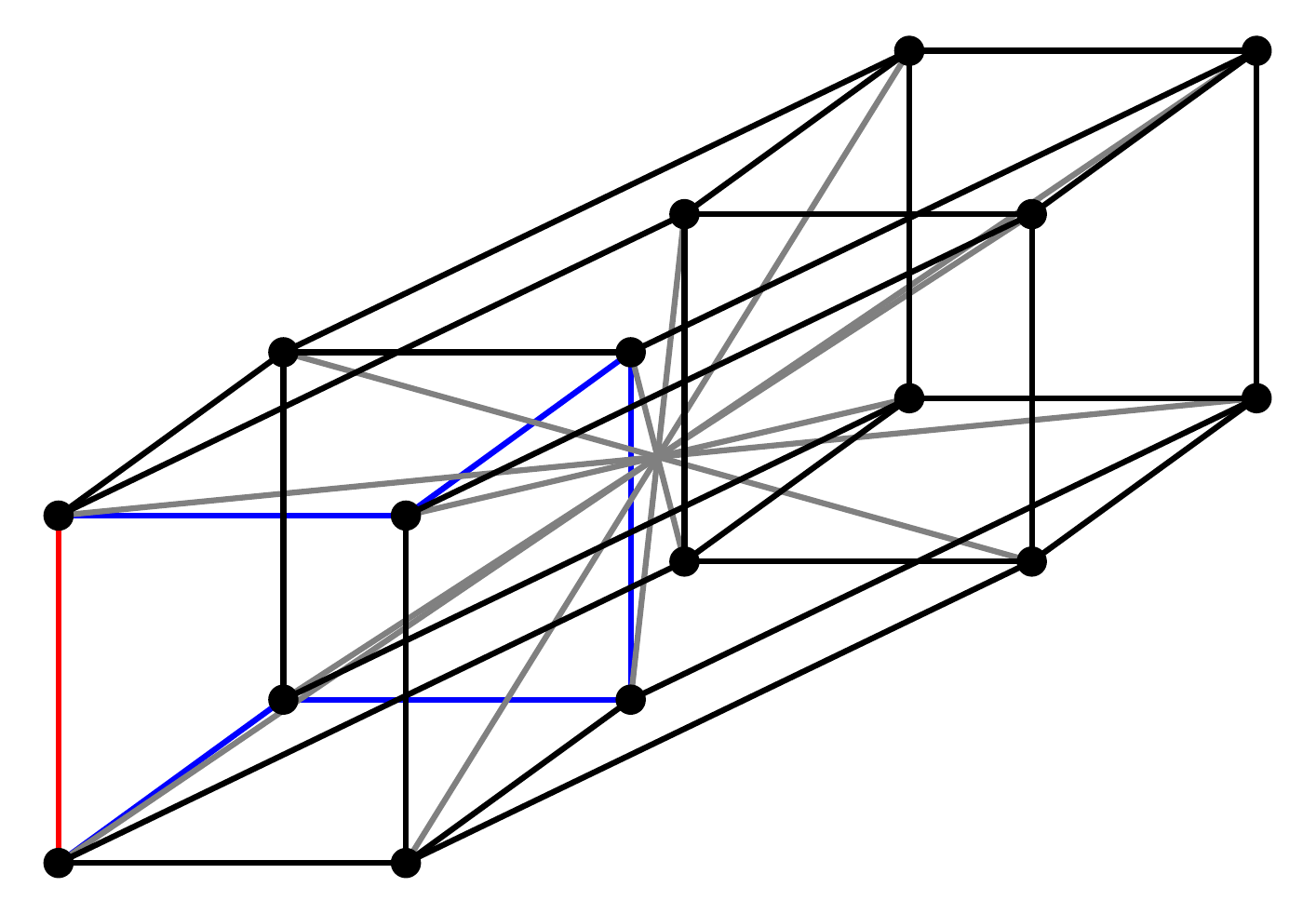}
    \caption{Cycle $C_1$.}
    \label{fig:FQnC0}
  \end{subfigure}
  \begin{subfigure}{.24\textwidth}
    \centering
    \includegraphics[width=\linewidth]{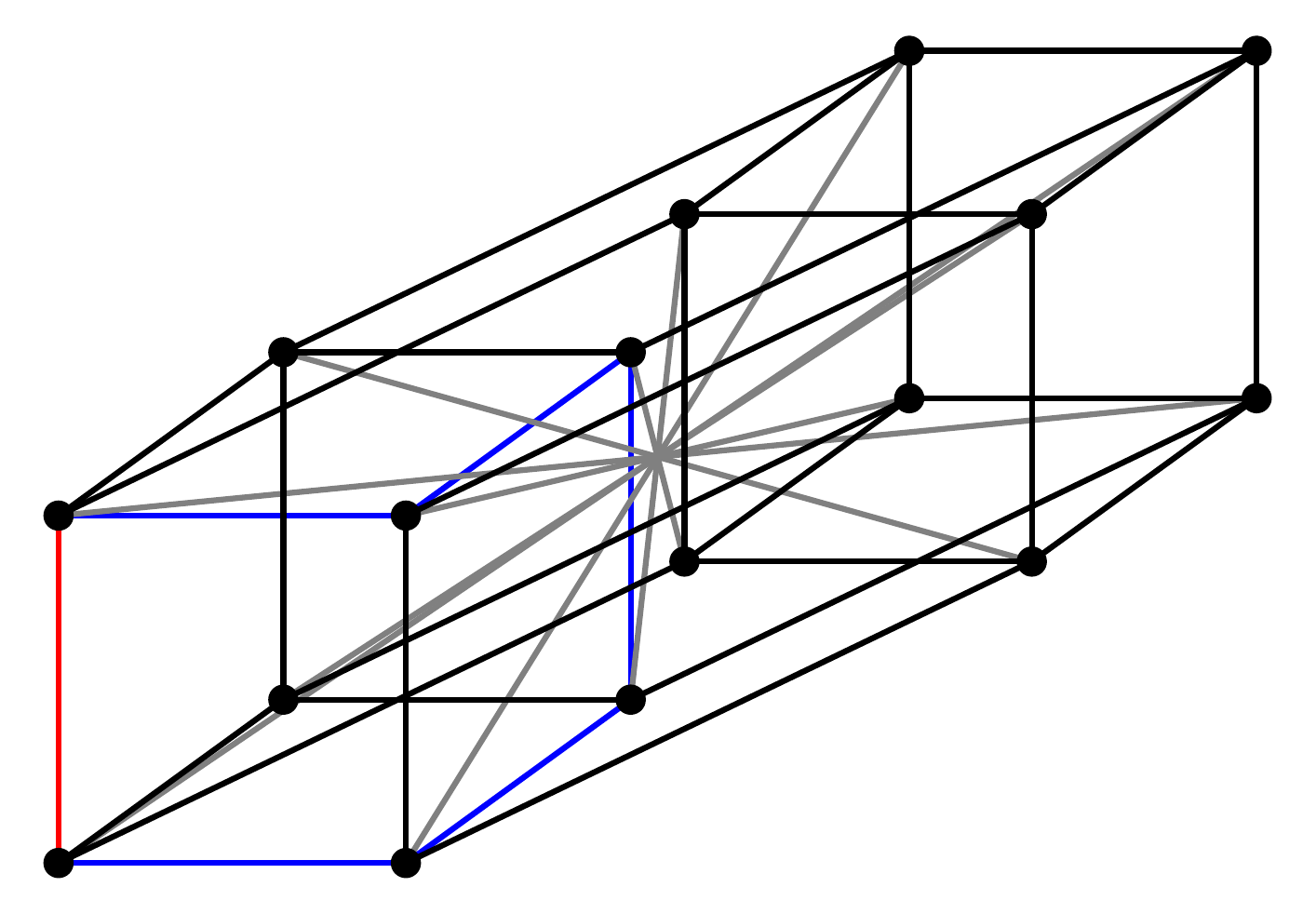}
    \caption{Cycle $C_2$.}
    \label{fig:FQnC1}
  \end{subfigure}
  \begin{subfigure}{.24\textwidth}
    \centering
    \includegraphics[width=\linewidth]{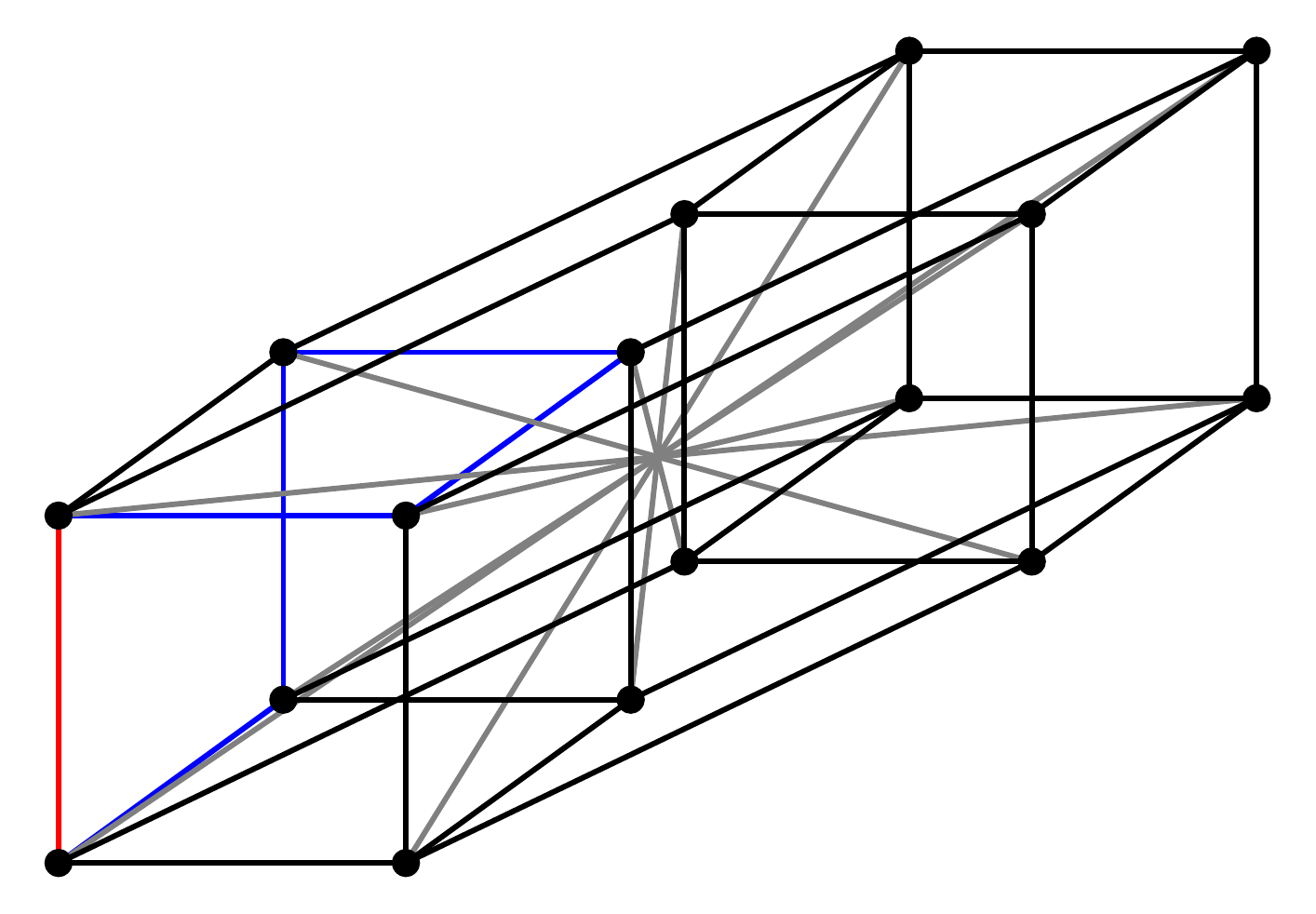}
    \caption{Cycle $C_3$.}
    \label{fig:FQnC03}
  \end{subfigure}
  
  \begin{subfigure}{.24\textwidth}
    \centering
    \includegraphics[width=\linewidth]{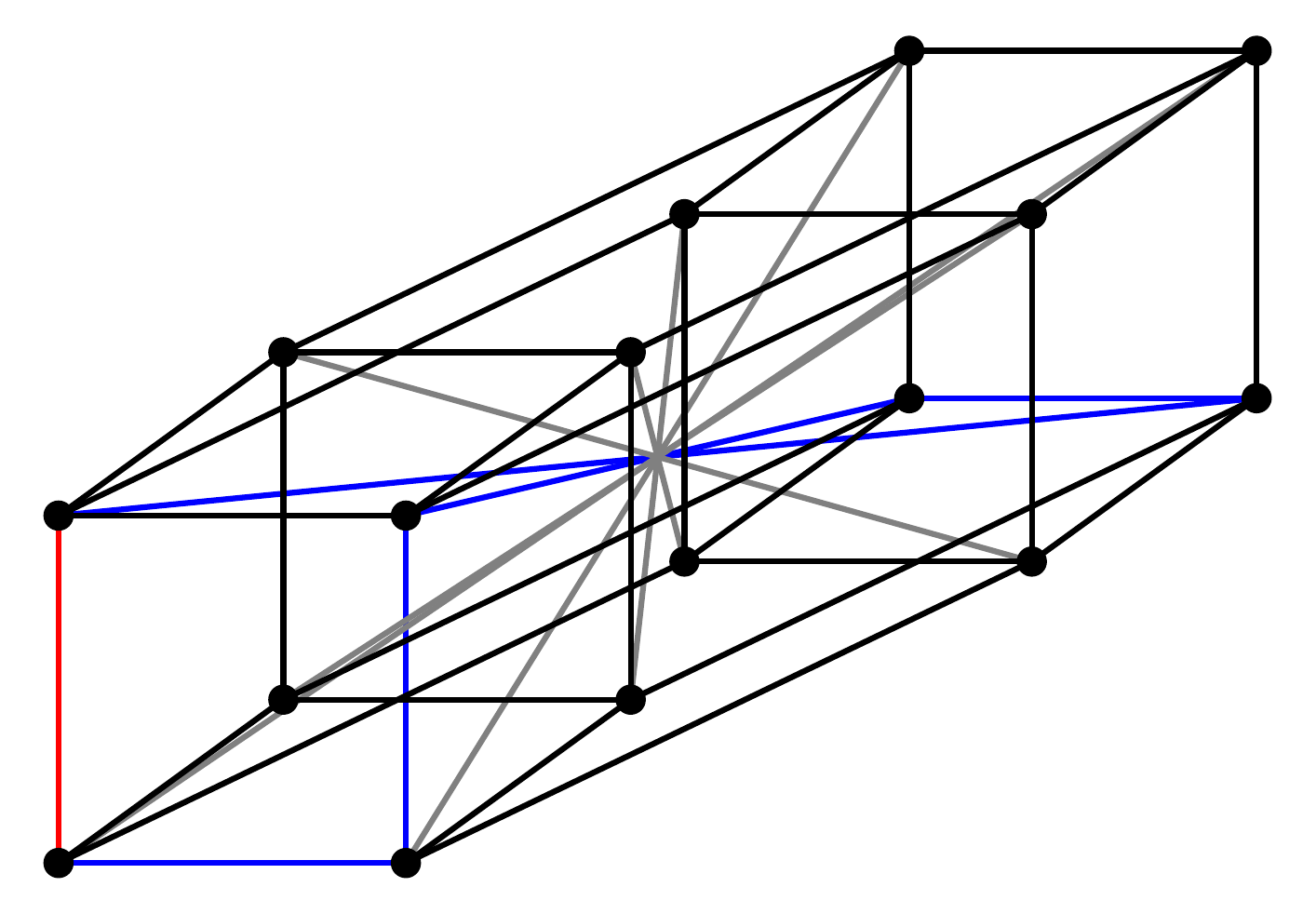}
    \caption{Cycle $C_4$.}
    \label{fig:FQnC2}
  \end{subfigure}
   \begin{subfigure}{.24\textwidth}
    \centering
    \includegraphics[width=\linewidth]{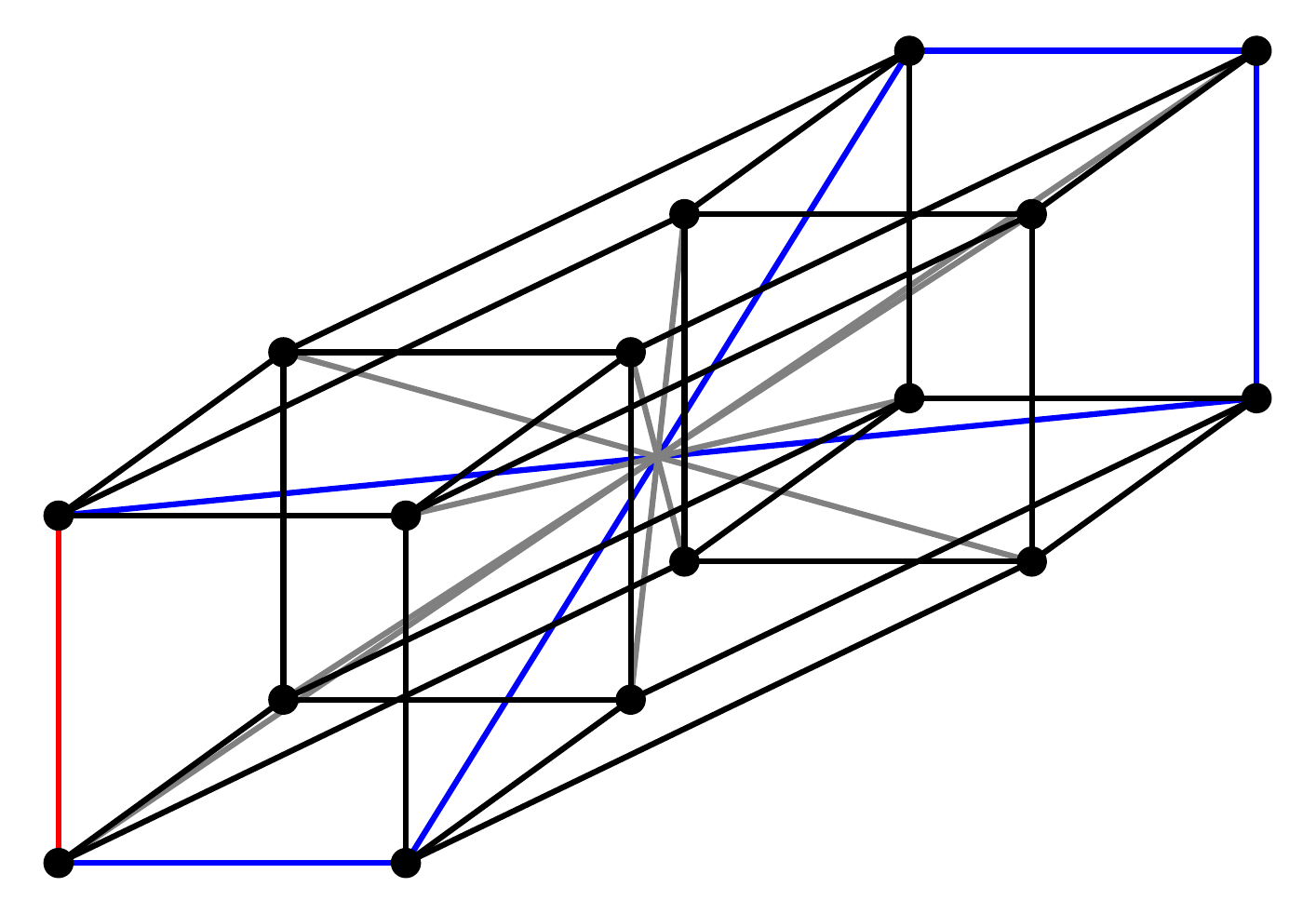}
    \caption{Cycle $C_5$.}
    \label{fig:FQnC3}
  \end{subfigure}
  \begin{subfigure}{.24\textwidth}
    \centering
    \includegraphics[width=\linewidth]{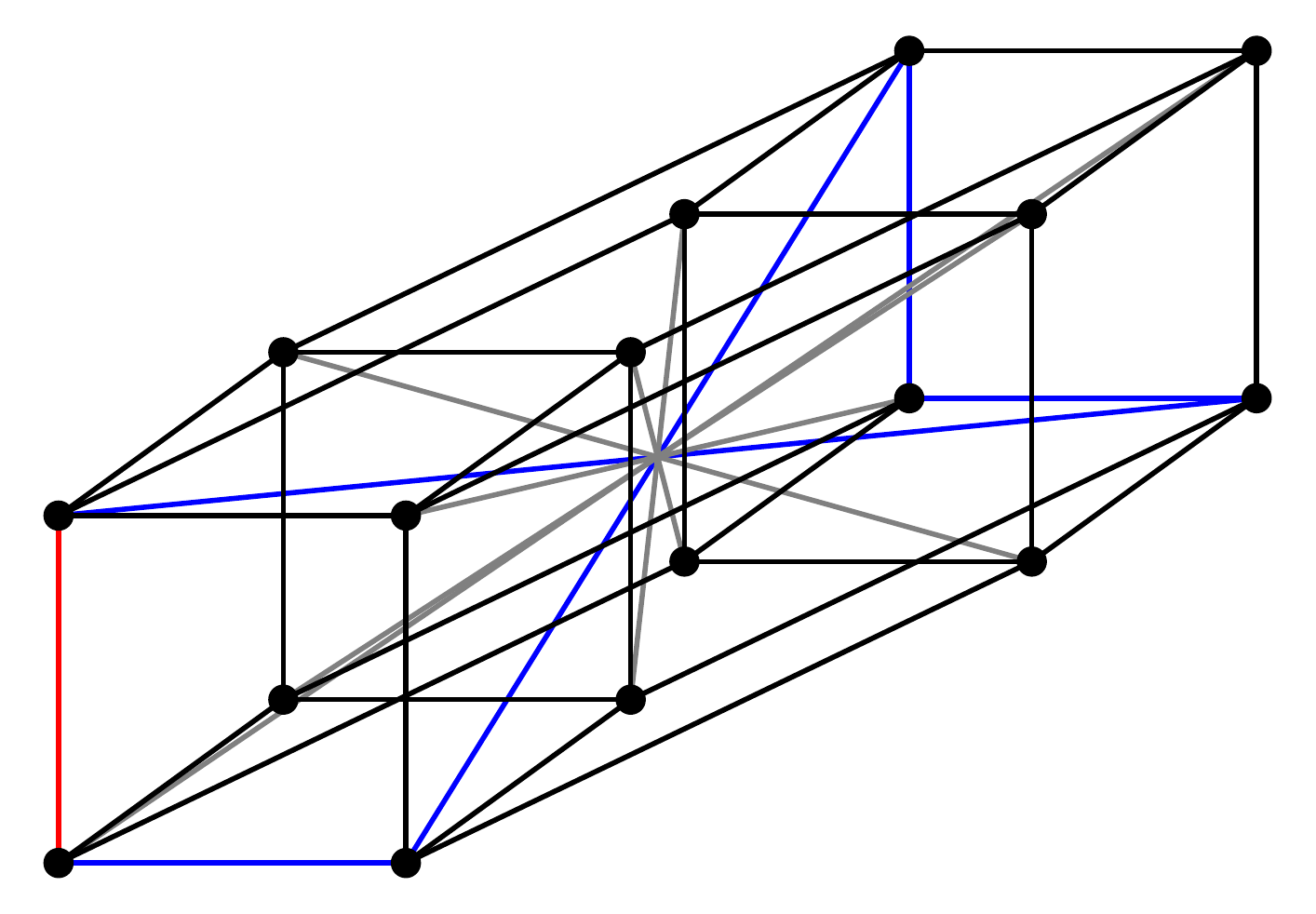}
    \caption{Cycle $C_6$.}
    \label{fig:FQnC4}
  \end{subfigure}
  \begin{subfigure}{.24\textwidth}
    \centering
    \includegraphics[width=\linewidth]{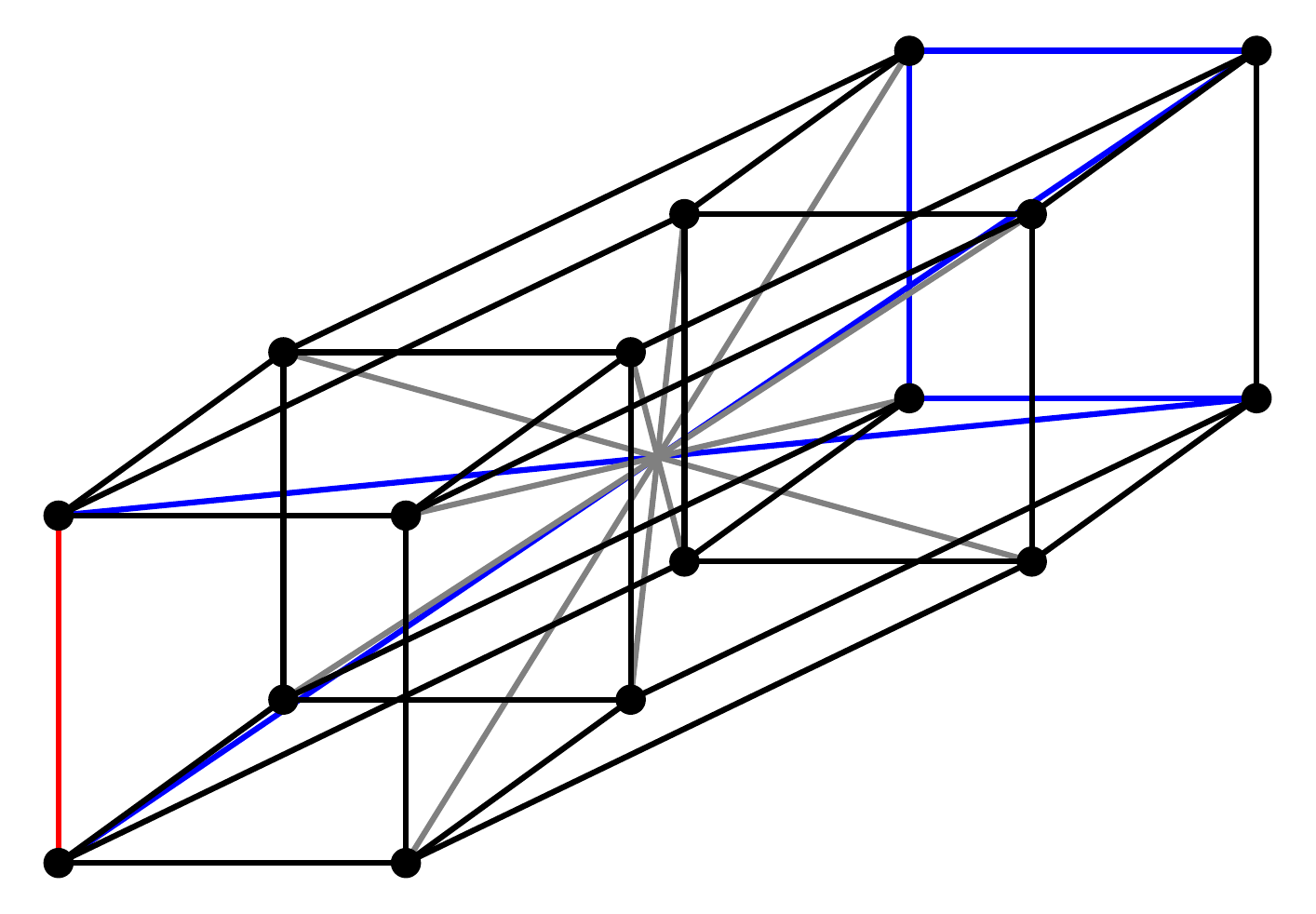}
    \caption{Cycle $C_7$.}
    \label{fig:FQnC5}
  \end{subfigure}
  
   \begin{subfigure}{.24\textwidth}
    \centering
    \includegraphics[width=\linewidth]{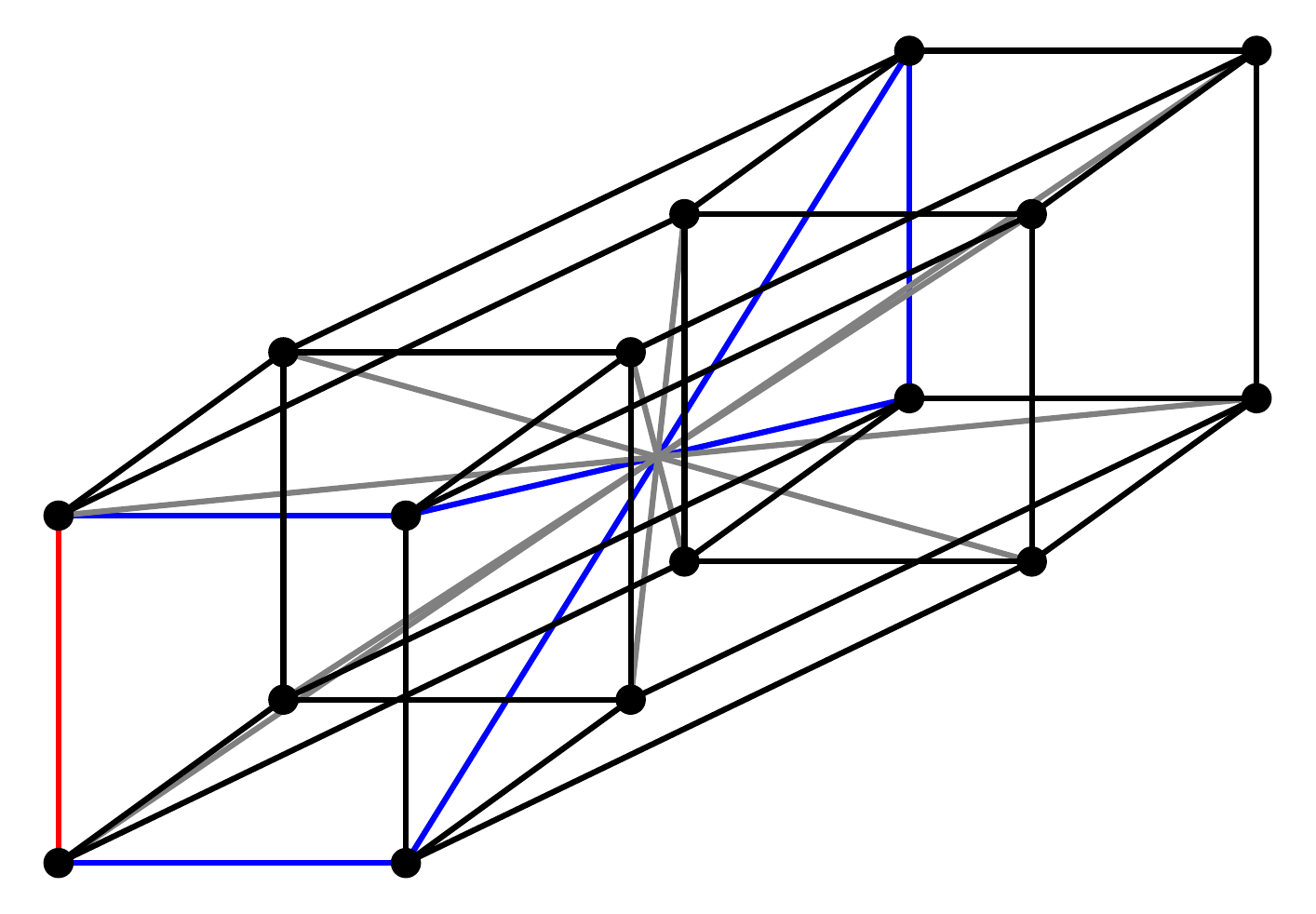}
    \caption{Cycle $C_8$.}
    \label{fig:FQnC6}
  \end{subfigure}
  \begin{subfigure}{.24\textwidth}
    \centering
    \includegraphics[width=\linewidth]{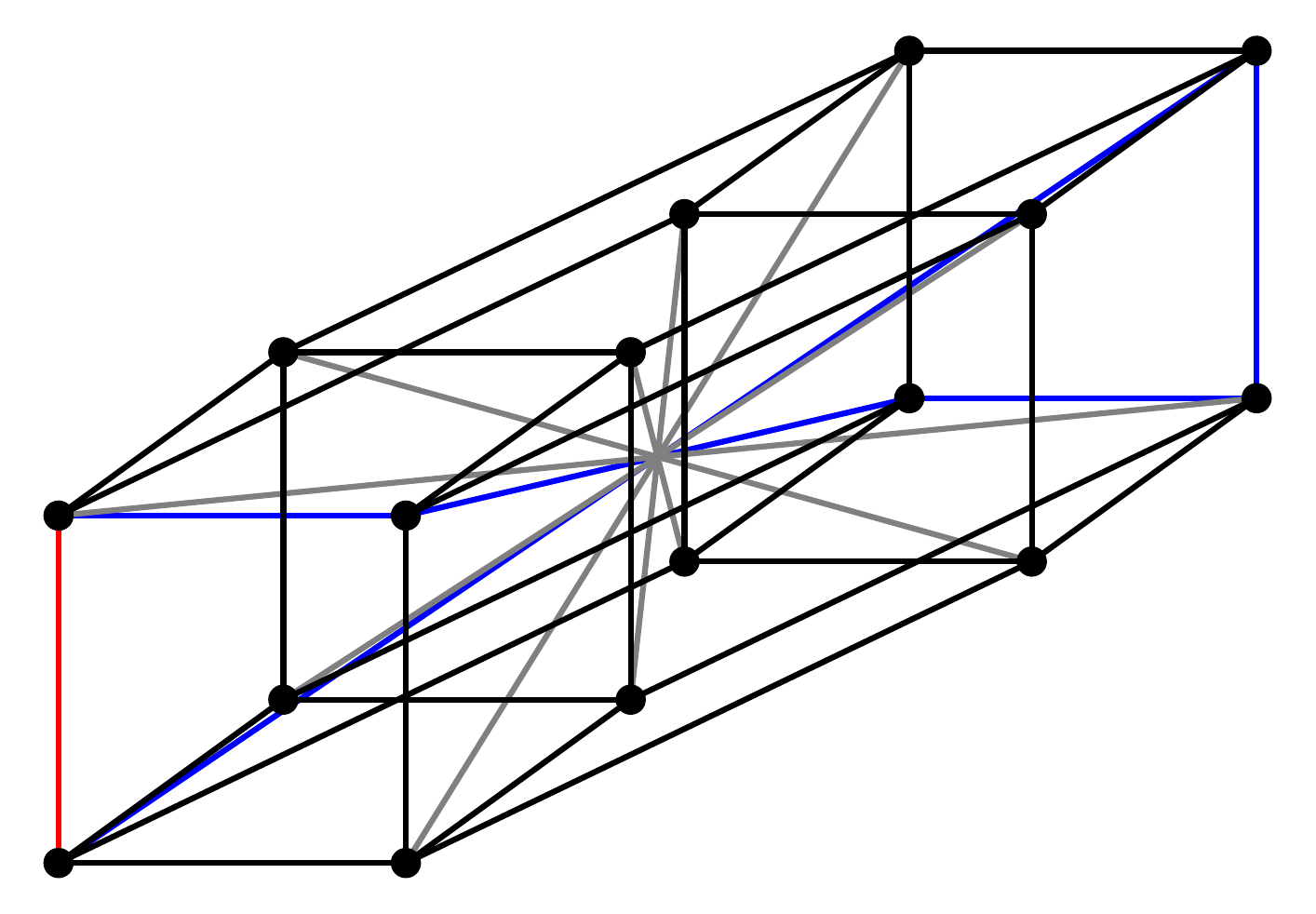}
    \caption{Cycle $C_9$.}
    \label{fig:FQnC7}
  \end{subfigure}
  \begin{subfigure}{.24\textwidth}
    \centering
    \includegraphics[width=\linewidth]{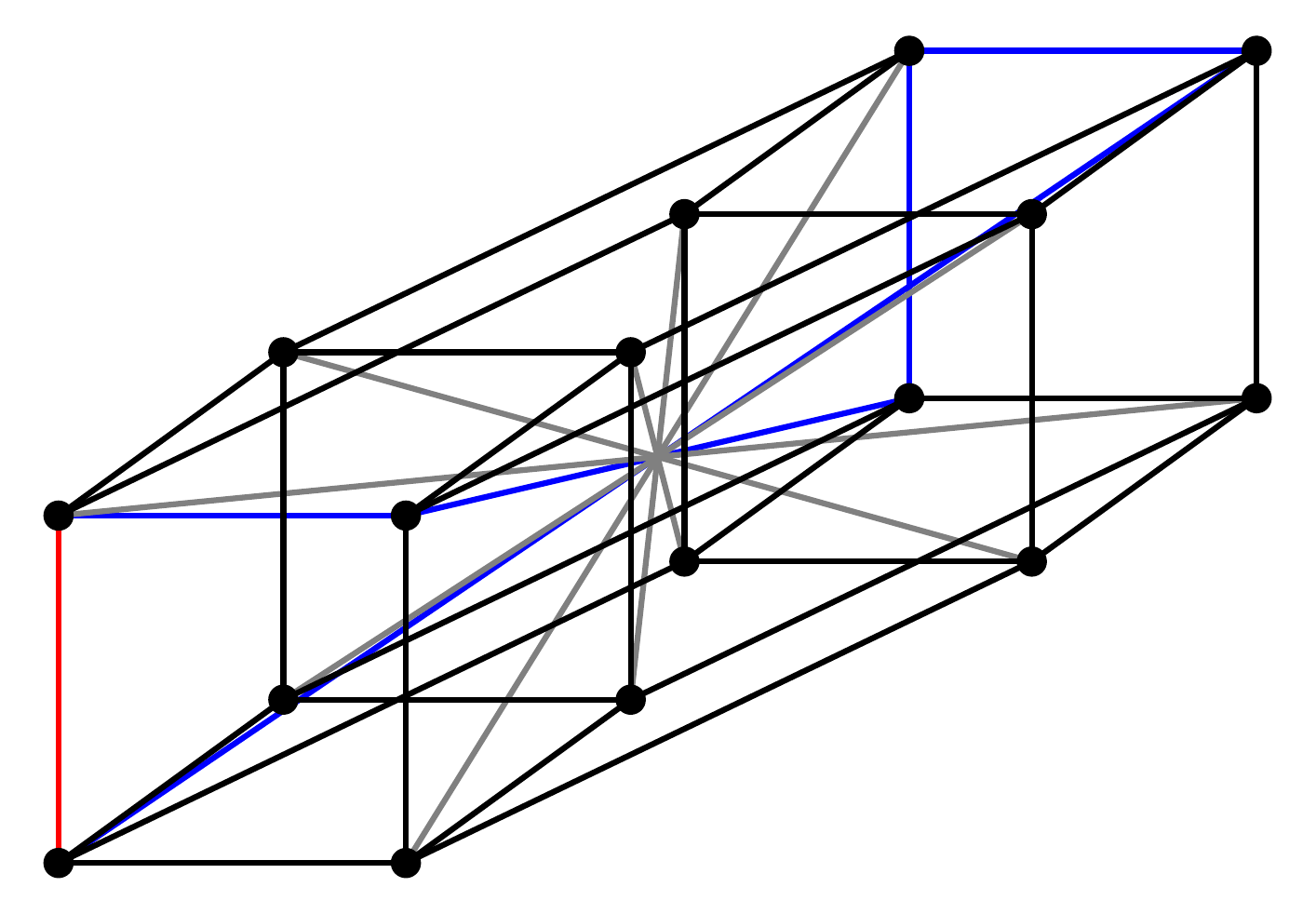}
    \caption{Cycle $C_{10}$.}
    \label{fig:FQnC8}
  \end{subfigure}
   \begin{subfigure}{.24\textwidth}
    \centering
    \includegraphics[width=\linewidth]{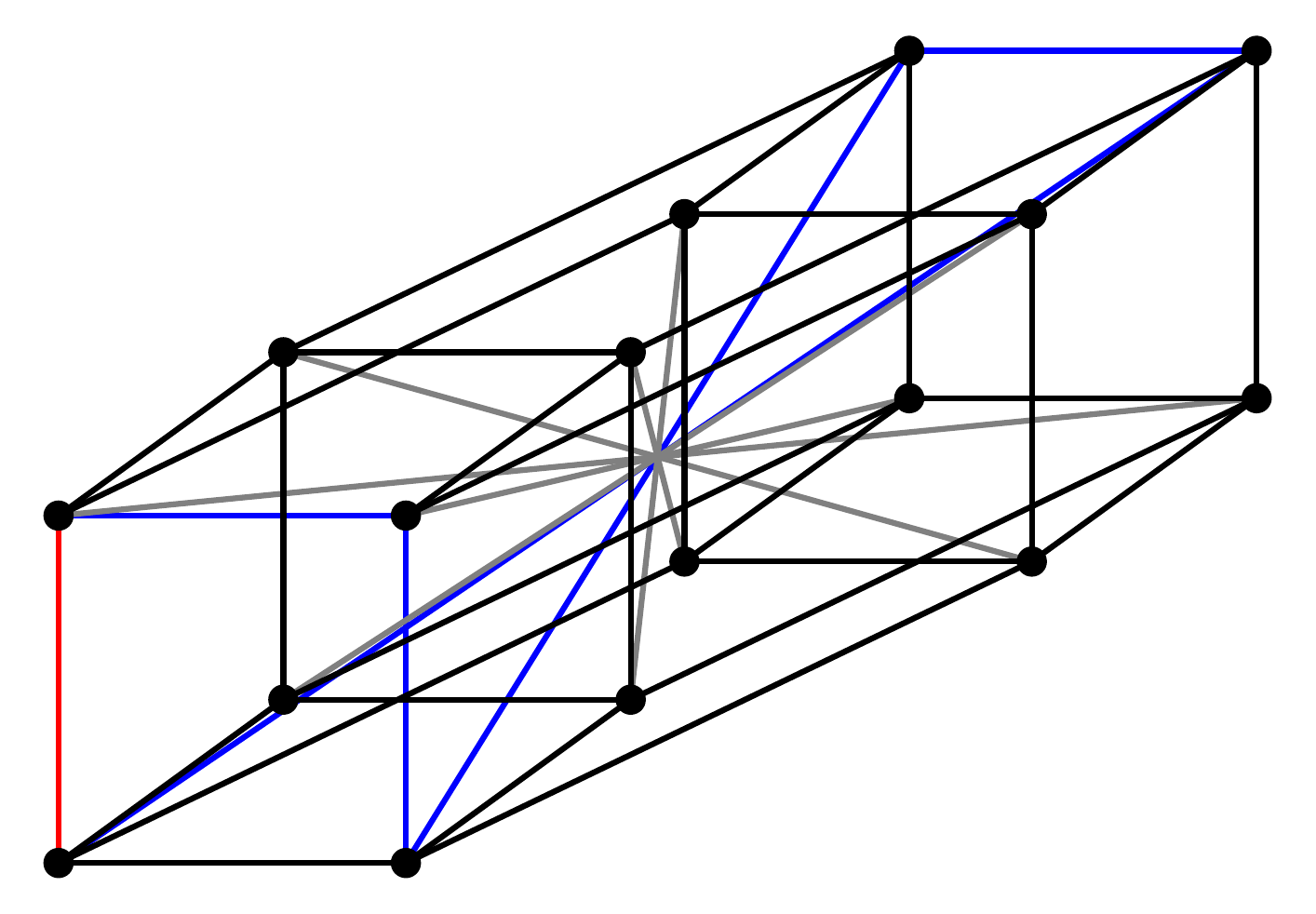}
    \caption{Cycle $C_{11}$.}
    \label{fig:FQnC9}
  \end{subfigure}
  \caption{Examples of all non-equivalent $6$-cycles in $\FQ_5$ starting with edge $1$. \label{fig:cylesFQ_n-starting1}}
\end{figure}
\textbf{If a cycle admits no diagonal edges} it has to be of one of the following forms:
\begin{equation*}
\begin{aligned}[t]
C_0 &= (1, i, 1, j, i, j) \quad \text{or} \\
C_2 &= (1, i, j, 1, j, i) \quad \text{or}
\end{aligned}
\qquad
\begin{aligned}[t]
C_1 &= (1, i, j, 1, i, j) \quad \text{or} \\
C_3 &= (1, i, j, i, 1, j).
\end{aligned}
\end{equation*}
The contribution of each cycle to the $\lambda$ value of a graph is $(n-2)(n-3)$.

\textbf{If a cycle admits only one diagonal edge} it has one of the five following forms:
\begin{equation*}
\begin{aligned}[t]
D_1 &= (1,d,i,j,k,l) \quad \text{or} \\
D_3 &= (1,i,j,d,k,l) \quad \text{or} \\
D_5 &= (1,i,j,k,l,d).
\end{aligned}
\qquad
\begin{aligned}[t]
D_2 &= (1,i,d,j,k,l) \quad \text{or} \\
D_4 &= (1,i,j,k,d,l) \quad \text{or} \\
\end{aligned}
\end{equation*}
These cycles exists only in the case of $\FQ_6$ and contribute $24$ to the $\lambda$ value of a graph.

\textbf{If a cycle admits two diagonal edges} it has to be of one of the following forms:
\begin{equation*}
\begin{aligned}[t]
C_4 &= (1, d, i, d, 1, i) \quad \text{or} \\
C_6 &= (1, d, i, 1, d, i) \quad \text{or} \\
C_8 &= (1, i, d, 1, d, i) \quad \text{or} \\
C_{10} &= (1, i, d, 1, i, d) \quad \text{or} \\
\end{aligned}
\qquad
\begin{aligned}[t]
C_5 &= (1, d, 1, i, d, i) \quad \text{or} \\
C_7 &= (1, d, i, 1, i, d) \quad \text{or} \\
C_9 &= (1, i, d, i, 1, d) \quad \text{or} \\
C_{11} &= (1, i, 1, d, i, d).
\end{aligned}
\end{equation*}
Each of these cycles contributes to the $\lambda$ value of a graph exactly $n-2$.

\textbf{If a cycle admits three diagonal edges} it is of the form $$D_6 = (1,d,i,d,j,d).$$
This cycle exists only in the case of $\FQ_4$ and contributes $2$ to the $\lambda$ value of a graph.

\begin{table}[ht!]
  \centering
  \resizebox{\ifbig 0.8 \fi \ifsmall 0.6 \fi \columnwidth}{!}{
  \begin{tabular}{c|c|c}
    \textbf{Label}         & \textbf{A representative of a \boldmath$6$-cycle}                                                                & \textbf{Contribution towards \boldmath$\lambda$} \\
    \hline
    $C_0$                  & $(1, i, 1, j, i, j) $ & $(n-2)(n-3)$ \\
    \hline
    $C_1$                  & $(1, i, j, 1, i, j)$ & $(n-2)(n-3)$ \\
    \hline
    $C_2$                  & $(1, i, j, 1, j, i)$ & $(n-2)(n-3)$ \\
    \hline
    $C_3$                  & $(1, i, j, i, 1, j)$ & $(n-2)(n-3)$ \\
    \hline
    $C_4$                 & $(1, d, i, d, 1, i)$ & $n-2$ \\
    \hline
    $C_5$                  & $(1, d, 1, i, d, i)$  & $n-2$ \\
    \hline
    $C_6$                  & $(1, d, i, 1, d, i)$ & $n-2$ \\
    \hline
    $C_7$                  & $(1, d, i, 1, i, d)$ & $n-2$ \\
    \hline
    $ C_8$                 & $(1, i, d, 1, d, i)$  & $n-2$ \\
    \hline
    $C_9$                  & $(1, i, d, i, 1, d)$  & $n-2$ \\
        \hline
    $C_{10}$                 & $(1, i, d, 1, i, d)$ & $n-2$ \\
    \hline
    $C_{11}$                  & $(1, i, 1, d, i, d)$ & $n-2$ \\
    
  \end{tabular}}
  \caption{All non-equivalent $6$-cycles in $\FQ_n$ ($n \notin \{1,2,3,4,6\}$) starting with hypercube edge $1$.}
  \label{tab:FQ6-cycles1}
\end{table}

\subsubsection{Fixing diagonal edge in a \texorpdfstring{\ifbig \boldmath \fi $6$}{6}-cycle}
Again we distinguish $6$-cycles by the number of diagonal edges they admit. Since these cycles start with $d^*$ they can admit either $1, 2$ or  $3$ diagonal edges.
All of the calculation is summarized in \cref{tab:FQ6-cycles2} and \cref{fig:cylesFQ_n-startingd}.

\begin{figure}[ht!]
  \centering
  \begin{subfigure}{.24\textwidth}
    \centering
    \includegraphics[width=\linewidth]{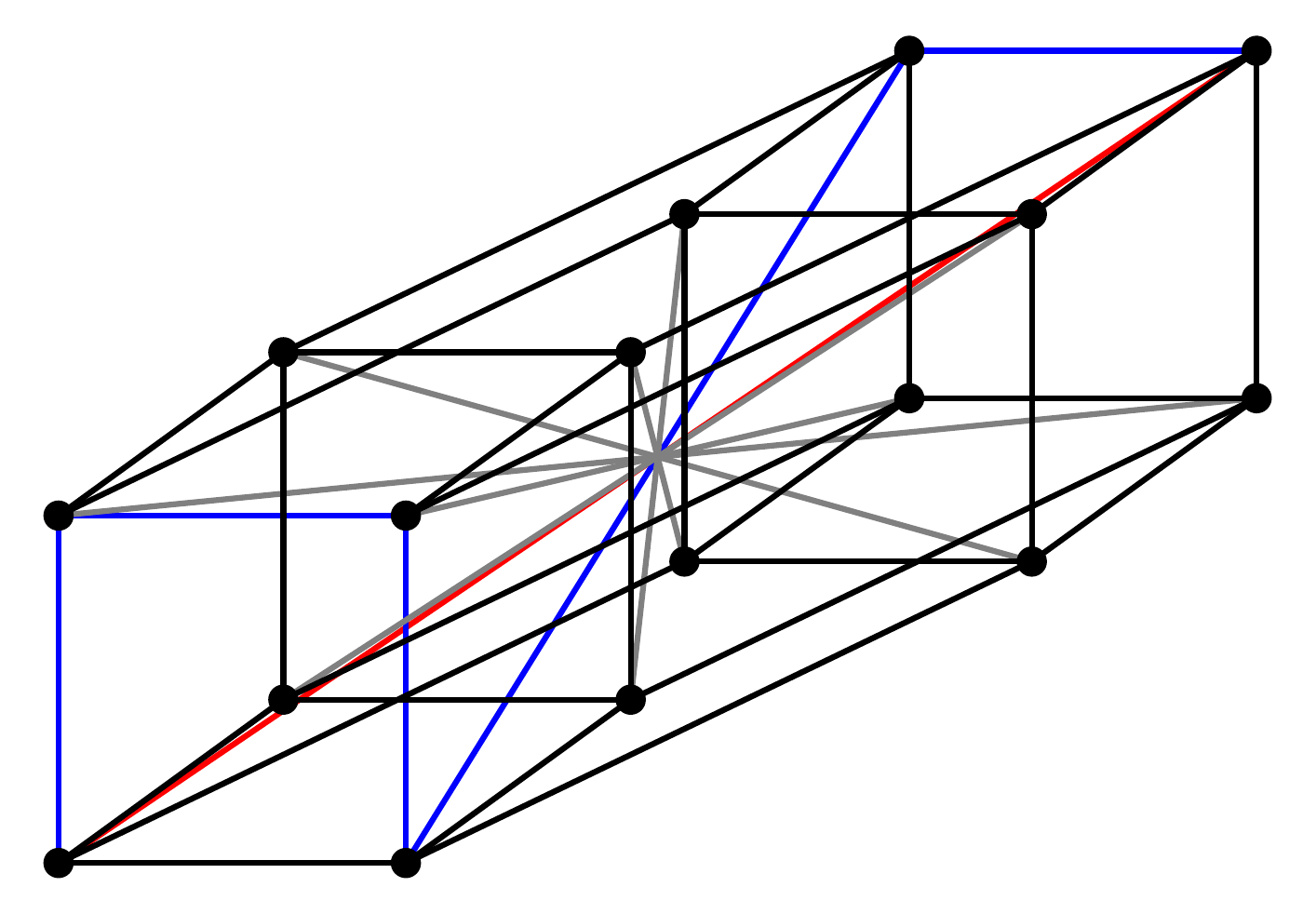}
    \caption{Cycle $C_a$.}
    \label{fig:FQnCa}
  \end{subfigure}
  \begin{subfigure}{.24\textwidth}
    \centering
    \includegraphics[width=\linewidth]{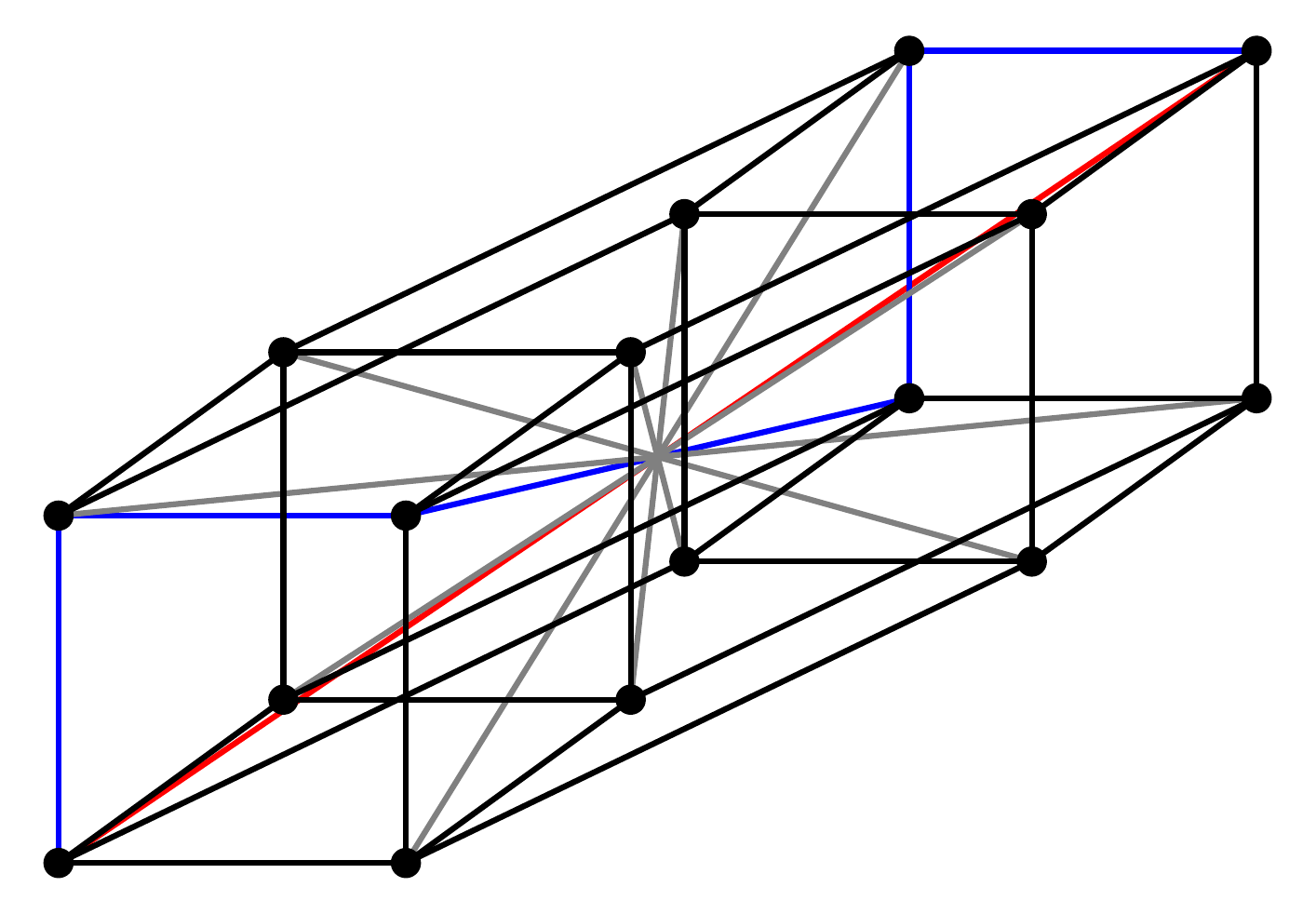}
    \caption{Cycle $C_b$.}
    \label{fig:FQnCb}
  \end{subfigure}
  \begin{subfigure}{.24\textwidth}
    \centering
    \includegraphics[width=\linewidth]{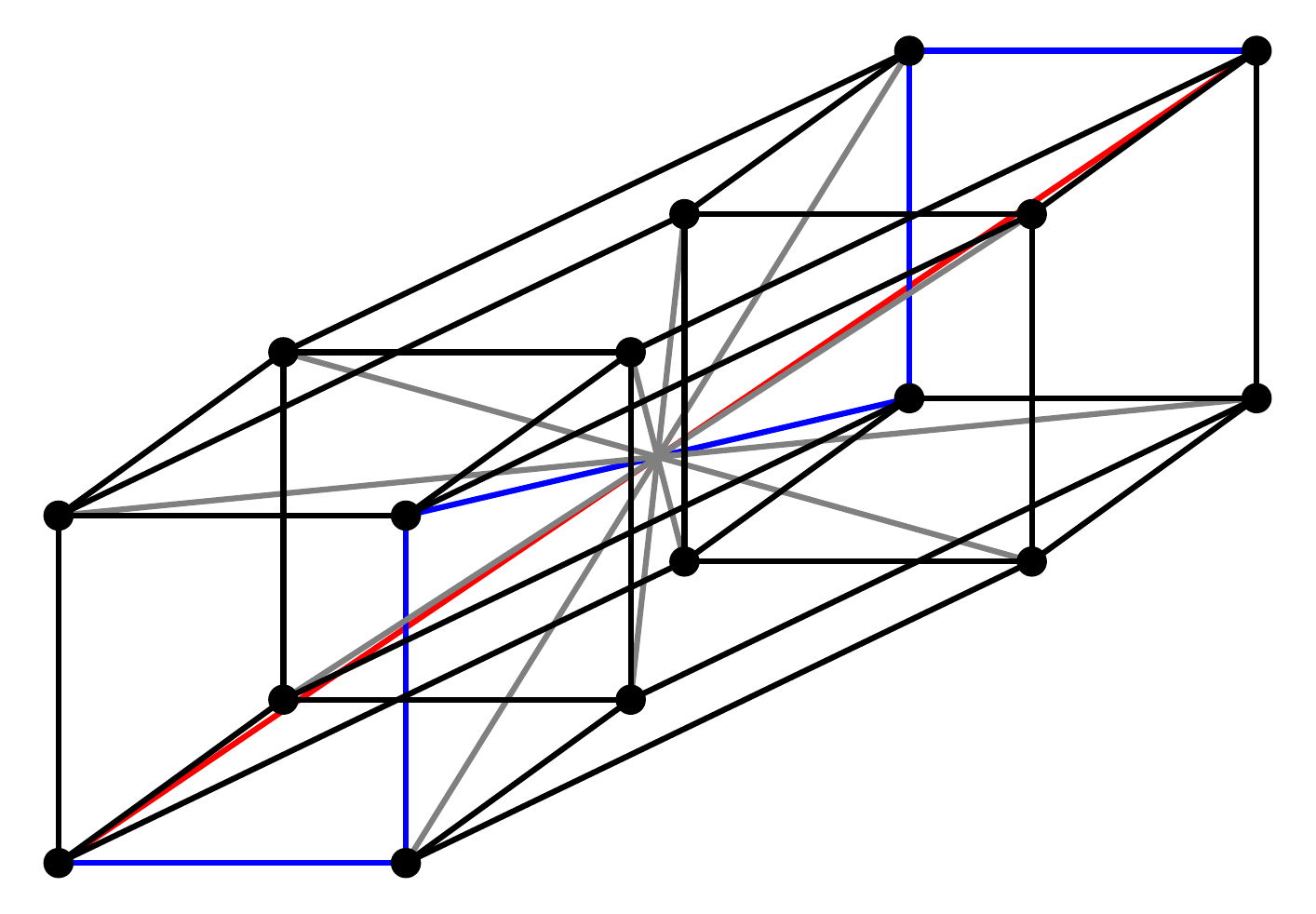}
    \caption{Cycle $C_c$.}
    \label{fig:FQnCc}
  \end{subfigure}
   \begin{subfigure}{.24\textwidth}
    \centering
    \includegraphics[width=\linewidth]{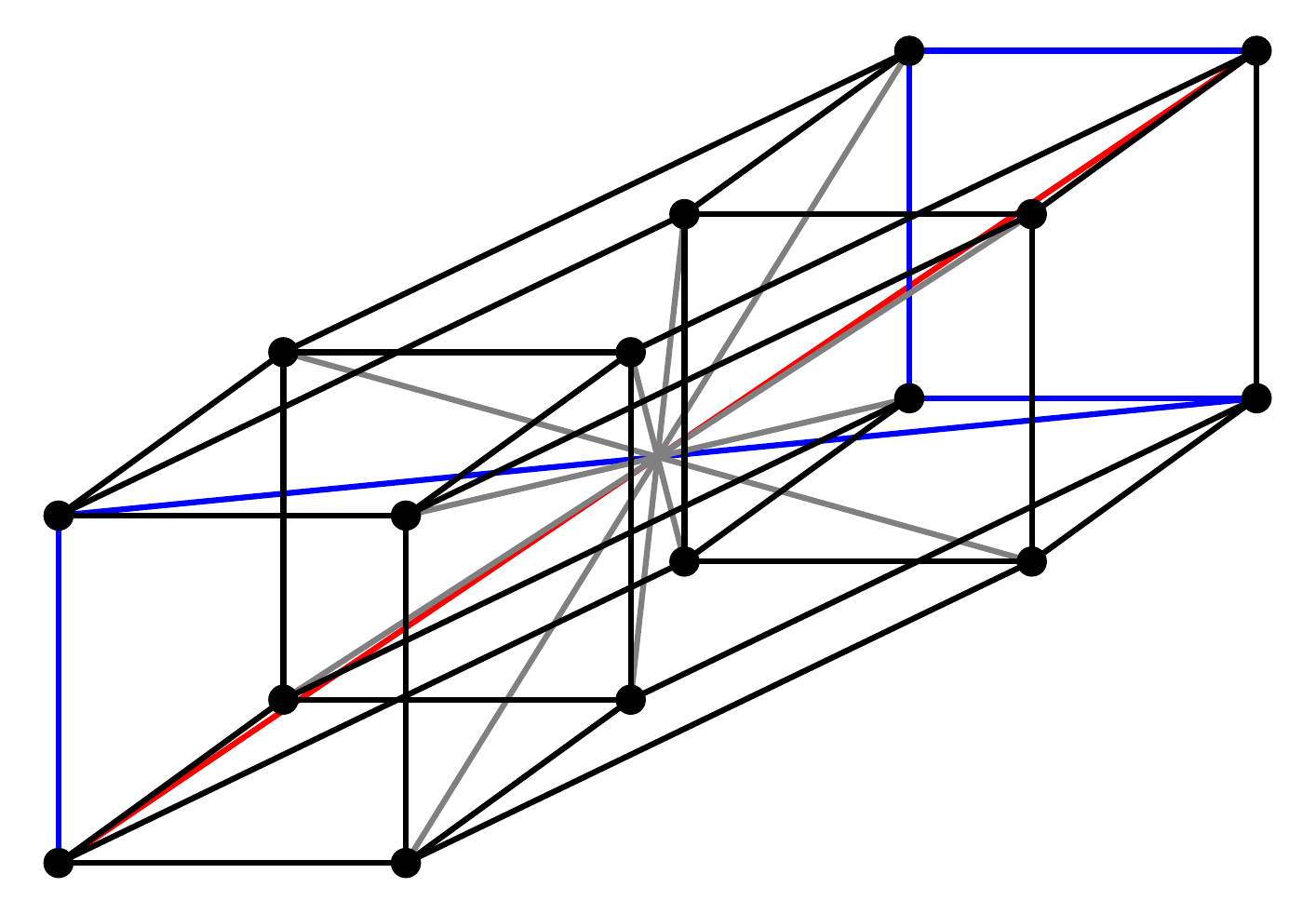}
    \caption{Cycle $C_d$.}
    \label{fig:FQnCd}
  \end{subfigure}

  \caption{Examples of all non-equivalent $6$-cycles in $\FQ_5$ starting with edge $d^*$. \label{fig:cylesFQ_n-startingd}}
\end{figure}

\textbf{If a cycle admits only one diagonal edge} it exists only in the case of $\FQ_6$ and is of the following form:
$$ D_a= (d^* \! \!, i,j,k,l,m),$$
where $1\leq m < n$ and $m \neq i,j,k,l$. It contributes $120$ to the $\lambda$ value of a graph.

\textbf{If a cycle admits two diagonal edges} it has one of the following forms:
\begin{equation*}
\begin{aligned}[t]
C_{a} &= (d^* \! \!, i, d, j, i, j) \quad \text{or} \\
C_{c} &= (d^* \! \!, i, j, d, j, i) \quad \text{or} \\
\end{aligned}
\qquad
\begin{aligned}[t]
C_{b} &= (d^* \! \!, i, j, d, i, j) \quad \text{or} \\
C_{d} &= (d^* \! \!, i, j, i, d, j).
\end{aligned}
\end{equation*}
These cycles contribute $(n-1) (n-2)$ to the $\lambda$ value of a graph.

\textbf{If a cycle admits three diagonal edges} it exists only in the case of $\FQ_4$ and is of the form:
$$ D_b= (d^* \! \! , i, d, j, d, k).$$
It contributes $6$ to the $\lambda$ value of a graph.

\begin{table}[ht!]
  \centering
  \resizebox{\ifbig 0.8 \fi \ifsmall 0.6 \fi \columnwidth}{!}{
  \begin{tabular}{c|c|c}
    \textbf{Label}         & \textbf{A representative of a \boldmath$6$-cycle}                                                                & \textbf{Contribution towards \boldmath$\lambda$} \\
    \hline
    $C_a$                  & $(d^*\! \!, i, d, j, i, j)$ & $(n-1)(n-2)$ \\
    \hline
    $C_b$                  & $(d^*\! \!, i, j, d, i, j)$ & $(n-1)(n-2)$ \\
    \hline
    $C_c$                 & $(d^*\! \!, i, j, d, j, i)$ & $(n-1)(n-2)$ \\
    \hline
    $C_d$                  & $(d^*\! \!, i, j, i, d, j)$  & $(n-1)(n-2)$ \\
  \end{tabular}}
  \caption{All non-equivalent $6$-cycles in $\FQ_n$ ($n \notin \{1,2,3,4,6\}$) starting with diagonal edge $d^*$.}
  \label{tab:FQ6-cycles2}
\end{table}

Using the above calculations we get the following result.
\cyclesFQ*

\subsection{Determining \texorpdfstring{\ifbig \boldmath \fi $[2,\lambda,6]$}{[2,lambda,6]}-cycle regularity}
Now we want to determine in how many different $6$-cycles an arbitrary path $p$ of length $2$ lies.
Since two diagonal edges are never incident we again have just two possibilities, one when $p$ consists of two hypercube edges and the other one when $p$ admits a hypercube and a diagonal edge.
We examine both cases and count in how many ways we can extend $p$ to a $6$ cycle in a folded cube.

\subsubsection{Path \texorpdfstring{\ifbig \boldmath \fi $p$}{p} consists of two hypercube edges}
Two incident edges cannot have the same label, therefore, because of the equivalence of two $6$-cycles, we can fix $p$ as the path on hypercube edges $1$ and $2$.
We want to determine in how many different $6$-cycles path $p$ lies.
We distinguish cases, regarding the number of diagonal edges the $6$-cycle admits.
All of the calculation is summarized in \cref{tab:FQ2-6-cycles} and \cref{fig:FQ2-6-cycles}.

\begin{figure}[ht!]
  \centering
  \begin{subfigure}{.24\textwidth}
    \centering
    \includegraphics[width=\linewidth]{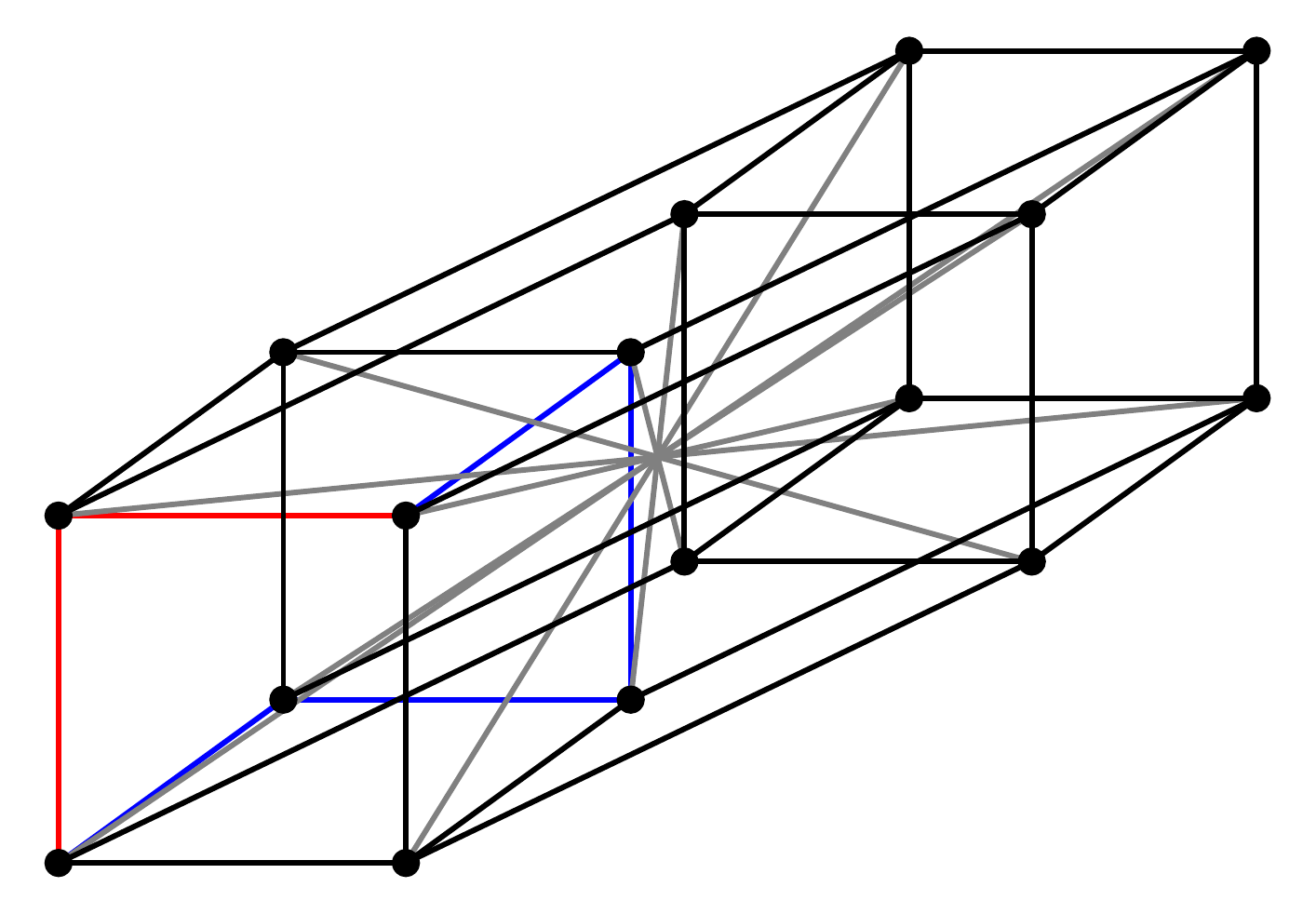}
    \caption{Cycle $C_0$.}
  \end{subfigure}
  \begin{subfigure}{.24\textwidth}
    \centering
    \includegraphics[width=\linewidth]{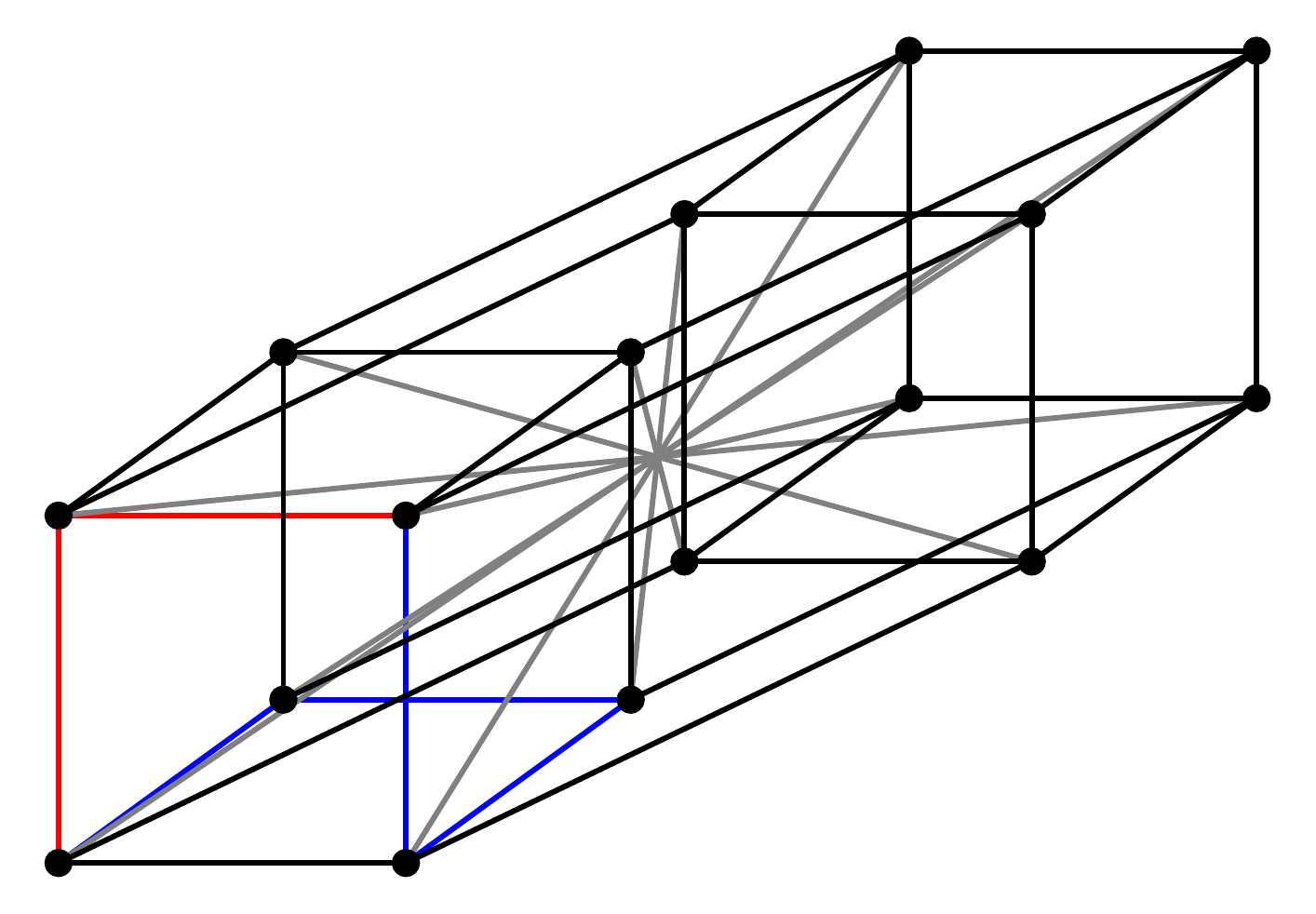}
    \caption{Cycle $C_1$.}
  \end{subfigure}
  \begin{subfigure}{.24\textwidth}
    \centering
    \includegraphics[width=\linewidth]{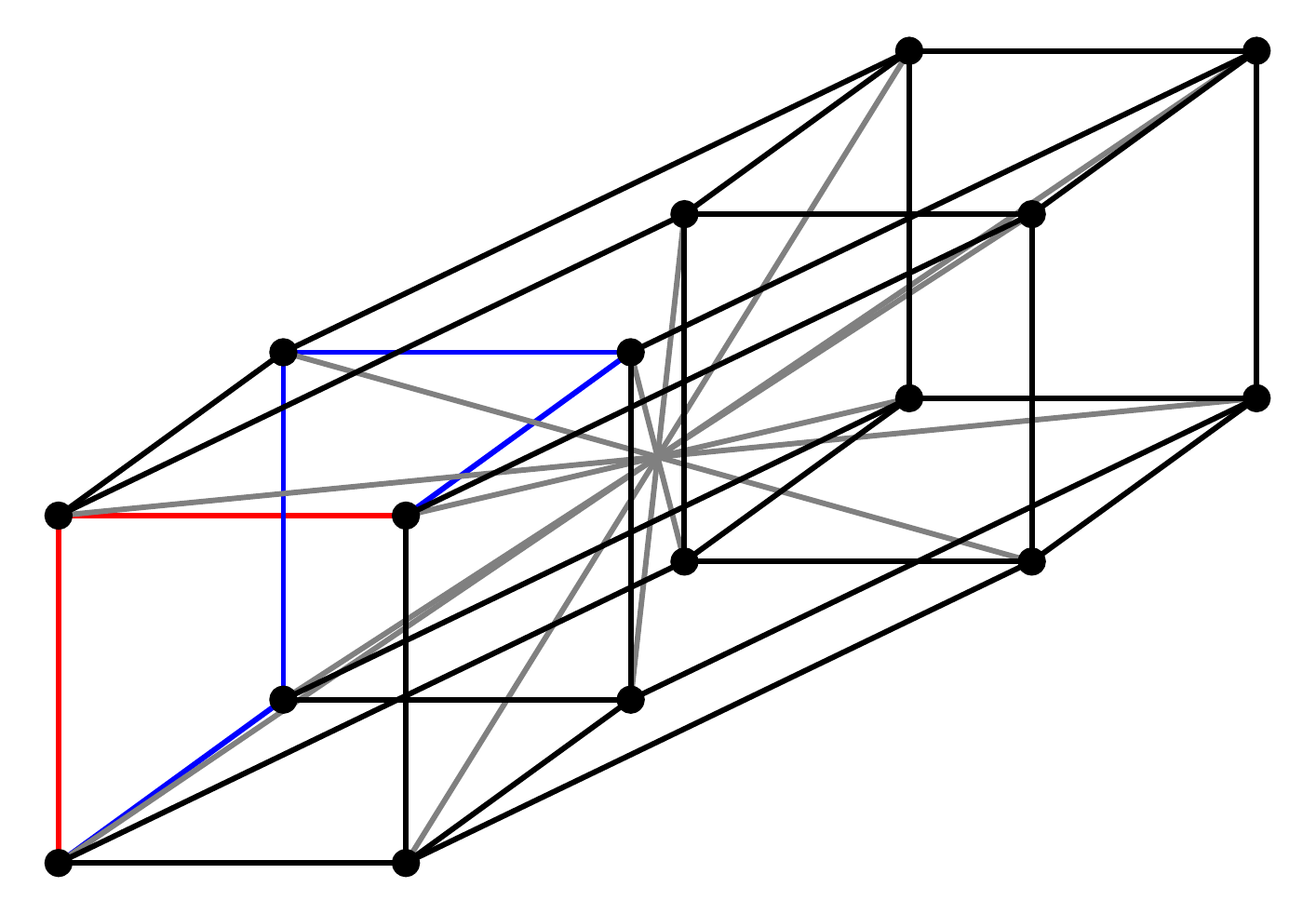}
    \caption{Cycle $C_2$.}
  \end{subfigure}
  \begin{subfigure}{.24\textwidth}
    \centering
    \includegraphics[width=\linewidth]{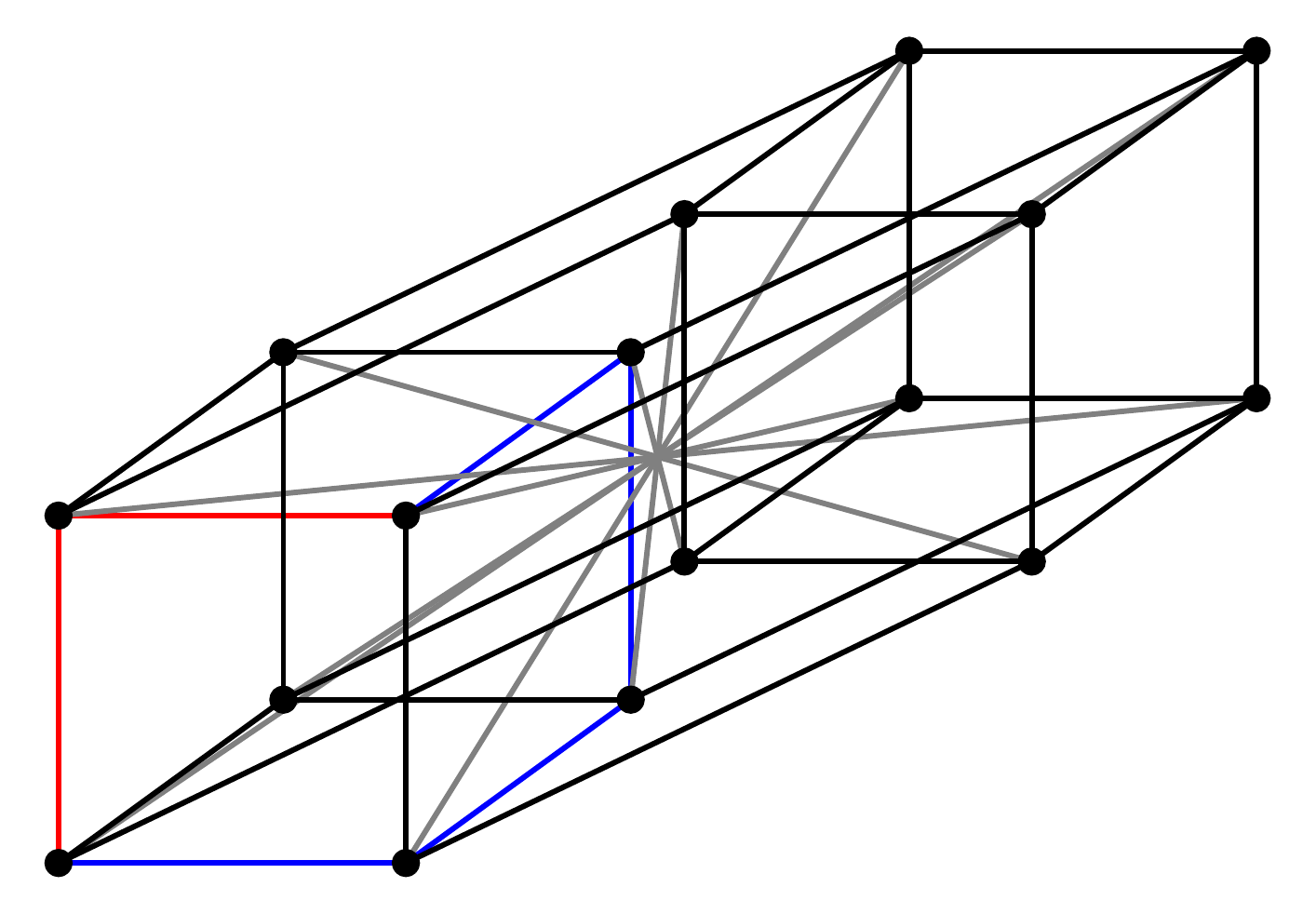}
    \caption{Cycle $C_3$.}
  \end{subfigure}
  
  \begin{subfigure}{.24\textwidth}
    \centering
    \includegraphics[width=\linewidth]{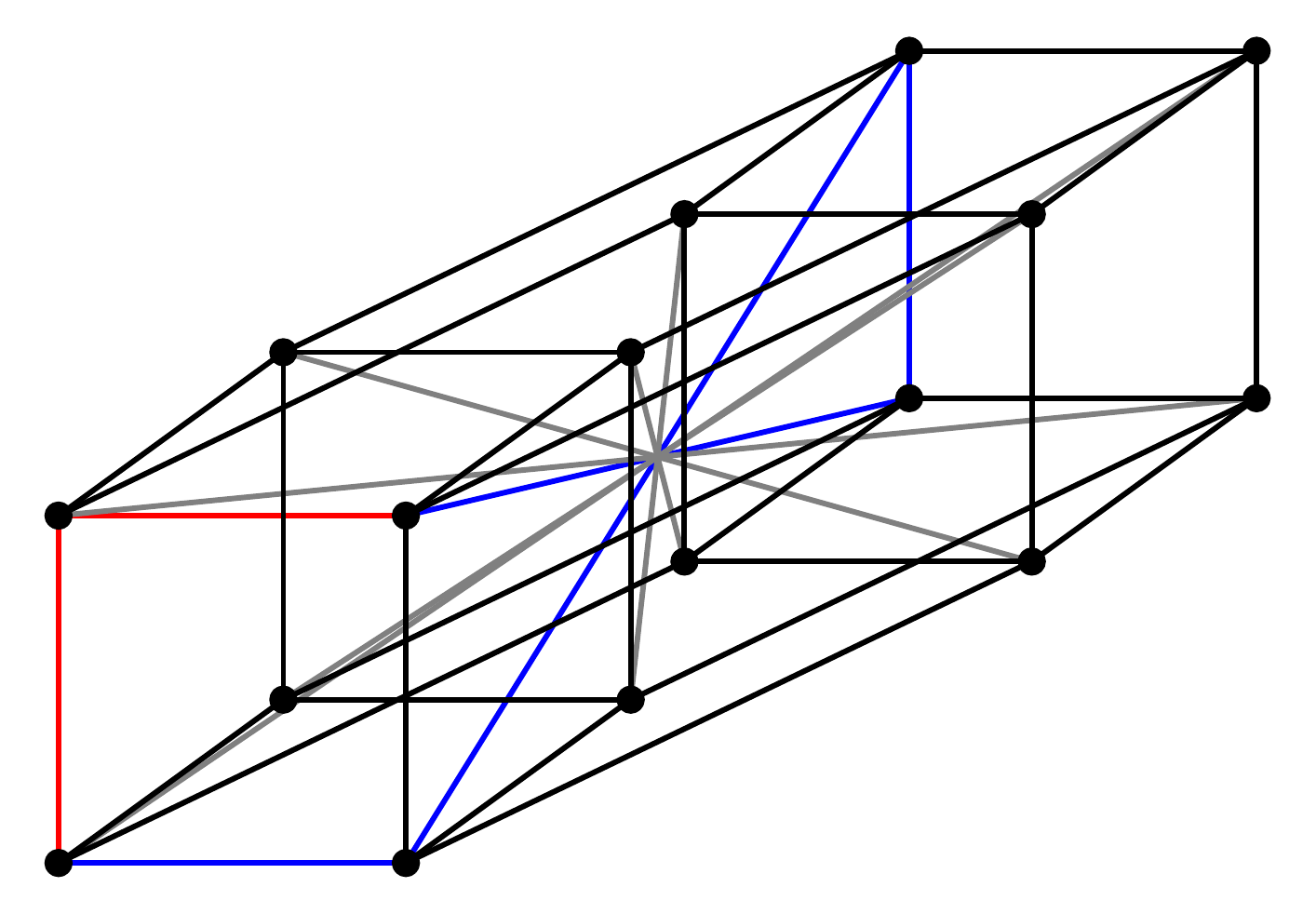}
    \caption{Cycle $C_4$.}
  \end{subfigure}
  \begin{subfigure}{.24\textwidth}
    \centering
    \includegraphics[width=\linewidth]{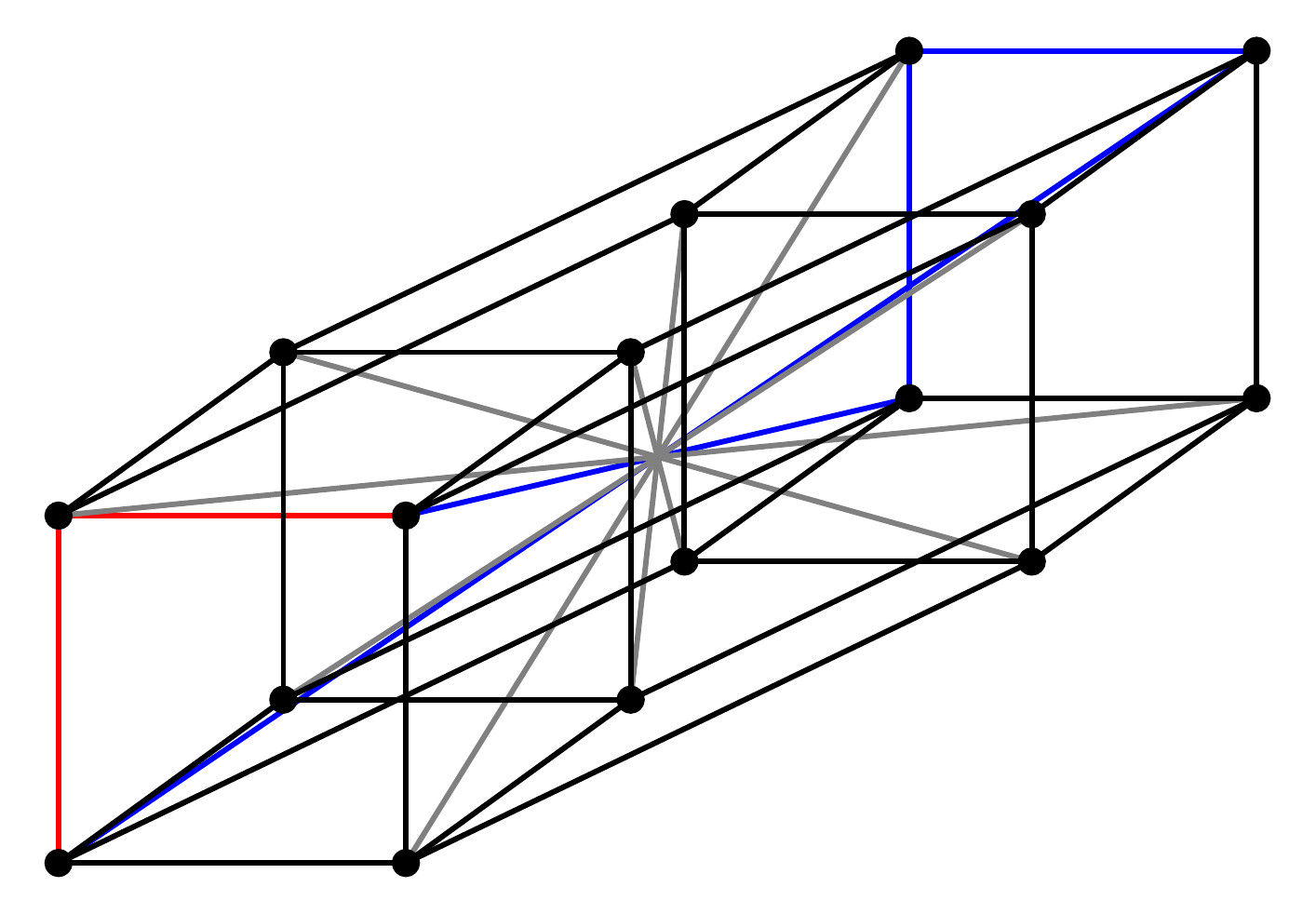}
    \caption{Cycle $C_5$.}
  \end{subfigure}
  \begin{subfigure}{.24\textwidth}
    \centering
    \includegraphics[width=\linewidth]{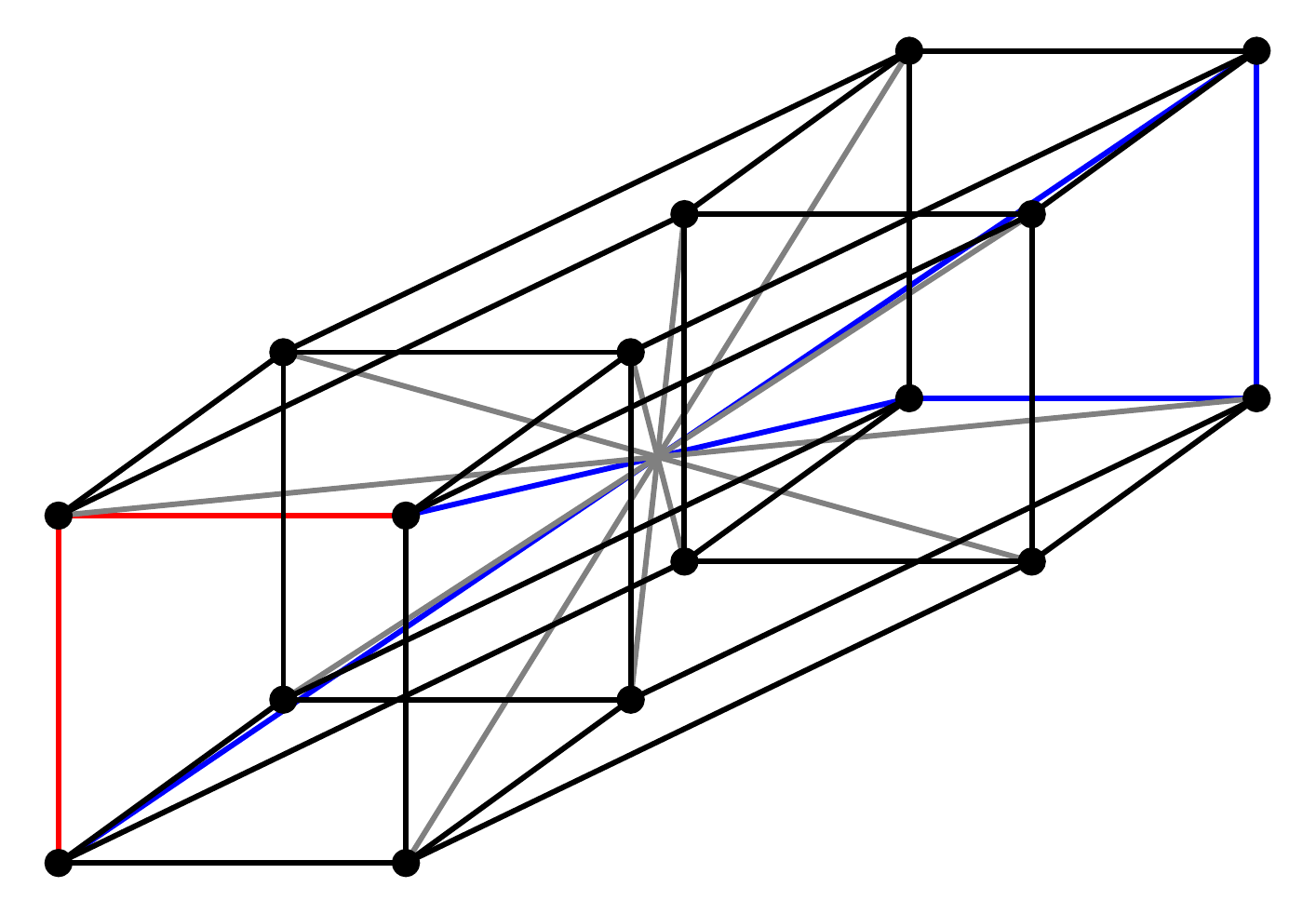}
    \caption{Cycle $C_6$.}
  \end{subfigure}
  \begin{subfigure}{.24\textwidth}
    \centering
    \includegraphics[width=\linewidth]{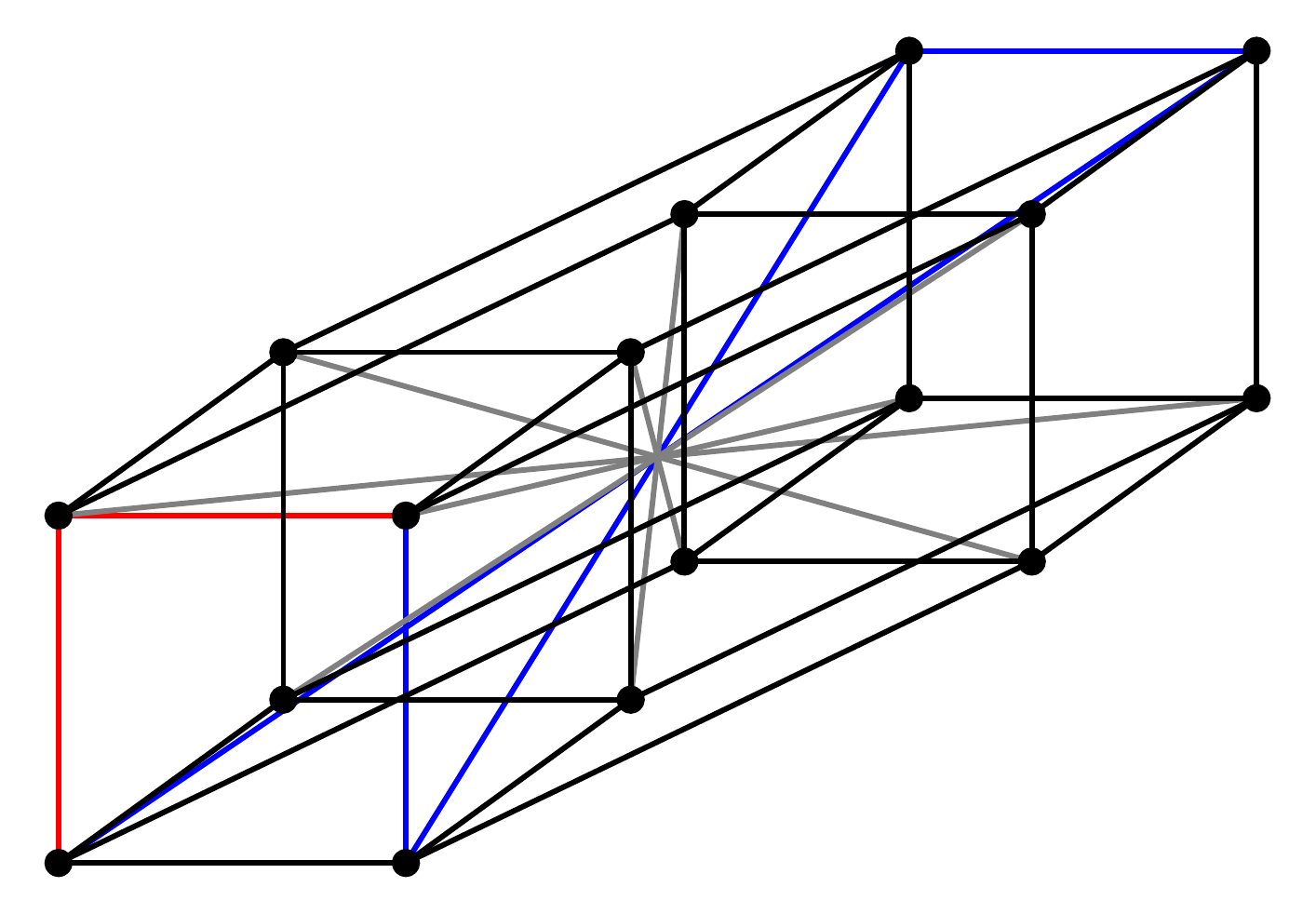}
    \caption{Cycle $C_7$.}
  \end{subfigure}
\caption{All non-equivalent $6$-cycles in $\FQ_5$ starting with path $p=(1,2)$. \label{fig:FQ2-6-cycles}}
\end{figure}

\textbf{If a cycle admits no diagonal edges} it has to be of one of the following forms:
\begin{equation*}
\begin{aligned}[t]
C_0 &= (1, 2, i, 1, 2, i) \quad \text{or} \\
C_2 &= (1, 2, i, 2, 1, i) \quad \text{or} \\
\end{aligned}
\qquad
\begin{aligned}[t]
C_1 &= (1, 2, 1, i, 2, i) \quad \text{or} \\
C_3 &= (1, 2, i, 1, i, 2).
\end{aligned}
\end{equation*}
The contribution of each cycle to $\lambda$ value is $(n-3)$.

\textbf{If a cycle admits only one diagonal edge} it has one of the five following forms:
\begin{equation*}
\begin{aligned}[t]
D_1 &= (1, 2, d, i, j, k) \quad \text{or} \\
D_3 &= (1, 2, i, j, d, k)  \quad \text{or} \\
\end{aligned}
\qquad
\begin{aligned}[t]
D_2 &= (1, 2, i, d, j, k)  \quad \text{or} \\
D_4 &= (1, 2, i, j, k, d).
\end{aligned}
\end{equation*}
These cycles exists only in the case of $\FQ_6$ and contribute $6$ to the $\lambda$ value of a graph.

\textbf{If a cycle admits two diagonal edges} it has to be of one of the following forms:
\begin{equation*}
\begin{aligned}[t]
C_4 &= (1, 2, d, 1, d, 2) \quad \text{or} \\
C_6 &= (1, 2, d, 2, 1, d) \quad \text{or} \\
\end{aligned}
\qquad
\begin{aligned}[t]
C_5 &= (1, 2, d, 1, 2, d) \quad \text{or} \\
C_7 &= (1, 2, 1, d, 2, d).
\end{aligned}
\end{equation*}
Each of these cycles contributes to the $\lambda$ value of a graph exactly $1$.

It is not hard to see that there does not exist a $6$-cycle on containing path $p$ and $3$ diagonal edges.


\begin{table}[ht!]
  \centering
  \resizebox{\ifbig 0.8 \fi \ifsmall 0.6 \fi \columnwidth}{!}{
  \begin{tabular}{c|c|c}
    \textbf{Label}         & \textbf{A representative of a \boldmath$6$-cycle}                                                                & \textbf{Contribution towards \boldmath$\lambda$} \\
    \hline
    $C_0$                  & $(1, 2, i, 1, 2, i) $ & $(n-3)$ \\
    \hline
    $C_1$                  & $(1, 2, 1, i, 2, i)$ & $(n-3)$ \\
    \hline
    $C_2$                  & $(1, 2, i, 2, 1, i)$ & $(n-3)$ \\
    \hline
    $C_3$                  & $(1, 2, i, 1, i, 2)$ & $(n-3)$ \\
    \hline
    $C_4$                 & $(1, 2, d, 1, d, 2)$ & $1$ \\
    \hline
    $C_5$                  & $(1, 2, d, 1, 2, d)$  & $1$ \\
    \hline
    $C_6$                  & $(1, 2, d, 2, 1, d)$ & $1$ \\
    \hline
    $C_7$                  & $(1, 2, 1, d, 2, d)$ & $1$ \\
    
  \end{tabular}}
  \caption{All non-equivalent $6$-cycles in $\FQ_n$ ($n \notin \{1,2,3,4,6\}$) starting with path $p=(1,2)$.}
  \label{tab:FQ2-6-cycles}
\end{table}

\subsubsection{Path \texorpdfstring{\ifbig \boldmath \fi $p$}{p} consists of a hypercube and diagonal edge}
In this case $p$ consists of a hypercube and a diagonal edge. Because of the equivalence of $6$-cycles we can fix $p$ as a path on $1$ and $d^*$.
Note, the order of edges in $p$ is not important as cycles through $(1,d^*)$ are the same as cycles through $(d^*\! \!,1)$, we just write them in an opposite order.
All of the calculation is summarized in \cref{tab:FQ2-6-cycles2} and \cref{fig:FQ2-6-cycles2}.

\begin{figure}[ht!]
  \centering
  \begin{subfigure}{.24\textwidth}
    \centering
    \includegraphics[width=\linewidth]{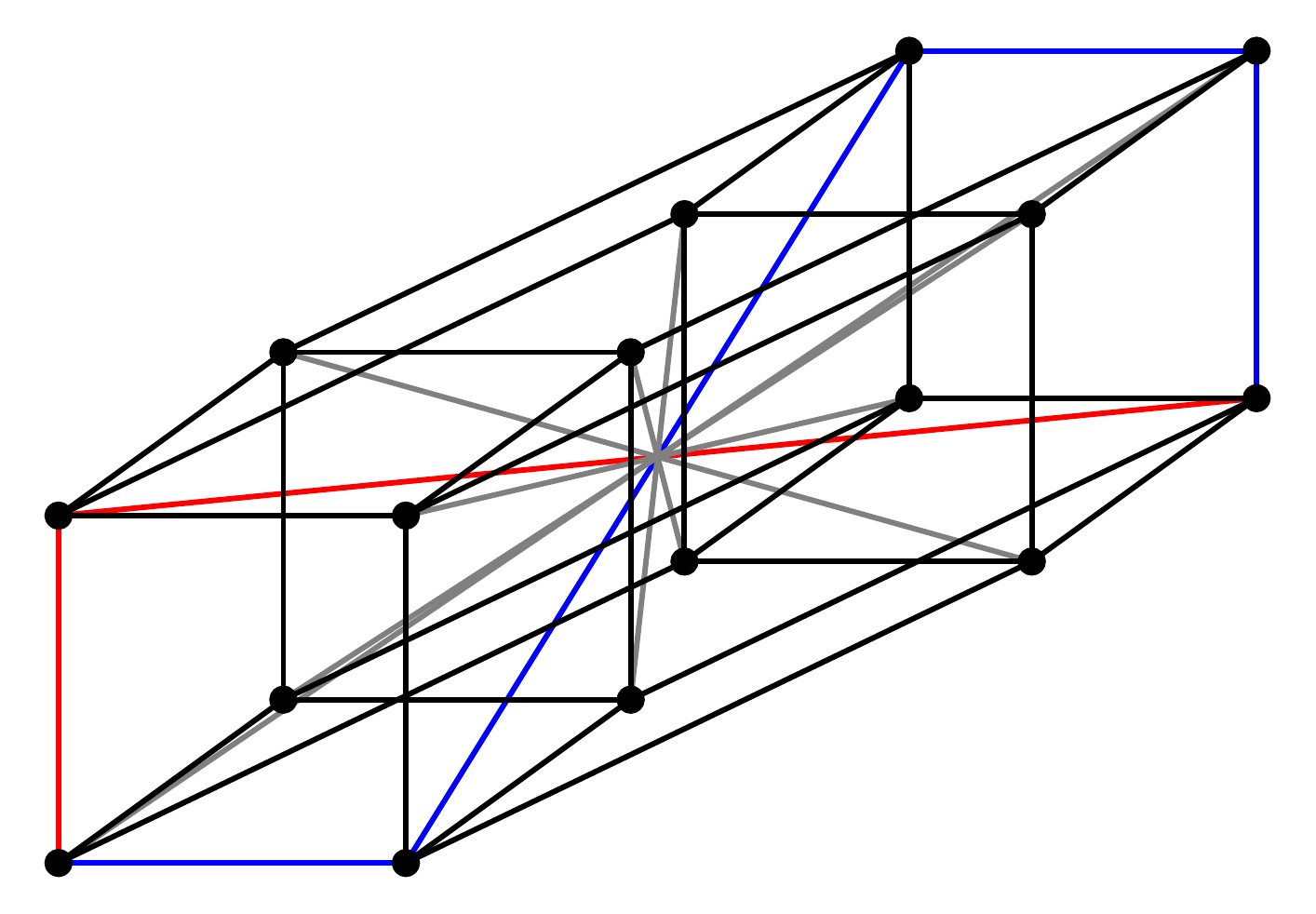}
    \caption{Cycle $C_a$.}
  \end{subfigure}
  \begin{subfigure}{.24\textwidth}
    \centering
    \includegraphics[width=\linewidth]{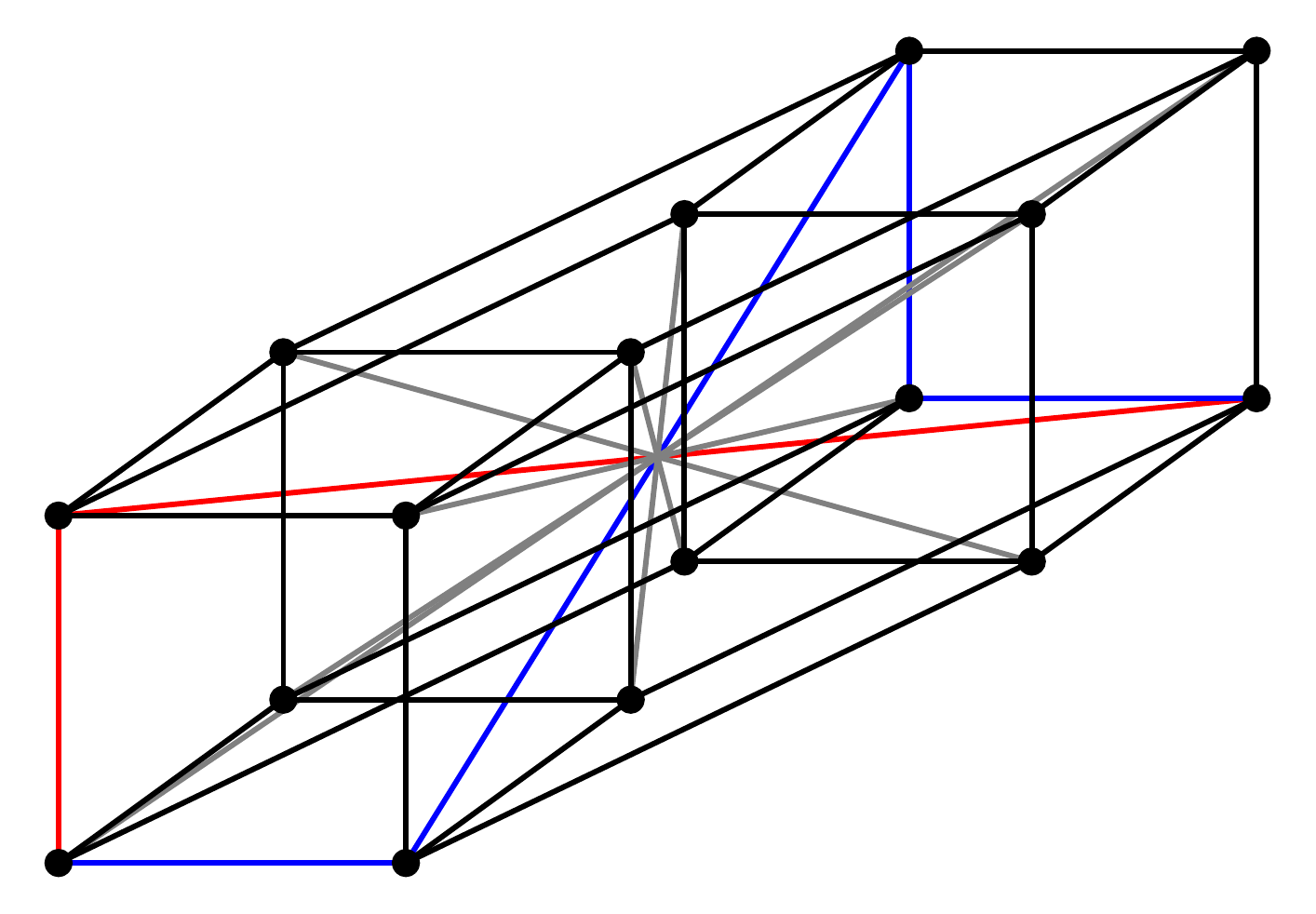}
    \caption{Cycle $C_b$.}
  \end{subfigure}
  \begin{subfigure}{.24\textwidth}
    \centering
    \includegraphics[width=\linewidth]{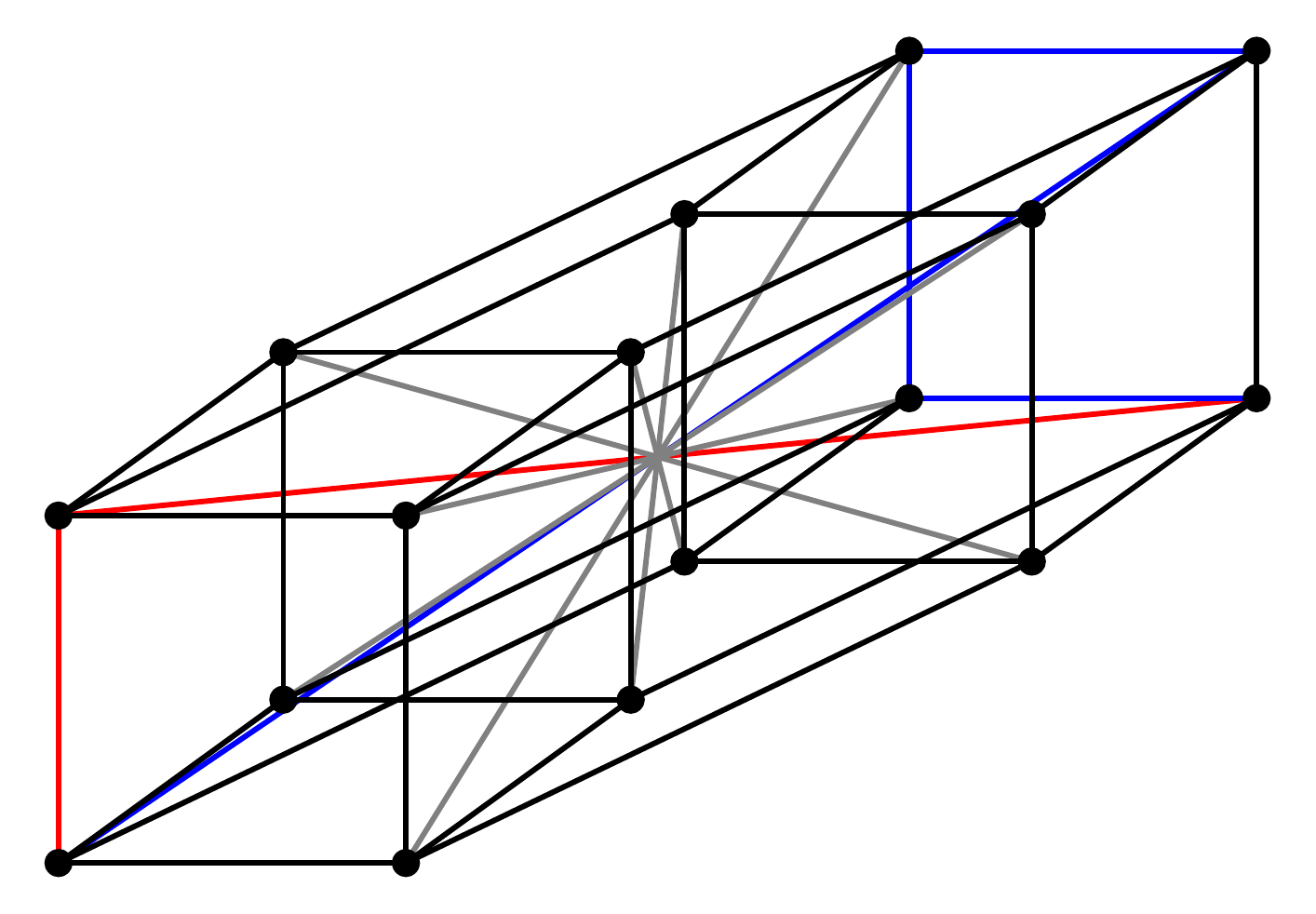}
    \caption{Cycle $C_c$.}
  \end{subfigure}
  \begin{subfigure}{.24\textwidth}
    \centering
    \includegraphics[width=\linewidth]{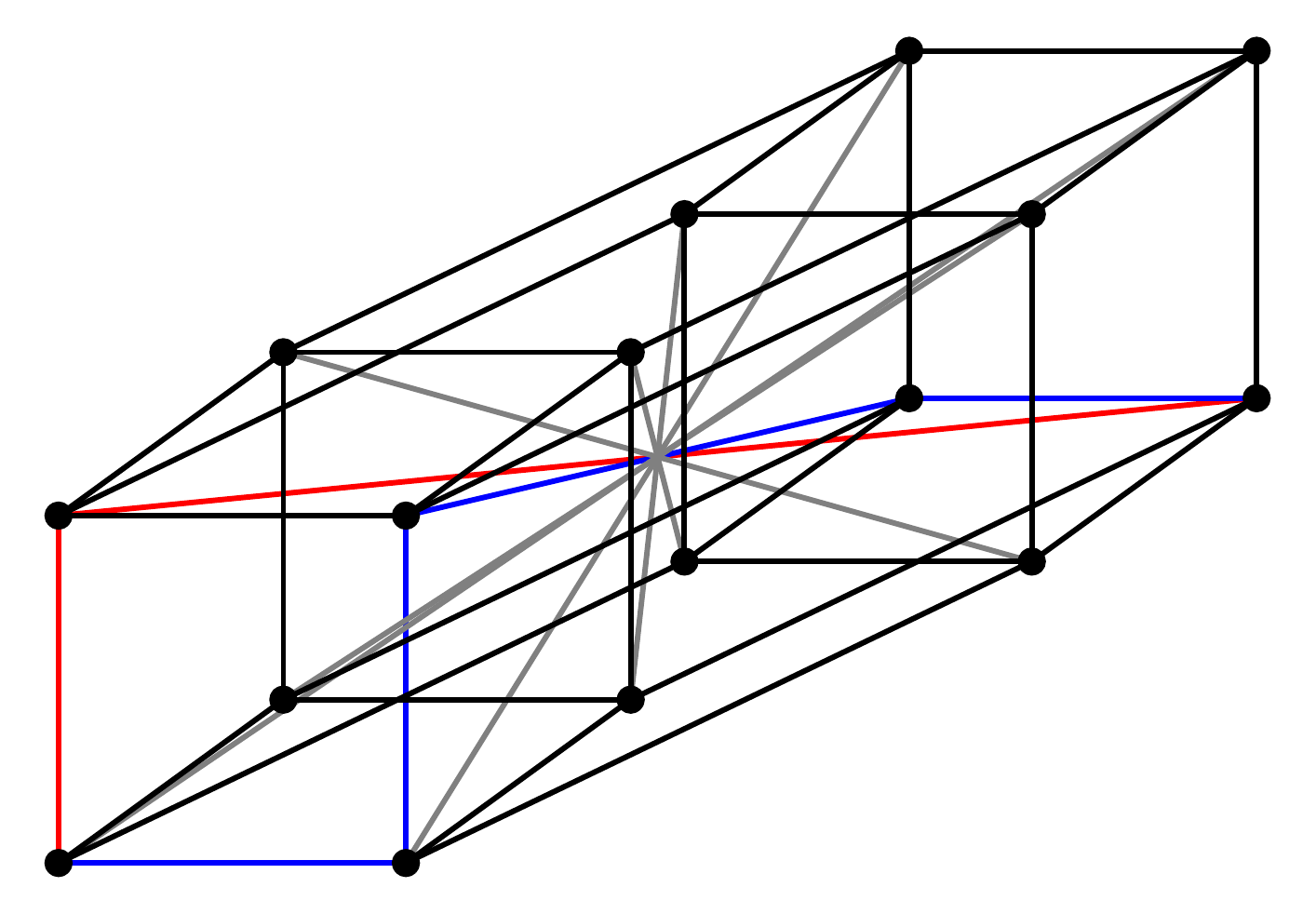}
    \caption{Cycle $C_d$.}
  \end{subfigure}
\caption{All non-equivalent $6$-cycles in $\FQ_5$ starting with path $p=(1,d^*)$. \label{fig:FQ2-6-cycles2}}
\end{figure}

\paragraph{If a cycle admits only one diagonal edge} it exists only in the case of $\FQ_6$ and is of one of the following forms:
\paragraph{If a cycle admits only one diagonal edge} it is of the following form:
$$ D_a = (1, d^* \! \!, m, i, j, k) .$$
This cycle contributes $24$ to the $\lambda$ value of a graph.

\paragraph{If a cycle admits two diagonal edges} it has one of the following forms:
\begin{equation*}
\begin{aligned}[t]
C_{a} &= (1, d^* \! \!, 1, m, d, m) \quad \text{or} \\
C_{c} &= (1, d^* \! \!, m, 1, m, d) \quad \text{or} \\
\end{aligned}
\qquad
\begin{aligned}[t]
C_{b} &= (1, d^* \! \!, m, 1, d, m) \quad \text{or} \\
C_{d} &= (1, d^* \! \!, m, d, 1, m).
\end{aligned}
\end{equation*}
These cycles contribute $(n-2)$ to the $\lambda$ value of a graph.

\paragraph{If a cycle admits three diagonal edges} it is of the following form:
$$ D_b = (1, d^* \! \!, m, d, j, d) .$$
This cycle exists only in the case of $\FQ_4$ and contributes $2$ to the $\lambda$ value of the graph.

\begin{table}[ht!]
  \centering
  \resizebox{\ifbig 0.8 \fi \ifsmall 0.6 \fi \columnwidth}{!}{
  \begin{tabular}{c|c|c}
    \textbf{Label}         & \textbf{A representative of a \boldmath$6$-cycle}                                                                & \textbf{Contribution towards \boldmath$\lambda$} \\
    \hline
    $C_a$                  & $(1, d^* \! \!, 1, m, d, m)$ & $(n-2)$ \\
    \hline
    $C_b$                  & $(1, d^* \! \!, m, 1, d, m)$ & $(n-2)$ \\
    \hline
    $C_c$                 & $(1, d^* \! \!, m, 1, m, d)$ & $(n-2)$ \\
    \hline
    $C_d$                  & $(1, d^* \! \!, m, d, 1, m)$  & $(n-2)$ \\
  \end{tabular}}
  \caption{All non-equivalent $6$-cycles in $\FQ_n$ ($n \notin \{1,2,3,4,6\}$) starting with path $p=(1,d^*)$.}
  \label{tab:FQ2-6-cycles2}
\end{table}

Using the above calculations we get the following result.

\cyclesFQpathOfLengthTwo*

A similar procedure can be extended to an arbitrary $m$-cycle.
Since folded cubes $\FQ_{2n}$ are bipartite (see \cref{thm:FQ-bipartite}), they admit only cycles of even length.
It is not hard to see that for an arbitrary folded cube $\FQ_n$ the value $\lambda$ of $[1, \lambda, 2k]$-cycle regularity is in $O(n^{k-1})$.
There is no precise characterization of $\lambda$ for all values of $k$.
For $k = 4$ we suspect the following.
\EightCyclesFQ*

\section{Recognition algorithms \label{section-RecognitionAlgorithms}}
In this section we describe recognition algorithms for all of the described graph families.
Since recognition algorithms for $I$-graphs and double generalized Petersen graphs are very similar, we analyze them together (see \cref{sec:recognizingIandDP}) and then we proceed with the analysis of the recognition algorithm of folded cubes (see \cref{sec:recognizingFQ}).
All of the algorithms are robust, which means that they receive as an input an arbitrary graph and determine if the graph is a part of the family of $I$-graphs, double generalized Petersen graphs or folded cubes. In the case the input graph belongs to the observed family, the algorithm provides its parameters, together with the certificate of correctness (exact isomorphsim).

\subsection{Recognizing \texorpdfstring{\ifbig \boldmath \fi $I$}{I}-graphs and double generalized Petersen graphs \label{sec:recognizingIandDP}}

This algorithm uses some of the properties of graph families discussed above and is presented in \cref{Algorithm1}.

Recall, the set of edges of an arbitrary $I$-graph or double generalized Petersen graph can be partitioned into $3$ parts, where edges from each part admit the same octagon value.

In particular, whenever the input graph $G$ is from the family of $I$-graphs or double generalized Petersen graphs, unless it is $[1,\lambda,8]$-cycle regular, we can immediately identify one of its edge orbits ($E_I, E_J$, or $E_S$) of size $|V(G)|/2$. \cref{thm:lin-Igr,thm:lin-DP} guarantee that such a graph is $[1,\lambda,8]$-cycle regular only in one of the cases described in previous sections (see \cref{fig:specialIgraphs,fig:DP-graphs}).
Since the octagon value of each edge is computed in constant time and there is a finite number of the $[1,\lambda,8]$-cycle regular $I$-graphs and double generalized Petersen graphs, the first part of \cref{cor:lin} holds.

\begin{algorithm}[ht!]
  \caption{Recognition procedure for $I$-graphs or for $\DP$ graphs, depending on the subprocedure \textsc{Extend}$(G,U)$.   \label{Algorithm1}}
  \begin{algorithmic}[1]
    \Require connected cubic graph $G$
    \State $\mathcal{P} \gets$ an empty dictionary \label{alg:1}
    \For{$e \in E(G)$} \label{alg1:octagon} \label{alg:2}
    \State $s = \textsc{octagonValue}(e)$ \Comment{calculate $\sigma(e)$} \label{alg:3}
    \State $\mathcal{P}[s].\text{append}(e)$ \label{alg:4}
    \EndFor
    \State $U \leftarrow$ an item of $\mathcal{P}$ with minimum positive cardinality \label{alg:5}
    \If{$G[U]$ is a $2$-factor}
    \State $U \gets \{ e\mid e \in E(G), \, e\text{ is adjacent to an edge of } U\}$ \Comment{$U$ is a perfect matching in $G$}\label{alg:7}
    \EndIf
    \State \Return \textsc{Extend}$(G,U)$
  \end{algorithmic}
\end{algorithm}

\subsubsection{Description of \texorpdfstring{\cref{Algorithm1}}{Algorithm 1}}
We first note, that if $G$ is not cubic then it does not belong to observed graph families. Since checking whether a graph is cubic takes linear time we simply assume that the input graph is cubic.
Furthermore,
if $G$ is not connected then it can only be a member of a family of $I$-graphs whenever it consists of multiple copies of a smaller $I$-graph $G'$.
However, this case can easily be resolved by separately checking each part, so we can assume that the input graph is connected.
\cref{Algorithm1} consists of the following $3$ parts.
 
\paragraph{Partitioning the edge partition set with respect to octagon value.}
        The algorithm determines the octagon value of each edge $e \in E(G)$ and builds partition set $\mathcal P$ of graph edges (see lines \ref{alg:1}  -- \ref{alg:4}).
        Since $G$ is cubic and all $8$-cycles containing edge $e$ consist of edges which are at distance at most $4$ from $e$, it is enough to check a subgraph $H$ of $G$ of order at most $62$, and calculate the octagon value of edge $e$. Therefore, calculation of $\textsc{octagonValue}(e)$ takes $O(1)$ time for each edge $e$ and this whole part is performed in $\Theta(|E(G)|)$ time.

\paragraph{Identifying the edge-orbit which corresponds to the set of spokes.}
        Throughout lines~\ref{alg:5} -- \ref{alg:7} we then  determine
        the edge-orbit which corresponds to the set of spokes.
        It is easy to see that this requires additional $O(|E(G)|/3)$ time.

\paragraph{\boldmath Using set $U$ for determining parameters of a given graph.}
       The algorithm uses computed set $U$ to determine exact isomorphism between $G$ and an $I$-graph or a double generalized Petersen graph, if it exists.
        This procedure differentiates regarding the graph family we are consiedering. The related procedure \textsc{Extend}$(G,U)$ is performed in $\Theta(|E(G)|)$ time.

\subsubsection{Subprocedure \texorpdfstring{\textsc{Extend} \ifbig \boldmath \fi$(G,U)$}{Extend(G,U)} for \texorpdfstring{\ifbig \boldmath \fi $I$}{I}-graphs} \label{sec:extendI}
It is easy to see that \textsc{Extend}$(G,U)$ can safely reject $G$, if $|V(G)|$ is not divisible by $2$.
For this subprocedure, set $n=\vert V(G) \vert / 2$ and
$G[U]$ by $H$. Again, there are two possibilities.

\begin{description}
  \item[\boldmath$H = G$.] In this case graph $G$ has a constant octagon value. All ten $I$-graphs with constant octagon value are depicted in \cref{fig:specialIgraphs}, and checking $G$ against them takes constant time.

  \item[\boldmath$H$ is of order $2n$ and is $1$-regular.]
        Since $U$ is a perfect matching of $G$ the set $E(G) \setminus U$ is a collection of cycles.
        If $G$ is an $I$-graph, then there exist positive integers $i, j, l_1, l_2$ with $j\leq i$ such that there are $j$ cycles of length $l_1$ and $i$ cycles of length $l_2$.  It remains to determine parameter $k$ and check whether $G\simeq I(n,j,k)$.  This procedure is described in \cref{sec:isomorphism}. Again, the whole procedure \textsc{Extend}$(G,U)$ is performed in linear time.

\end{description}

\subsubsection{Subprocedure \texorpdfstring{\textsc{Extend}\ifbig \boldmath \fi $(G,U)$}{Extend(G,U)} for double generalized Petersen graphs} \label{sec:extendDP}

It is easy to see that \textsc{Extend}$(G,U)$ can safely reject $G$, if $|V(G)|$ is not divisible by 4.
Let us denote $\vert V(G) \vert / 4$ by $n$ and $G[U]$ by $H$. There are two possibilities.

\begin{description}
  \item[\boldmath$H = G$.] In this case graph $G$ has a constant octagon value. Therefore it is from the family of double generalized Petersen graphs whenever it is isomorphic to one of the $[1,\lambda,8]$-cycle regular members of the family. Since there are only two such double generalized Petersen graphs (see \cref{fig:DP-graphs}), this procedure is performed in constant time.

  \item[\boldmath$H$ is of order $4n$ and is $1$-regular.]

        Since $U$ is a perfect matching of $G$ the set $E(G) \setminus U$ is a collection of cycles. In order for $G$ to be a double generalized Petersen graph, there must be two cycles of length $n$ and $m$ cycles of length $a$, where $m, a$ are positive integers. It remains to determine parameter $k$. Since we already uniquely identified the set of spokes $U$, this can be achieved easily (for the detailed procedure see \cref{sec:isomorphism}), and hence the whole procedure \textsc{Extend}$(G,U)$ is performed in linear time.
        
\end{description}

In both versions of subprocedure $\textsc{Extend}(G,U)$, if graphs $H$ and $G$ do not satisfy any of the above possibilities, or if the isomorphism does not exist, the algorithm returns False. If there is an isomorphism, then the algorithm returns parameters of the corresponding $I$-graph or double generalized Petersen graph.

\subsubsection{Finding exact isomorphisms} \label{sec:isomorphism}
In this section we provide a procedure for determining exact isomorphisms with $I$-graphs or double generalized Petersen graphs.

\subsubsection*{Exact isomorphism of double generalized Petersen graphs}
Since $U$ is a perfect matching of $G$, the set $E(G) \setminus U$ is a collection of cycles.
There are two cycles of length $n$ and $m$ cycles of length $a$.
To determine parameter $k$ first fix one $n$-cycle as $(u_0, \dots, u_{n-1})$.
Using $U$ determine their neighbours $w_i$.
With the help of a cycle from $E(G) \setminus U$ that goes through vertex $w_0$, label its neighbours, different than $u_0$, as $y_k$ and $y_{-k}$, in an arbitrary way. Again using a perfect matching $U$ determine $x_k$ and $x_{-k}$.
Now, using an $n$-cycle from $E(G) \setminus U$, that goes through $x_k$ and $x_{-k}$, label vertices $x_0, x_1, \dots , x_{n-1}$
in such a way that $p = x_{-k}, x_{-k + 1}, \dots x_0, x_1, \dots, x_k$ is a path using an even number of edges. From this determine also parameter $k$.
In the case of $n$ being even we get two paths from $x_{-k}$ to $x_k$ of even length. Both paths will give us a different value for parameter $k$, but since \cref{thm:isomorphDP} holds, their graphs are isomorphic.
Now that all vertices $x_i$ are labeled, determine all vertices $y_i$ using $U$.
At this point it remains to check if all cycles from $E(G) \setminus U$ respect computed labeling.
In the case they do a graph $G$ is isomorphic to $\DP(n,k)$, if they do not, then $G$ is not from our family.

It is easy to see that this procedure is completed in $\Theta(|E(G)|)$ time.

\subsubsection*{Exact isomorphism of \texorpdfstring{\ifbig \boldmath \fi $I$}{I}-graphs}
\cref{alg:isomI} is an extension of the subprocedure \textsc{Extend}$(G,U)$ of \cref{Algorithm1} and provides an exact isomorphism from a graph $G$ to an $I$-graph, if there exists. When considering the given algorithm one should note that all the subscripts are given modulo $n$ where $n$ equals $|V(G)|/2$.

\begin{algorithm}[ht!]
  \caption{Finding exact isomorphisms with an $I$-graph. \label{alg:isomI}}
  \begin{algorithmic}[1]
    \Require connected cubic graph $G$ on $2n$ vertices, together with a perfect matching $U$
    \State $k,j,l_1,l_2 \gets$ number of cycles in the set $ F = E(G) \setminus U$ and their lengths \label{alg2:line1}
    \State Take a $l_1$-cycle $C$ from $F$ and label its vertices as $(u_0,u_j,\dots, u_{(l_1 - 1)j})$.
    \For{both cycles of type $C^*$ through $u_0 u_j$ \label{alg2:line03}}
    \State label the corresponding vertices $w_0, w_k, u_k, u_{k+j}, w_{k+j}, v_j$
    \For{$e \in \{u_0 u_j, u_j u_{2j}, \dots , u_{(l_1-1)j} u_0 \}$ \label{alg2:line2}} 
    \State Identify cycle $C'$ through $e$ of form $C^*$ respecting current labeling.
    \State Label unlabeled vertices accordingly.
    \EndFor
    \For{any $l_2$-cycle $C$}
    \State find an edge of $C$ labeled on both ends and label the rest of $C$.
    \State Using $U$ label remaining vertices. 
    \EndFor
    \State If labeling is consistent \textbf{return True}.
    \EndFor
    \State \textbf{Return False}.
  \end{algorithmic}
\end{algorithm}
Graph $G$ and set $U$ meet the requirements as this was checked already by \cref{Algorithm1}.
Let $n = |V(G)| / 2$. During the whole analysis we consider indices modulo $n$.

Because $U$ is the perfect matching, $F = E(G) \setminus U$ is the union of at least two disjoint cycles. 
Since $G$ is connected, $F$ admits either $2$ cycles of length $n$ or more cycles of two different lengths. 
On line~\ref{alg2:line1} we set the number of longer cycles as $j$ and fix their lengths as $l_1$ and the number of shorter cycles as $k$ with the length $l_2$.
This is done in time $O(|E(G)|)$.

Then we take an $l_1$-cycle $C$ from $F$ and fix it on the outer rim as $(u_0,u_j,\dots,$ $u_{(l_1 - 1)j})$.
For an edge $u_0u_j$ we construct a cycle of form $C^*$ and label its vertices accordingly.
Since the orientation of inner edges can be chosen in two ways and only one of these two produce the desired $8$-cycle, we repeat the procedure at most $2$ times.
When $C^*$ is determined, we cyclicly go through the rest of the edges of $C$, identify their corresponding cycle $C'$ of form $C^*$ and label its vertices.
Now we take one $l_2$-cycle and find an edge with both incident vertices already labeled and we extend this labeling through the whole cycle.
At the end, we use $U$ to label all remaining vertices.
If the labeling is not consistent then either $C^*$ at line~\ref{alg2:line03} was not oriented properly or $G$ is not an $I$-graph.
In the case of cycles from $F$ being of the same lengths, constructing $C^*$ actually constructs a $4$-cycle, which labels just $2$ new vertices instead of $6$, but the procedure still returns a labeled $I$-graph. 
This procedure is performed in $O(|E(G)|)$ time.

It is easy to see that
\cref{alg:isomI} labels all vertices.

\subsection{Recognizing folded cubes \label{sec:recognizingFQ}}
Since folded cubes are arc-transitive and $[2,\lambda,4]$ and $[2, \lambda, 6]$-cycle regular, using $4$ or $6$-cycles would not help us with the recognition problem.
Therefore we constructed a solution which depends solely on the nice structure of folded cubes.

Before we start with the description of the algorithm, we need to fix the following notation.
For vertices $u,v$ of graph $G$, the set $W_{uv}$ represents all the vertices of a graph that are closer to $u$ than $v$.

The algorithm \cref{alg:recognizingFQ} consists of two parts.
In the first one we recognize the set of diagonal edges which is then used in the second one (see \cref{alg:FQsubprocedure}) to determine if the input graph is isomorphic to a folded cube.

\begin{algorithm}[ht]
  \caption{Determining diagonal edges of a folded cube. \label{alg:recognizingFQ}}
  \begin{algorithmic}[1]
    \Require $n$-regular, connected graph $G$ on $2^{n-1}$ vertices.
    \State $S_1,\dots, S_{n-1}\gets$ empty sets, 
    $S_{n}\gets $ arbitrary edge of $E(G)$, \label{ln:1}
    $i\gets n$
    \While{$i>1$ \label{alg:line2While}}
    \State fix $uv\in S_i$, and move it from $S_i$ to $S_{1}$
    \For{arbitrary edges $au, vb \in E(G) \setminus \{uv\}$ \label{alg:line4For}}
    \If{$ab \in E(G)$}
    \State add $ab$ to $S_{\deg(a)-1}$ and (if possible) remove it from $S_{\deg(a)}$ \label{ln:6}
    \State delete edges $au, vb$ from $G$
    \EndIf
    \EndFor
    \State $i \gets \min\{j\mid S_j\neq \emptyset\}$
    \EndWhile
    \State \Return $\textsc{Extend}(G,S_1)$.
  \end{algorithmic}
\end{algorithm}

If $G$ is not connected, $n$-regular on $2^{n-1}$ vertices, then it does not belong to the family of folded cubes. Since checking these conditions takes linear time, we simply assume that the input graph of \cref{alg:recognizingFQ} meets the requirements. 

\paragraph{Determining set of diagonal edges.}
We start \cref{alg:recognizingFQ} by initializing empty sets $S_1, \dots , S_{n-1}$, setting $i = n$ and $S_n = \{uv\}$, where $uv$ is an arbitrary edge in $E(G)$. 
The \texttt{while}-loop starting on line~\ref{alg:line2While} runs until all $S_2, \dots, S_n$ are emptied and invokes the following procedure.
Take a diagonal edge $uv$ from the set $S_i$, move it to the set $S_{1}$ and start searching for $4$-cycles going through it, i.e. determining diagonal edges that lie in the same $4$-cycle as $uv$.
This is performed with the \texttt{for}-loop, on line~\ref{alg:line4For}.
Whenever we find a $4$-cycle $(ua,uv,vb,ab)$,
we remove the non-diagonal edges $ua$ and $vb$ from the graph, 
followed by reassigning both corresponding diagonal edges to the sets matching their (now decreased) degree. 
After determining all of the $4$-cycles through $uv$ we conclude the current cycle of the \texttt{while}-loop by setting $i$ to be the minimum index of currently non-empty sets $S_1, \dots , S_n$. 
As the \texttt{while}-loop is concluded, i.e. all identified diagonal edges are assigned to $S_1$, 
the only thing left to check is if $G$ is isomorphic to an $n$-dimensional folded cube.
This is done using subprocedure $\textsc{Extend}(G,S_1)$, which relies on the linear recognition algorithm of hypercubes, by Hammack et.~al.~\cite{Hammack/Imrich/Klavzar:2011}.
For a better understanding of the algorithm we point out the following observations:
\begin{itemize}
\item Folded cubes are arc-transitive, so we can fix an arbitrary edge $uv \in E(G)$ to be the initial diagonal edge and add it to $S_n$, see line~\ref{ln:1}.
Every subsequent diagonal edge is identified as an opposite edge of an already discovered diagonal in some $4$-cycle of $G$.

\item In a folded cube each vertex is incident to exactly one diagonal edge and each non-diagonal edge is incident to exactly two diagonal edges (one from each endpoint). 
It follows that after we identify any $4$-cycle, say $(ua,uv,vb,ab)$, with $uv$ being the initial diagonal edge, edges $ua$ and $vb$ cannot appear in any other $4$-cycle containing some different diagonal edges. 
This is the reason why they may safely be discarded, after both corresponding diagonal edges ($uv$ and $ab$) have been identified.
\item
At the start of any cycle of the \texttt{while}-loop, the degree of vertices of an arbitrary diagonal edge is the same.
This is easily observed as any deletion involves two non-diagonal edges in a fixed $4$-cycle, say $(ua,uv,vb,ab)$,  therefore the degree of both endpoints of $uv$, as well as of $ab$, decreases exactly by $1$.
This also implies that reassigning those edges in line \ref{ln:6} is non-ambiguous.
\item 
At each step of the algorithm the set of determined diagonal edges is partitioned into sets $S_1, \dots , S_{n}$, where $S_i$ ($i = 1, \dots , n$) consists of diagonal edges $uv$ with the degree $i$ of $u$ and $v$, in the current graph ($G$ without some edges).
\end{itemize}

The whole \texttt{while}-loop is repeated $|E(G)|/n$ times (i.e. once for every diagonal edge) and the \texttt{for}-loop needs $O(i^2)$ time to complete.
Therefore the running time of this part is $O(|E(G)|/n\cdot (n/2)^2) = O(|E(G)| \cdot \log |V|)  = o(N \log N)$, where $N = |V(G) + E(G)|$.

\begin{algorithm}[ht!]
  \caption{Subprocedure $\textsc{Extend}(G,D)$ for folded cubes.   \label{alg:FQsubprocedure}}
  \begin{algorithmic}[1]
    \If{$H = G[E(G) \setminus D]$ not bipartite or $|E(H)| \neq 2^{n-1}n$}
    \State return \textbf{false}. 
    \EndIf
    \State Take an arbitrary edge $uv \in E(H)$. \label{algFQ:lineRecursion}
    \If{edges between $H[W_{uv}]$ and  $H[W_{vu}]$ do not define a perfect matching corresponding to an isomorphism between $H[W_{uv}]$ and  $H[W_{vu}]$}
    \State Return \textbf{false}.
    \ElsIf{$H[W_{uv}] = K_1$}
    \State Append label $0$ to vertex in $W_{uv}$ and $1$ to vertex in $W_{vu}$.
    \Else 
    \State Append label $0$ to vertices in $W_{uv}$ and $1$ to vertices in $W_{vu}$. 
    \State Repeat line~\ref{algFQ:lineRecursion} with $H[W_{uv}]$ and with $H[W_{vu}]$.
    \EndIf
    \State Check if diagonal edges $D$ respect the labeling of vertices in $G$.
  \end{algorithmic}
\end{algorithm}
\paragraph{Determining if \boldmath$G$ is isomorphic to a folded cube.}
To determine if a given graph is isomorphic to the hypercube, we need to check that $H = G[E(G) \setminus S_1]$ is an $(n-1)$-dimensional hypercube, label it
and at the end inspect vertices joined by diagonal edges $S_1$,
if they are labeled correctly.
This procedure is described in \cref{alg:FQsubprocedure} and uses
the same approach as the linear recognition algorithm for hypercubes introduced by Hammack et.~al \cite{Hammack/Imrich/Klavzar:2011}. For the sake of completeness we include the algorithm, while omitting its analysis.

At the end we see that the whole recognition algorithm runs in $o(N \log N)$ time, which proves the second part of \cref{cor:lin}.
\section{Conclusion}
It is easy to observe that the gap between the linear lower bound and our upper bound for recognizing folded cubes can be expressed in a sub-logarithmic multiplicative factor. It is hence natural to ask whether the recognition can be done faster, i.e. in linear time.
From the structural point of view, this paper focuses on determining various cyclic parameters for some well-studied families. 
Namely, for the folded cubes we describe the value of $\lambda$ in the context of their $[1,\lambda,4]$, $[1,\lambda,6]$ and $[2,\lambda,6]$-cycle regularity.

For the families of $I$-graphs and double generalized Petersen graphs, we
also settled $[1,\lambda,8]$-cycle regularity.
We also  make an observation of the isomorphism property of certain double generalized Petersen graphs (stated in \cref{thm:isomorphDP}), which is a step towards characterizing them.
We note that full characterization of isomorphisms for the family of double generalized Petersen graphs then remains open.

It seems that partitioning edges with respect to  the number of distinct cycles of certain type those edges belong to led us to the construction of two linear-time recognition algoritms:
$I$-graphs and double generalized Petersen graphs.
To the best of our knowledge, in addition to this work,  such a procedure was so far only used in \cite{Krnc/Wilson:2017} for the family of generalized Petersen graphs.
We believe that a similar approach should give interesting results for other parametric graph families of bounded degree, such as 
Johnson graphs, rose window graphs, Tabačjn graphs, $Y$-graphs, or $H$-graphs.




\paragraph*{Acknowledgements}
The authors would like to thank 
Assist.~Prof.~Nino Bašić, 
Prof.~Štefko Miklavič, 
Prof.~Martin Milanič, and
Prof.~Tomaž Pisanski, for help and fruitful discussions, which led to the improvement of this work.
The second author acknowledges partial support of the Slovenian Research Agency (research programs P1-0383, P1-0297 and research projects J1-1692, J1-9187) and the European Commission for funding the InnoRenew CoE project (Grant Agreement \# 739574) under the Horizon 2020 Widespread-Teaming program and the Republic of Slovenia.

\bibliography{paper}

\end{document}